\documentclass[twoside,11pt]{article}

\usepackage{jmlr2e}

\usepackage{epsfig}
\usepackage{graphicx}
\usepackage{float}

\usepackage{hyperref}
\usepackage{url}
\usepackage{stackengine}
\usepackage{booktabs}

\usepackage{fontenc}
\usepackage{nicefrac}
\usepackage{subcaption}
\usepackage{algpseudocode}

\usepackage{amsmath, mathtools, bm}
\usepackage{dsfont}
\usepackage{mathabx}

\usepackage[usenames, dvipsnames]{xcolor}

\usepackage[T1]{fontenc}
 \usepackage[norelsize, ruled, vlined]{algorithm2e}
 \usepackage{algcompatible}
\usepackage{algpseudocode}

\usepackage{colortbl}

\usepackage{enumitem}
\setlist{noitemsep, topsep=0cm}
\usepackage{tikz}
\usepackage{wrapfig}
\usepackage[utf8]{inputenc}
\usepackage{lscape}

\usepackage{multirow, multicol}
\usepackage{nicematrix}

\allowdisplaybreaks
\numberwithin{theorem}{section}
\newtheorem{assumption}[theorem]{Assumption}

\newcommand{\mmode}[1]{\( #1 \)}

\newcommand{\abs}[1]{\left\lvert #1 \right\rvert}

\newcommand{\pmatr}[1]{\begin{pmatrix}#1\end{pmatrix}}

\newcommand{\R}[1]{\mathbb{R}^{#1}}

\newcommand{\identity}[1]{\mathbb{I}_{#1}}

\newcommand{\hilbert}{\mathbb{H}}

\newcommand{\nbhood}[2]{B(#1, #2)}
\newcommand{\closednbhood}[2]{B[#1, #2]}

\renewcommand{\P}[1]{\mathbb{P} \left( #1 \right)}

\newcommand{\indicator}{\mathds{1}}
\newcommand{\summ}[2]{\sum\limits_{#1}^{#2}}

\newcommand{\pnorm}[2]{\left\lVert#1\right\rVert_{#2}}
\newcommand{\norm}[1]{\left\lVert#1\right\rVert}

\newcommand{\inprod}[2]{\left\langle#1 \; , \; #2\right\rangle}

\newcommand{\gradient}{\nabla}
\newcommand{\projection}{\Pi}

\newcommand{\define}{\coloneqq}

\newcommand{\opt}{^\ast}
\newcommand{\transp}{^\top}

\renewcommand{\geq}{\geqslant}

\renewcommand{\leq}{\leqslant}

\newcommand{\measurement}{x}

\newcommand{\measdim}{n}

\newcommand{\rep}{f}

\newcommand{\linmap}{\phi}
\newcommand{\linadj}{\phi^a}

\newcommand{\eps}{\epsilon}
\newcommand{\epsbar}{\bar{\eps}}

\newcommand{\accelparam}{\rho}

\newcommand{\stepsize}{\gamma}
\newcommand{\stepsizeh}{\stepsize (\minvar)}

\newcommand{\cost}{c}
\newcommand{\costball}{B_{\cost}}

\newcommand{\minvar}{h}
\newcommand{\minvarf}{f}

\newcommand{\maxvar}{\lambda}
\newcommand{\maxvarh}[1]{\maxvar (\minvar^{#1})}

\newcommand{\noise}{\xi}

\newcommand{\direction}{d}

\newcommand{\oraclevar}{g}
\newcommand{\oraclemin}{g(\minvar)}

\newcommand{\hcone}{\mathcal{K(\eps)}}
\newcommand{\smallhcone}{\mathcal{K(\epsbar)}}

\newcommand{\scaling}{r}

\newcommand{\seth}{\mathcal{H}}
\newcommand{\sethsmall}{\seth \duet}

\newcommand{\setlambda}{\Lambda}

\newcommand{\setlambdasmall}{\setlambda \duetlambda}

\newcommand{\phih}[1]{\linmap (\minvar_{#1})}

\newcommand{\etagrad}[1]{\nabla \eta ( \minvar{#1} )}

\newcommand{\etah}[1]{\eta \left( \minvar_{#1} \right)}
\newcommand{\etabar}{\bar{\eta}}
\newcommand{\etahat}{\widehat{\eta}}

\newcommand{\boundlambda}{B}

\newcommand{\duet}{\left( \epsbar, \etahat \right)}
\newcommand{\duetlambda}{(\etabar , \boundlambda)}

\newcommand{\Mh}{M(\minvar)}
\newcommand{\rh}{r(\minvar)}

\newcommand{\eh}[1]{e \left( \minvar_{#1} \right)}
\newcommand{\ed}[1]{e \left( \direction_{#1} \right)}

\newcommand{\lagrangian}{L \left( \maxvar , \minvar \right)}
\newcommand{\lagrangianh}[1]{L \left( \maxvar (\minvar^{#1}) , \minvar \right)}

\newcommand{\eigval}{\sigma}
\newcommand{\mineig}{\bar{\eigval}}
\newcommand{\maxeig}{\widehat{\eigval}}

\newcommand{\sqfunc}{l(\maxvar)}

\newcommand{\hessian}{H (\maxvar)}

\newcommand{\hessianeta}{\left( \Delta^2 \etah{} \right) }
\newcommand{\gradlambdah}{\left( \frac{\partial \maxvarh{}}{\partial \minvar} \right) }

\newcommand{\ipxphih}[1]{\inprod{\measurement}{\linmap \left( \minvar_{#1} \right) }}

\newcommand{\normphih}[1]{\norm{\linmap \left( \minvar_{#1} \right)}}
\newcommand{\normx}{\norm{\measurement}}
\newcommand{\normlambda}{\norm{\maxvar}}
\newcommand{\normlambdah}{\norm{ \maxvarh{} }}

\newcommand{\xminuseps}{\big( \normx^2 - \eps^2 \big)}
\newcommand{\xminusetaphih}[1]{\measurement - \etah{#1} \phih{#1} }

\newcommand{\xminusepslambda}{\measurement - \frac{\eps}{\normlambda} \maxvar}

\newcommand{\epsminuse}[1]{ \eps^2 - \eh{#1} }
\newcommand{\epsminusebar}{ \eps^2 - \epsbar^2}

\newcommand{\sqterm}{\sqrt{\inprod{\maxvar}{\measurement} - \eps \norm{\maxvar} }}
\newcommand{\sqtermh}[1]{\sqrt{\inprod{\maxvar(\minvar^{#1})}{\measurement} - \eps \norm{\maxvar(\minvar^{#1})} }}

\newcommand{\smooth}{\beta}
\newcommand{\strongconv}{\alpha}

\DeclareSymbolFont{symbolsC}{U}{pxsyc}{m}{n}

\DeclareMathOperator{\interior}{int} 

\DeclareMathOperator{\trace}{tr}
\DeclareMathOperator{\image}{image}

\DeclareMathOperator{\Span}{span}

\DeclareMathOperator{\sgn}{sgn}

\DeclareMathOperator*{\sbjto}{subject \; to}

\DeclareMathOperator*{\argmin}{argmin}
\DeclareMathOperator*{\argmax}{argmax}

\firstpageno{1}

\usepackage{lastpage}
\jmlrheading{26}{2025}{1-\pageref{LastPage}}{12/22; Revised
6/25}{6/25}{22-1380}{Mohammed Rayyan Sheriff, Floor Fenne Redel, and Peyman Mohajerin Esfahani}
\ShortHeadings{Fast algorithms for constrained Linear Inverse Problems}{M.R. Sheriff, F.F. Redel, and P. Mohajerin Esfahani}

\begin{document}

\title{Fast Algorithm for Constrained Linear Inverse Problems}

\author{\name Mohammed Rayyan Sheriff \email m.r.sheriff@tudelft.nl \\
       \name Floor Fenne Redel \email f.f.redel@student.tudelft.nl \\
       \name Peyman {Mohajerin Esfahani} \email p.mohajerinesfahani@tudelft.nl \\
       \addr Delft Center for Systems \& Control\\
Delft University of Technology\\
Delft, The Netherlands \\%
}

\editor{Pradeep Ravikumar}

\maketitle

\begin{abstract}
We consider the constrained Linear Inverse Problem (LIP), where a certain atomic norm (like the $\ell_1 $ norm) is minimized subject to a quadratic constraint. Typically, such cost functions are non-differentiable which makes them not amenable to the fast optimization methods existing in practice. We propose two equivalent reformulations of the constrained LIP with improved convex regularity: (i) a smooth convex minimization problem, and (ii) a strongly convex min-max problem. These problems could be solved by applying existing acceleration-based convex optimization methods which provide better $ O \left( \nicefrac{1}{k^2} \right)$ theoretical convergence guarantee, improving upon the current best rate of $O \left( \nicefrac{1}{k} \right)$. We also provide a novel algorithm named the Fast Linear Inverse Problem Solver (FLIPS), which is tailored to maximally exploit the structure of the reformulations. We demonstrate the performance of FLIPS on the classical problems of Binary Selection, Compressed Sensing, and Image Denoising. We also provide open source \texttt{MATLAB} and \texttt{PYTHON} package for these three examples, which can be easily adapted to other LIPs.
\end{abstract}

\begin{keywords}
linear inverse problems, min-max problems, sparse coding, image processing.
\end{keywords}

\section{Introduction}\label{chap:intro}
Linear Inverse Problems simply refer to the task of recovering a signal from its noisy linear measurements. LIPs arise in many applications, such as image processing \citep{elad2006image,Yaghoobi2009CompressibleApproximations, Aharon2006K-SVD:Representation, Olshausen1997Sparse}, compressed sensing \citep{Donoho2006CompressedSensing, candes2006near,candes2008introduction,Gleichman2011BlindSensing}, recommender systems \citep{recht2010guaranteed}, and control system engineering \citep{nagahara2015maximum}. Formally, given a signal \mmode{\minvarf \in \hilbert}, and its noisy linear measurements \mmode{\R{\measdim} \ni \measurement = \linmap (\minvarf) + \xi}, where, \mmode{\linmap : \hilbert \longrightarrow \R{\measdim}} is a linear measurement operator and \mmode{\xi \in \R{\measdim}} is the measurement noise. The objective is to recover the signal \mmode{\minvarf} given its noisy measurements \mmode{\measurement}, and the measurement operator \mmode{\linmap}. Of specific interest is the case when the number of measurements available are fewer than the ambient dimension of the signal, i.e., \mmode{\measdim < \dim (\hilbert)}. In which case, we refer to the corresponding LIP as being `ill-posed' since there could be potentially infinitely many solutions satisfying the measurements even for the noiseless case. In principle, one cannot recover a generic signal \mmode{\minvarf} from its measurements if the problem is ill-posed. However, the natural signals we encounter in practice often have much more structure to be exploited. For instance, natural images and audio signals tend to have a sparse representation in a well-chosen basis, matrix valued signals encountered in practice have low rank, etc. Enforcing such a low-dimensional structure into the recovery problem often suffices to overcome its ill-posedness. This is done by solving an optimization problem with an objective function that promotes the expected low-dimensional structure in the solution like sparsity, low-rank, etc. It is now well established that under very mild conditions, such optimization problems and even their convex relaxations often recover the true signal almost accurately \citep{Donoho2006ForSolution, Donoho2006CompressedSensing, candes2008introduction}.

\subsection{Problem setup}
Given \mmode{\measurement \in \R{\measdim}}, the linear operator \mmode{ \linmap : \hilbert \longrightarrow \R{\measdim} }, and \mmode{\eps > 0}, the object of interest in this article is the following optimization problem
\begin{equation}
\label{eq:lip-main}
\begin{cases}
	  \begin{aligned}
		& \argmin_{\minvarf \; \in \; \hilbert}  && \cost (\minvarf) \\
		& \sbjto && \norm{ \measurement - \linmap (\minvarf) } \leq \eps ,
	\end{aligned}
	\end{cases}
\end{equation}  
where \mmode{\hilbert} is some finite-dimensional Hilbert space with the associated inner-product \mmode{\inprod{\cdot}{\cdot}}. The constraint \mmode{ \norm{ \measurement - \linmap (\minvarf) } \leq \eps } is measured using the norm derived from an inner product on \mmode{\R{\measdim}} (it is independent from the inner product \mmode{\inprod{\cdot}{\cdot}} on the Hilbert space \mmode{\hilbert}).

The objective function is the mapping \mmode{\cost : \hilbert \longrightarrow \R{}} which is known to promote the low-dimensional characteristics desired in the solution. For example, if the task is to recover a sparse signal, we chose \mmode{\cost(\cdot) = \pnorm{\cdot}{1}}; if \mmode{\hilbert} is the space of matrices of a fixed order, then \mmode{\cost(\cdot) = \pnorm{\cdot}{*}} (the Nuclear-norm) if low-rank matrices are desired \citep{Candes2009ExactOptimization}. In general, the objective function is assumed to be \emph{norm like}.

\noindent \textbf{Main challenges of \eqref{eq:lip-main} and existing state of the art methods to solve it}. One of the main challenges to tackle while solving \eqref{eq:lip-main} is that, the most common choices of the cost function \mmode{\cost (\cdot)} like the \mmode{\ell_1} norm, are not differentiable everywhere. In particular, the issue of non-differentiability gets amplified since it is prevalent precisely at the suspected optimal solution (sparse vectors). Thus, canonical gradient-based schemes do not apply to~\eqref{eq:lip-main} with such cost functions. A common workaround is to use the notion of sub-gradients instead, along with a diminishing step-size. However, the Sub-Gradient Descent method (SGD) for generic convex problems converges only at a rate of \( O(\nicefrac{1}{\sqrt{k}})\) \citep{Boyd2003SubgradientMethods}. For high-dimensional signals like images, this can be tiringly slow since the computational complexity scales exponentially with the signal dimensions.

\subsection{Existing methods} 
The current best algorithms overcome the bottleneck of non differentiability in \eqref{eq:lip-main} by instead working with the more flexible notion of \emph{proximal gradients}, and applying them to a suitable reformulation of the LIP \eqref{eq:lip-main}. We primarily focus on two state-of-the-art methods in this article: the Chambolle-Pock algorithm (CP) \citep{Chambolle2010AImaging} and the C-SALSA \citep{Afonso2009AnProblems} algorithm, that solve the LIP \eqref{eq:lip-main}.

\noindent (i)\emph{ The Chambolle-Pock algorithm}: Using the \emph{convex indicator function} \mmode{\indicator_{\closednbhood{\measurement}{\eps}} (\cdot)} of \mmode{\closednbhood{\measurement}{\eps}},\footnote{The convex indicator function \mmode{\indicator_S (z)} of a given convex set \mmode{S} is \mmode{\indicator_{S} (z) = 0 } if \mmode{z \in S}; \mmode{= +\infty} if \mmode{z \notin S}.} the constraints in LIP \eqref{eq:lip-main} are incorporated into the objective function to get
\begin{equation}\label{eq:main_with_indicator}
    \min _{\minvarf \in  \hilbert} \quad \cost (\minvarf) + \indicator_{B[\measurement,\epsilon]}( \linmap( \minvarf)).
\end{equation}
The indicator function \mmode{\indicator_S (\cdot)} of any closed convex set \mmode{S} is proper and lower-semicontinuous. Consequently, its convex-conjugate satisfies \((\indicator^*_{S})^*=\indicator_{S}\). Since \mmode{ \indicator\opt_{B[\measurement,\epsilon]}( u) =  \langle \measurement, u \rangle + \epsilon \|u\| }, for \mmode{\closednbhood{\measurement}{\eps}} in particular, we have
\begin{equation}\label{eq:CP}
    \indicator_{B[\measurement,\epsilon]}( \linmap( \minvarf)) \; = \; \max _{u \in  \R{\measdim}} \quad \langle \linmap(\minvarf),u \rangle - \big( \langle \measurement, u \rangle + \epsilon \|u\| \big).
\end{equation}
Incorporating \eqref{eq:CP} in \eqref{eq:main_with_indicator}, the LIP reduces to the following equivalent min-max formulation
\begin{equation}\label{eq:CP_minmax}
    \min_{f \in \hilbert} \; \max _{u \in  \R{\measdim}} \quad \cost(\minvarf) + \langle \linmap(\minvarf),u \rangle - (\langle \measurement, u \rangle + \epsilon \|u\|).
\end{equation}
The min-max problem \eqref{eq:CP_minmax} falls under a special subclass of convex-concave min-max problems with bi-linear coupling between the minimizing (\minvarf) and maximizing (u) variables. A primal-dual algorithm was proposed in \citep{Chambolle2010AImaging, Chambolle2016OnAlgorithm} to solve such problems under the condition that the mappings \(f \longmapsto \cost(\minvarf)\) and \(u \longmapsto \inprod{x}{u} + \epsilon \|u\|\) are proximal friendly. It turns out that in many relevant problems particularly where \(\cost(\minvarf)=\|\minvarf\|_1\) the proximal operator of \mmode{\minvarf \longmapsto \cost(\minvarf)} is indeed easily computable \citep{Beck2009AProblems}. Moreover, the proximal operator for the mapping \mmode{u \longmapsto \inprod{x}{u} + \epsilon \|u\|} corresponds to block soft-thresholding, and is also easy to implement. Under such a setting the CP algorithm has an ergodic convergence rate of \( O(\nicefrac{1}{k})\) for the duality gap of \eqref{eq:CP_minmax}. This is already an improvement over the \( O(\nicefrac{1}{\sqrt{k}})\) rate in canonical sub-gradient ``descent'' algorithms, and is currently the best convergence guarantee that exists for Problem \eqref{eq:lip-main}.

\noindent \emph{(ii) The C-SALSA algorithm}: The Constrained Split Augmented Lagrangian Shrinkage Algorithm (C-SALSA) \citep{Afonso2009AnProblems} is an algorithm in which the Alternating Direction Method of Multipliers (ADMM) is applied to problem \eqref{eq:main_with_indicator}. For this problem, ADMM solves the LIP based on variable splitting using an Augmented Lagrangian Method (ALM). In a nutshell, the algorithm iterates between optimizing the variable \mmode{\minvarf} and the Lagrange multipliers until they converge. Even though the convergence rates for C-SALSA are not better than that of the CP algorithm, it is empirically found to be fast. Of particular interest is the case when \(\phi\) satisfies \(\phi\transp \phi = \identity{\measdim} \), in which case, further simplifications in the algorithm can be done that improve its speed  for all practical purposes.

One of the objectives of this work is to provide an algorithm that is demonstrably faster than the existing methods, particularly for large scale problems. Given that solving an LIP is such a commonly arising problem in signal processing and machine learning, a fast and easy to implement is always desirable. This article precisely caters to this challenge. In \citep{Sheriff2020NovelProblems}, the LIP \eqref{eq:lip-main} was equivalently reformulated as a convex-concave min-max problem:
\begin{equation}
\label{eq:JMLR-lip-min-max}
\min_{\minvar \in \costball} \ \sup_{\maxvar \in \setlambda} \quad 2 \sqterm \; - \; \inprod{\maxvar}{\phih{}} ,
\end{equation}
where \mmode{\costball = \{ \minvar \in \hilbert : \cost (\minvar) \leq 1 \}} and \mmode{\setlambda \define \{ \maxvar \in \R{\measdim} : \inprod{\maxvar}{\measurement} - \eps \normlambda > 0 \}}. A solution to the LIP \eqref{eq:lip-main} can be computed from a saddle point \mmode{(\minvar\opt , \maxvar\opt)} of the min-max problem \eqref{eq:JMLR-lip-min-max}. Even though primal-dual schemes like Gradient Descent-Ascent with appropriate step-size can be used to compute a saddle-point of the min-max problem \eqref{eq:JMLR-lip-min-max}; such generic methods fail to exploit the specific structure of the min-max form \eqref{eq:JMLR-lip-min-max}. It turns out that equivalent problems with better convex regularity (like smoothness) can be derived from \eqref{eq:JMLR-lip-min-max} by exploiting the specific nature of this min-max problem. 

\subsection{Contribution}
In view of the existing methods mentioned above, we summarize the contributions of this work as follows: 
\begin{enumerate}[leftmargin = *, label = (\alph*)]
\item \textbf{Exact reformulations with improved \mmode{O(\nicefrac{1}{k^2})} convergence rates.} We build upon the min-max reformulation \eqref{eq:JMLR-lip-min-max} and proceed further on two fronts to obtain equivalent reformulations of the LIP \eqref{eq:lip-main} with better convex regularities. These reformulations open the possibility for applying acceleration based methods to solve the LIP \eqref{eq:lip-main} with faster rates of convergence~\mmode{O(\nicefrac{1}{k^2})}, which improves upon the existing best rate of \mmode{O(\nicefrac{1}{k})}.
\begin{enumerate}[leftmargin = *, label = (i)]
    \item \emph{Exact smooth reformulation}: We explicitly solve the maximization over \mmode{\maxvar} in \eqref{eq:JMLR-lip-min-max} (Proposition \ref{proposition:max-problem-solution}) to obtain an equivalent smooth convex minimization problem (Theorem~\ref{theorem:smooth-reformulation}).
    
    \item \emph{Strongly convex min-max reformulation}: We propose a new min-max reformulation \eqref{eq:ACP-min-max-problem} that is slightly different from \eqref{eq:JMLR-lip-min-max} and show that it has strong-concavity in \mmode{\maxvar} (Lemma \ref{lemma:convex-regularity-of-dual-function}). This allows us to apply accelerated version of the Chambolle-Pock algorithm \citep[Algorithm 4]{Chambolle2016OnAlgorithm}; which converges at a rate of \mmode{O(\nicefrac{1}{k^2})} in duality gap for ergodic iterates (Remark \ref{remark:ACP-rate})
\end{enumerate}

\item \textbf{Tailored fast algorithm}: 
We present a novel algorithm (Algorithm \ref{algo:FLIPS}) called the Fast LIP Solver (FLIPS), which exploits the structure of the proposed smooth reformulation \eqref{eq:LIP-smooth-problem} better than the standard acceleration based methods. The novelty of FLIPS is that it combines ideas from canonical gradient descent schemes to find a descent direction; and then from the Frank-Wolfe (FW) algorithm \citep{Jaggi2013RevisitingOptimization} in taking a step in the descent direction. We provide an explicit characterization of the optimal step size which could be computed without significant additional computations. We demonstrate the performance of FLIPS on the standard problems of Binary Selection, Compressed sensing, and Image Denoising. particularly for Image Denoising, we show that FLIPS outperforms the state-of-the-art methods for \eqref{eq:lip-main} like CP and C-SALSA in both number of iterations and CPUtime.

\item \textbf{Open source Matlab package}: Associated with this algorithm, we also present an open-source Matlab package that includes the proposed algorithm (and also the implementation of CP and C-SALSA) \citep{FLIPSpackage}.
\end{enumerate}

This article is organized as follows. In Section \ref{section:equivalent-formulations-of-LIP} we discuss two equivalent reformulations of the LIP; one as a smooth minimization problem in subsection \ref{subsection:equivalent-smooth-problem}, and then as a min-max problem with strong-convexity in subsection \ref{subsection:strongly-convex-min-max-LIP}. Subsequently, in Section \ref{chap:algorithm} the FLIPS algorithm is presented, followed by the numerical simulations in section \ref{chap:numericalsimulations}. All the proofs of results in this article are relegated towards the end of this article in Section \ref{chap:discussion-proofs} for better readability.

\noindent \textbf{Notations.}
Standard notations have been employed for the most part. The interior of a set \(\mathcal{S}\) as int\((\mathcal{S})\). The \(n\times n\) identity matrix is denoted by \(\identity{n}\). For a matrix \(M\) we let \(\trace(M)\) and \(\image (M) \) denote its trace and image respectively. The gradient of a continuously differentiable function \(\eta(\cdot)\) evaluated at a point \mmode{\minvar} is denoted by \mmode{\etagrad{}}. 

Generally  \(\norm{\cdot}\) is the norm associated with the inner product \mmode{\inprod{\cdot}{\cdot}} of the Hilbert space \mmode{\hilbert}, unless specified otherwise explicitly. Given two Hilbert spaces \mmode{\big( \, \hilbert_1 , \inprod{\cdot}{\cdot}_1 \big)} and \mmode{\big( \, \hilbert_2 , \inprod{\cdot}{\cdot}_2 \big)} and a linear map \mmode{T : \hilbert_1 \longrightarrow \hilbert_2}, its adjoint \(T^a\) is another linear map \mmode{T^a : \hilbert_2 \longrightarrow \hilbert_1} such that \mmode{ \inprod{v}{T(u)}_2 \; = \; \inprod{T^a(v)}{u}_1 } for all \mmode{u \in \hilbert_1} and \mmode{v \in \hilbert_2}.

\section{Equivalent reformulations with improved convex regularity}
\label{section:equivalent-formulations-of-LIP}
We consider the LIP \eqref{eq:lip-main} under the setting of following two assumptions that are enforced throughout the article.
\begin{assumption}[Cost function]
\label{assumption:cost-function-main}
The cost function \mmode{\cost : \hilbert \longrightarrow \R{}} is
\begin{enumerate}[leftmargin = *, label = (\alph*)]
    \item \emph{positively homogenuous}: For every \mmode{\scaling \geq 0} and \mmode{\minvarf \in \hilbert}, \mmode{\cost (\scaling \minvarf) \; = \; \scaling \cost (\minvarf)}
    
    \item \emph{inf compact}: the unit sub-level set \mmode{\costball \define \{ \minvarf \in \hilbert : \cost (\minvarf) \leq 1 \}} is compact 
    
    \item \emph{quasi convex}: the unit sub-level set \mmode{\costball} is convex.
\end{enumerate}
\end{assumption}
In addition to the conditions in Assumption \ref{assumption:cost-function-main}, if the cost function is symmetric about the origin, i.e., \mmode{\cost (-\minvarf) = \cost (\minvarf)} for all \mmode{\minvarf \in \hilbert}, then it is indeed a \emph{norm} on \mmode{\hilbert}. Thus, many common choices like the \mmode{\ell_1} and Nuclear norms for practically relevant LIPs are included in the setting of Assumption \ref{assumption:cost-function-main}.

\begin{assumption}[Strict feasibility]
\label{assumption:lip-main}
We shall assume throughout this article that \mmode{\norm{\measurement} > \eps > 0} and that the corresponding LIP \eqref{eq:lip-main} is strictly feasible, i.e., there exists \mmode{\minvarf \in \hilbert} such that \mmode{\norm{\measurement - \linmap (\minvarf)} < \eps}.
\end{assumption}

\subsection{Reformulation as a \emph{smooth} minimization problem}
\label{subsection:equivalent-smooth-problem}
Let the mapping \mmode{e : \R{\measdim} \longrightarrow [0,+\infty)} be defined as
\begin{equation}
\label{eq:e-func-def}
\eh{} \define \quad \min_{\theta \in \R{}} \norm{\measurement - \theta \linmap(\minvar)}^2 \quad = \quad \begin{cases}
\begin{aligned}
& \norm{\measurement}^2 && \text{if } \linmap (\minvar) = 0, \\
& \norm{\measurement}^2 - \frac{\abs{\inprod{\measurement}{\linmap(\minvar)}}^2 }{\norm{\linmap(\minvar)}^2 }  && \text{otherwise}.
\end{aligned}
\end{cases}
\end{equation}
Consider the family of convex cones \mmode{ \left\{ \smallhcone : \epsbar \in (0, \eps] \right\} } defined by 
\begin{equation}
\label{eq:hcone-definition}
\smallhcone \define \left\{ \minvar \in \hilbert : \inprod{\measurement}{\linmap(\minvar)} \; > \; 0, \text{ and } \eh{} \leq \epsbar^2 \right\} , \text{ for every } \epsbar \in (0, \eps] .
\end{equation}
Equivalently, observe that \mmode{ \minvar \in \smallhcone} if and only if \mmode{ \ipxphih{} \geq \normphih{} \sqrt{\big( \normx^2 - \epsbar^2 \big)} }. Therefore, it is immediately evident that \mmode{\smallhcone} is convex for every \mmode{\epsbar \in (0, \eps] }.

\begin{definition}[New objective function]
\label{def:etah-definition}
Consider \mmode{\measurement}, \mmode{\linmap}, and \mmode{\eps > 0} such that Assumption \ref{assumption:lip-main} holds, and let the map \mmode{\eta : \hcone \longrightarrow [0, + \infty)} be defined by
\begin{equation}
\label{eq:etah-definition}
\etah{} \define \frac{\norm{\measurement}^2 - \eps^2}{\ipxphih{} \; + \; \normphih{} \sqrt{ \eps^2 - \eh{}  } } .
\end{equation}
\end{definition}

\begin{figure}[h]
    \centering
    \begin{subfigure}[b]{0.49\textwidth}
         \centering
         \includegraphics[width=\textwidth]{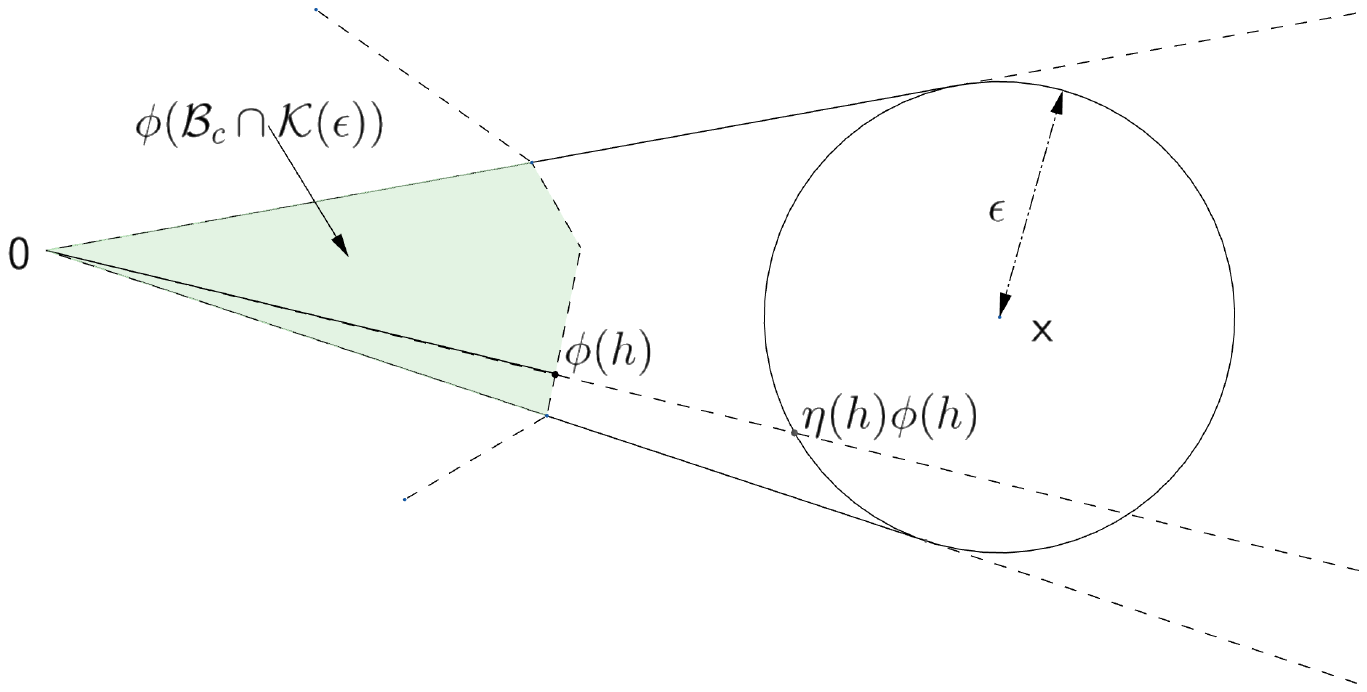}
         \caption{Diagram presenting the relation between $\phi(\costball\cap\mathcal{K(\eps)})$,  $\minvar$, $\etah{}$,  $\linmap(h)$, $\measurement$ and $\epsilon$.}
         \label{fig:2Ddiagram}
    \end{subfigure}
    \hfill
    \begin{subfigure}[b]{0.45\textwidth}
         \centering
         \includegraphics[width=\textwidth]{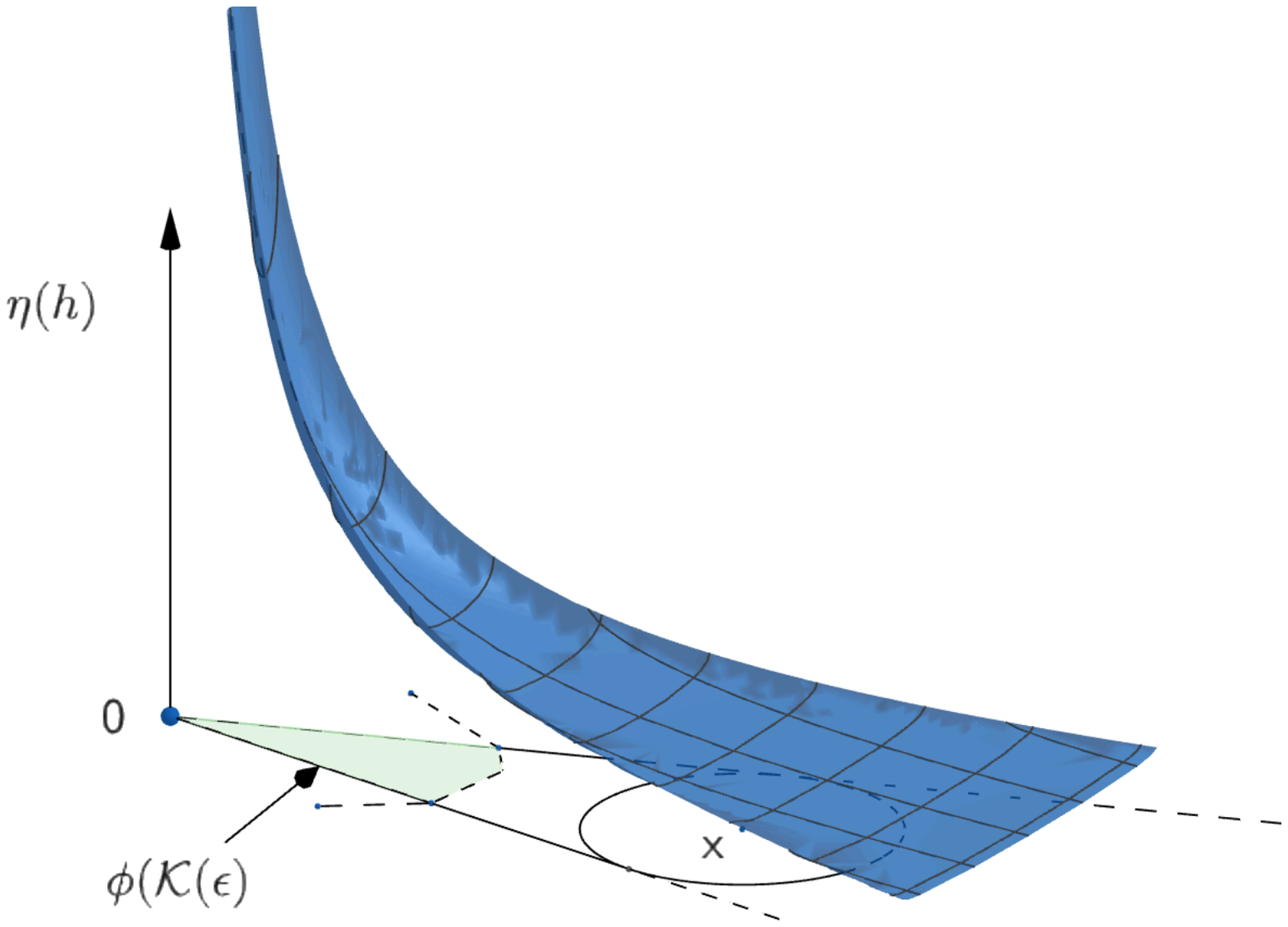}
         \caption{Diagram presenting the $\etah{}$ evaluated over $\phi(\mathcal{K({\eps})})$.}
         \label{fig:3Ddiagram}
    \end{subfigure}
    \caption{Graphical overview of \mmode{\eta, \linmap\big(\hcone\big)}.}
    \label{fig:diagrams}
\end{figure}

\begin{remark}[Physical interpretation of \mmode{\hcone} and \mmode{\etah{}}]\rm{
For any \mmode{\minvar \in \hilbert}, \mmode{\minvar \in \hcone} if and only if the ray \mmode{\{\theta \phih{} : \theta \geq 0\}} (i.e., the line going from origin and passing through \mmode{\phih{}}) intersects with \mmode{\closednbhood{\measurement}{\eps}}. Now, for any \mmode{\minvar \in \hcone},  the value \mmode{\etah{}} is the minimum amount by which the point \mmode{\phih{}} must be scaled so that it intersects with the closed neighborhood \mmode{\closednbhood{\measurement}{\eps}} of \mmode{\measurement} as depicted in Figure~\ref{fig:2Ddiagram}.
}\end{remark}

\begin{proposition}[Derivatives of \mmode{\eta}]
\label{prop:eta-derivatives}
Consider \mmode{\measurement, \linmap , \eps > 0} such that Assumption \ref{assumption:lip-main} holds, and 
\mmode{\etah{}} as defined in \eqref{eq:etah-definition}. then the following assertions hold.
\begin{enumerate}[leftmargin = *, label = \rm{(\roman*)}]
\item \emph{\textbf{Convexity}:} The function \mmode{ \eta : \hcone \longrightarrow [0,+\infty) } is convex.

\item \emph{\textbf{Gradients}:} The function \mmode{ \eta : \hcone \longrightarrow [0,+\infty) } is differentiable at every \mmode{\minvar \in \interior \left( \hcone \right) = \{ \minvar \in \hilbert : \inprod{\measurement}{\linmap(\minvar)} > 0, \text{ and } \eh{} < \eps^2 \} }, and the derivative is given by
    \begin{equation}
    \label{eq:eta-derivatives}
    \etagrad{} = \frac{- \etah{} }{ \normphih{} \sqrt{\eps^2 - \eh{}} } \; \linadj \Big( \xminusetaphih{} \Big)  \quad \text{ for all } \minvar \in \interior \left( \hcone \right) ,
    \end{equation}
    where \mmode{\linadj} is the adjoint operator of \mmode{\linmap}.

\item \emph{\textbf{Hessian}:} the function \mmode{ \eta : \hcone \longrightarrow [0,+\infty) } is twice continuously differentiable at every \mmode{\minvar \in \interior \left( \hcone \right) }. The hessian is the linear operator
\begin{equation}
\label{eq:hessian-eta-definition}
\hilbert \ni v \longmapsto \hessianeta (v) \define \left( \linadj \circ M(\minvar) \circ \linmap \right) (v) \; \in \; \hilbert ,
\end{equation}
where \mmode{ M: \interior (\hcone) \longrightarrow \R{\measdim} \times \R{\measdim} } is a continuous matrix-valued map.\footnote{Please see \eqref{eq:Mh-definition} for the precise definition of \mmode{\Mh}.}
\end{enumerate}
\end{proposition}

\begin{definition}[Convex Regularity]
Let \mmode{\seth} be a convex set  and consider \mmode{\eta : \seth \longrightarrow \R{} }.
\begin{enumerate}
\item {\em (Smoothness)} - The mapping \mmode{\eta : \seth \longrightarrow \R{}} is said to be \mmode{\smooth}-smooth if there exists \mmode{\smooth \geq 0} such that the inequality
\begin{equation}
\label{eq:def-smoothness}
\norm{ \etagrad{} - \etagrad{'} } \ \leq \ \smooth \; \norm{\minvar - \minvar'} \text{ holds for all } \minvar , \minvar' \in \seth.
\end{equation}

\item {\em (Strong-Convexity)} -  The mapping \mmode{\eta : \seth \longrightarrow \R{}} is said to be \mmode{\strongconv}-strongly convex if there exists \mmode{\strongconv > 0} such that the inequality
\begin{equation}
\label{eq:def-strong-convexity}
\eta (\minvar') \; \geq \; \etah{} + \inprod{\etagrad{}}{\minvar' - \minvar} + \frac{\strongconv}{2}  \norm{\minvar' - \minvar}^2 \text{ holds for all } \minvar , \minvar' \in \seth.
\end{equation}
\end{enumerate}
\end{definition}

\noindent\textbf{Challenges for smoothness: }
As such, the mapping \mmode{\eta : \costball \cap \hcone \longrightarrow [0, + \infty)] } is not smooth in the sense of \eqref{eq:def-smoothness}. This is due to two reasons. 
\begin{enumerate}[leftmargin = *]
\item {\em High curvature around origin} - Consider any \mmode{\minvar \in \costball \cap \hcone}, then for \mmode{\theta \in (0, 1]} it is immediate from \eqref{eq:etah-definition} that \mmode{\eta ( \theta \minvar) \; \propto \; \nicefrac{1}{\theta} }, and therefore, the mapping \mmode{(0,1] \ni \theta \longmapsto \eta (\theta \minvar)} is not smooth. Consequently, the mapping \mmode{\eta : \costball \cap \hcone \longrightarrow \R{} } cannot be smooth. As shown in Figure \ref{fig:3Ddiagram} as a simple example, it is easily seen that \mmode{\eta} achieves arbitrarily large values and arbitrarily high curvature as \mmode{\norm{h} \longrightarrow 0}.

\item {\em High curvature at the boundary of \mmode{\hcone}} - It must be observed that \mmode{\eta} is not differentiable on the boundary of the cone \mmode{\hcone}. Moreover, as \mmode{\eh{} \uparrow \eps^2}, i.e., \mmode{\minvar} approaches the boundary of the cone \mmode{\hcone} from its interior, it is apparent from \eqref{eq:eta-derivatives} that the gradients of \mmode{\eta} are unbounded.
\end{enumerate}
It turns out that by avoiding these two scenarios (which will be made more formal shortly), \mmode{\eta} is indeed smooth over the rest of the set.

\begin{proposition}[Convex regularity of \mmode{\eta}]
\label{proposition:eta-convex-regularity}
Consider the LIP in \eqref{eq:lip-main} under the setting of Assumption \ref{assumption:lip-main} with \mmode{\cost\opt} being its optimal value. For every \mmode{\etahat \geq \cost\opt} and \mmode{\epsbar \in (0, \eps) }, consider the convex set
\begin{equation}
\label{eq:Hbar-set-definition}
\sethsmall \define \{ \minvar \in \costball \cap \smallhcone : \etah{ } \leq \etahat \} .
\end{equation}
\begin{enumerate}[leftmargin = *, label = \rm{(\roman*)}]
\item {\textbf{\emph{Smoothness:}}} There exists constant \mmode{\smooth > 0} (see Remark \ref{remark:smoothness-parameter}), such that the mapping \mmode{\eta : \sethsmall \longrightarrow [0, +\infty) } is \mmode{\smooth}-smooth in the sense of \eqref{eq:def-smoothness}.

\item \textbf{\emph{Strong convexity:}} In addition, if the linear operator \mmode{\linmap} is invertible, then there exists constant \mmode{\strongconv > 0 } (see Remark \ref{remark:strong-convexity-parameter}), such that the mapping \mmode{\eta : \sethsmall \longrightarrow [0, +\infty) } is \mmode{\strongconv}-strongly convex in the sense of \eqref{eq:def-strong-convexity}.
\end{enumerate}
\end{proposition}

\begin{theorem}[Smooth reformulation]
\label{theorem:smooth-reformulation}
Consider the LIP \eqref{eq:lip-main} under the setting of Assumption \ref{assumption:cost-function-main} and \ref{assumption:lip-main}, and let \mmode{\cost\opt} be its optimal value. Then the following assertions hold
\begin{enumerate}[leftmargin = *, label = \rm{(\roman*)}]
    \item Every optimal solution \mmode{\minvarf\opt} of the LIP \eqref{eq:lip-main} satisfies \mmode{e (\minvarf\opt ) < \eps^2}. 
    
    \item \textbf{\emph{The smooth problem:}} Consider any \mmode{\duet} such that \mmode{e(\minvarf\opt) \leq \epsbar^2 < \eps^2} and \mmode{ \cost\opt < \etahat}. Then the optimization problem
\begin{equation}
\label{eq:LIP-smooth-problem}
\min_{\minvar \; \in \; \sethsmall} \ \etah{} .
\end{equation}
is a smooth convex optimization problem equivalent to the LIP \eqref{eq:lip-main}. In other words, the optimal value of \eqref{eq:LIP-smooth-problem} is equal to \mmode{\cost\opt} and \mmode{\minvar\opt} is a solution to \eqref{eq:LIP-smooth-problem} if and only if \mmode{\cost\opt \minvar\opt } is an optimal solution to \eqref{eq:lip-main}.

\end{enumerate}
\end{theorem}

\begin{remark}[Choosing \mmode{\epsbar}]\rm{
Consider the problem of recovering some true signal \mmode{\minvarf\opt} from its noisy linear measurements \mmode{\measurement = \linmap (\minvarf\opt) + \noise}, where \mmode{\noise} is some additive measurement noise. Then, \mmode{\eps} is chosen in \eqref{eq:lip-main} such that the probability \mmode{\P{\norm{\noise} \leq \eps}} is very high. Then for \emph{any} \mmode{\epsbar \in (0, \eps)} we have \mmode{e (\minvarf\opt) < \epsbar^2 } with probability at least \mmode{\P{ \norm{\noise} \leq \eps} \cdot \P{ \abs{ \inprod{\noise}{\frac{\linmap (\minvarf\opt)}{\norm{\linmap (\minvarf\opt)}}} }^2 > \eps^2 - \epsbar^2 } }. Thus, in practise, one could select \mmode{\epsbar} to be just smaller than \mmode{\eps} based on available noise statistics. For empirical evidence, we have demonstrated in Section \ref{subsection:epsbar-numerical-evidence} by plotting the histogram of \mmode{e(\minvarf\opt)} for 28561 instances of LIPs arising in a single image denoising problem solved with a fixed value of \mmode{\eps}. It is evident from Figure \eqref{fig:ef_epsilonplot} that there is a strict separation between the value \mmode{\eps^2}, and the maximum value of \mmode{e(\minvarf\opt)} among different LIPs.
\label{remark:chosing-epsbar}}
\end{remark}

\begin{remark}[Choosing the upper bound \mmode{\etahat}] \rm{
Since \mmode{\etahat} is any upper bound to the optimal value \mmode{\cost\opt} of the LIP \eqref{eq:lip-main}, and equivalently \eqref{eq:LIP-smooth-problem}, a simple candidate is to use \mmode{\etahat = \etah{}} for any feasible \mmode{\minvar}. In particular, for \mmode{\minvar_0 = \big( \nicefrac{1}{\cost (\minvarf')} \big) \minvarf' }, where \mmode{ \minvarf' \define \argmin\limits_{\minvarf } \ \norm{\measurement - \linmap (\minvarf)}^2 } is the solution to the least squares problem, it can be shown that
\begin{equation}
\label{eq:etahat-candidate}
\etahat \define \etah{0} = \cost (\minvarf') \left( 1 - \sqrt{1 - \frac{\normx^2 - \eps^2}{\norm{\measurement '}^2}} \right) \quad \text{where } \measurement' = \linmap (\minvarf') .
\end{equation}
We would like to emphasise that the value of \mmode{\etahat} is required only to conclude smoothness of \eqref{eq:LIP-smooth-problem} and the corresponding smoothness constants. It is \emph{not necessary} for the implementation of the proposed algorithm FLIPS. The inequality \mmode{\etah{} \leq \etahat} in the constraint \mmode{\minvar \in \sethsmall} is ensured for all iterates of FLIPS as it generates a sequence of iterates such that \mmode{\eta} is monotonically decreasing.
}\label{remark:chosing-etahat}
\end{remark}

\begin{remark}[Smoothness parameter]\rm{Let \mmode{\maxeig (\linadj \circ \linmap)} be the maximum eigenvalue of the linear operator \mmode{\big( \linadj \circ \linmap \big)}, then for any \mmode{\etahat \geq \cost\opt}, and \mmode{\epsbar > 0} such that \mmode{e(\minvarf\opt) \leq \epsbar^2 < \eps^2}, consider
\begin{equation}
\label{eq:smoothness-parameter}
\smooth \duet \define \frac{\eps^2 \etahat^3 }{ \big( \epsminusebar \big)^{\nicefrac{3}{2}} } \frac{ \big( \normx + \eps \big) }{\big( \normx - \eps \big)^2 } \; \maxeig (\linadj \circ \linmap) .
\end{equation}
Then the mapping \mmode{\eta : \sethsmall \longrightarrow [0, +\infty)} is \mmode{\smooth \duet}-smooth in the sense of \eqref{eq:def-smoothness}.}\label{remark:smoothness-parameter}\end{remark}

\begin{remark}[Strong convexity parameter]\rm{Let \mmode{\mineig (\linadj \circ \linmap)} be the minimum eigenvalue of the linear operator \mmode{\big( \linadj \circ \linmap \big)}, then for any \mmode{\etabar \in (0, \cost\opt]}, consider the constant
\begin{equation}
\label{eq:strong-convexity-parameter}
\strongconv (\etabar) = \frac{2 {\etabar}^3 }{ \xminuseps } \; \mineig \big( \linadj \circ \linmap \big) .
\end{equation}
Then the mapping \mmode{\eta : \sethsmall \longrightarrow [0, +\infty)} is \mmode{\strongconv (\etabar)}-strongly convex in the sense of \eqref{eq:def-strong-convexity}. Since \mmode{\cost\opt} is not known a priori, we need a valid lower bound \mmode{\etabar} that is easy to compute from the problem parameters \mmode{x, \eps}, and \mmode{\linmap}, which we provide in the following remark.
}\label{remark:strong-convexity-parameter}
\end{remark}

\begin{remark}[Choosing the lower bound \mmode{\etabar}]
\label{remark:lower-bound-etabar}\rm{
Let \mmode{ \cost' (\minvarf) \define \max\limits_{\minvar \in \costball} \ \inprod{\minvarf}{\minvar} } denote the dual function of \mmode{\cost}. Then the quantity 
\begin{equation}
\label{eq:etabar-candidate}
\etabar = \frac{\normx \big( \normx - \eps \big) }{\cost' \big( \linadj (\measurement) \big)},
\end{equation}
is a positive lower bound to the optimal value \mmode{\cost\opt} of the LIP \eqref{eq:lip-main}.\footnote{To see why \mmode{\etabar} is a valid lower bound, consider \eqref{eq:JMLR-lip-min-max}. We observe that \mmode{\scaling \measurement \in \setlambda } for all \mmode{\scaling > 0}, then interchanging the order of min-max in \eqref{eq:JMLR-lip-min-max}, we get
\[
\cost\opt \; = \; \max_{\maxvar \in \setlambda} \, \left\{ 2 \sqterm \, - \, \cost' (\linadj (\maxvar)) \right\} \; \geq \; \max_{\scaling > 0} \, \left\{ 2 \sqrt{\scaling} \sqrt{\normx^2 - \eps \normx} \; - \; \scaling \cost' (\linadj (\measurement)) \right\} \; = \; \etabar.
\]
} } \end{remark}

\begin{remark}[Applying accelerated gradient descent for \eqref{eq:LIP-smooth-problem}]
\label{remark:APGD-algorithm}\rm{
Reformulating the LIP \eqref{eq:lip-main} as the smooth minimization problem \eqref{eq:LIP-smooth-problem} allows us to apply accelerated gradient descent methods \citep{Ye.E.Nesterov1983AO1/k2, Beck2009AProblems}, to improve the theoretical convergence rate from \mmode{O(\nicefrac{1}{k})} to \mmode{O(\nicefrac{1}{k^2})}. We know that by applying the projected accelerated gradient descent algorithm \citep{Beck2009AProblems} for \eqref{eq:LIP-smooth-problem}
\begin{equation}
\label{eq:PAGD}
\begin{cases}
\begin{aligned}
z_k & = \Pi_{\sethsmall} \Big( \minvar_k - (\nicefrac{1}{b}) \nabla \eta (\minvar_k) \Big) \\
t_{k + 1} & = \frac{1+\sqrt{1+4t_k^2}}{2} \\
\minvar_{k + 1} &= z_k + \frac{t_k - 1}{t_{k = 1}} (z_k - z_{k - 1}), 
\end{aligned}
\end{cases}
\end{equation}
the sub-optimality \mmode{\eta(\minvar_k) - \cost\opt} diminishes at a rate of \mmode{O(\nicefrac{1}{k^2})}, which is an improvement over the existing best rate of \mmode{O(\nicefrac{1}{k})} for the CP algorithm. Moreover, if the linear map \mmode{\linmap} is invertible, then since the objective function \mmode{\eta (\cost)} is strongly convex, the iterates in \eqref{eq:PAGD} (or even simple projected gradient descent) converge exponentially (with a slightly different step-size rule).

One of the challenges in implementing the algorithm \eqref{eq:PAGD} is that it might not be possible in general, to compute the orthogonal projections onto the set \mmode{\sethsmall}. However, if somehow the inequality \mmode{e(\minvar_k) \leq \epsbar^2} is ensured always along the iterates, then the problem of projection onto the set \mmode{\sethsmall} simply reduces to projecting onto the set \mmode{\costball}, which is relatively much easier. In practice, this can be achieved by selecting \mmode{\epsbar} such that \mmode{e(\minvarf\opt) < \epsbar^2 < \eps^2}, and a very small step-size (\mmode{\nicefrac{1}{b}}) so that the iterates \mmode{(\minvar_k)_k} do not violate the inequality \mmode{e(\minvar_k) < \epsbar^2}. In our observation, empirically, one can tune the value of \mmode{b} for a given LIP so that the criterion: \mmode{e(\minvar_k) < \epsbar^2} is always satisfied. However, if one has to solve a number of LIPs for various values of \mmode{\measurement} via algorithm \eqref{eq:PAGD}; tuning the value(s) of \mmode{b} could be challenging and tedious. Thus, applying off-the-shelf accelerated methods directly to the smooth reformulation \eqref{eq:LIP-smooth-problem} might not be the best choice in practice. This is one of the reasons we propose a different algorithm (FLIPS) that is tailored to solve the smooth reformulation by maximally exploiting the structure of the problem.
}\end{remark}


\subsection{Equivalent min-max problem with strong-convexity}
\label{subsection:strongly-convex-min-max-LIP}

Consider the LIP \eqref{eq:lip-main} under the setting of Assumptions \ref{assumption:cost-function-main} and \ref{assumption:lip-main}, let \mmode{\cost\opt} be its optimal value and \mmode{\minvarf\opt} be an optimal solution. Consider any \mmode{\etabar, \etahat, \epsbar > 0} such that \mmode{\etabar \leq \cost\opt < \etahat} and \mmode{ e(\minvarf\opt) \leq \epsbar^2 < \eps^2 } (to choose values of \mmode{ \epsbar, \etahat, \etabar}, see Remarks \ref{remark:chosing-epsbar}, \ref{remark:chosing-etahat}, and \ref{remark:lower-bound-etabar}   respectively). Denoting \mmode{\sqfunc \define \sqterm}, we define the constant \mmode{\boundlambda > 0 } and the set \mmode{\setlambdasmall \subset \R{\measdim}} by
\begin{equation}
\label{eq:set-lambda-bar}
\begin{cases}
\begin{aligned}
\boundlambda & \define \frac{\eps \etahat^2}{\xminuseps \sqrt{\epsminusebar} } , \\
\setlambdasmall & \define \big\{ \maxvar \in \R{\measdim} : \; \sqfunc \geq \etabar \text{ and } \normlambda \leq \boundlambda \big\}  .
\end{aligned}
\end{cases}
\end{equation}

\begin{lemma}[Min-max reformulation with strong convexity]
\label{lemma:ACP-min-max}
Consider the LIP \eqref{eq:lip-main} under the setting of Assumptions \ref{assumption:cost-function-main} and \ref{assumption:lip-main}, and let \mmode{\cost\opt} denote its optimal value.
Then the min-max problem 
\begin{equation}
\label{eq:ACP-min-max-problem}
\begin{cases}
\min\limits_{\minvar \in \costball} \; \sup\limits_{\maxvar \in \setlambdasmall }  \quad \lagrangian \; = \; 2 \sqfunc \; - \; \inprod{\maxvar}{\linmap (\minvar)} ,
\end{cases}
\end{equation}
is equivalent to the LIP \eqref{eq:lip-main}. In other words, a pair \mmode{(\minvar\opt , \maxvar\opt) \in \costball \times \setlambdasmall} is a saddle point of \eqref{eq:ACP-min-max-problem} if and only if \mmode{\cost\opt \minvar\opt} is an optimal solution to the LIP \eqref{eq:lip-main}, and
\[
\maxvar\opt = \frac{\cost\opt}{\norm{\linmap (\minvar\opt)} \sqrt{\eps^2 - e(\minvar\opt)} } \big( \measurement - \cost\opt \linmap (\minvar\opt)  \big) .
\]
\end{lemma}

The min-max problem \eqref{eq:ACP-min-max-problem} falls into the interesting class of convex-concave min-max problems with a bi-linear coupling between the minimizing variable \mmode{\minvar} and the maximizing variable \mmode{\maxvar}. Incorporating the constraints \mmode{\minvar \in \costball} and \mmode{\maxvar \in \setlambdasmall} with indicator functions, the min-max problem writes
\[
\begin{cases}
\min\limits_{\minvar \in \hilbert} \; \max\limits_{\maxvar \in \R{\measdim} }  \quad \indicator_{\costball} (\minvar) \; - \; \inprod{\maxvar}{\linmap (\minvar)} \; - \; \big( \indicator_{\setlambdasmall} (\maxvar) - 2 \sqfunc \big) .
\end{cases}
\]
If the constraint sets \mmode{\costball} and \mmode{\setlambdasmall} are projection friendly, the min-max problem \eqref{lemma:ACP-min-max} can be solved by directly applying the Chambolle-Pock (CP) primal-dual algorithm. Without any further assumptions, the duality gap of min-max problem \eqref{eq:ACP-min-max-problem} converges at a rate of \mmode{O(\nicefrac{1}{k})} for the Chambolle-Pock algorithm. This rate of convergence is currently the best, and same as the one when CP is applied directly to the min-max problem \eqref{eq:CP_minmax} discussed in the introduction. However, in addition, if the mapping \mmode{\setlambdasmall \ni \maxvar \longmapsto -2\sqfunc } is smooth and strongly convex, acceleration techniques can be incorporated into the Chambolle-Pock algorithm. One of the contribution of this article towards this direction is to precisely establish that indeed this mapping is smooth and strongly concave under the setting of Assumption \ref{assumption:lip-main}. in which case, the rate of convergence improves from \mmode{O(\nicefrac{1}{k})} to \mmode{O(\nicefrac{1}{k^2})}.

\begin{lemma}[Convex regularity in min-max reformulation]
\label{lemma:convex-regularity-of-dual-function}
Consider \mmode{\measurement \in \R{\measdim}} and \mmode{\eps > 0} such that \mmode{\normx > \eps }. For any \mmode{\etabar, \etahat, \epsbar > 0} such that \mmode{\etabar \leq \cost\opt < \etahat} and  \mmode{ e (\minvarf\opt) \leq \epsbar^2 < \eps^2 }; let \mmode{\boundlambda > 0} be as given in \eqref{eq:set-lambda-bar}, and let \mmode{\strongconv' , \smooth' } be constants given by
\begin{equation}
\label{eq:convex-regularity-constants-of-dual-functions}
\strongconv' \define \frac{\eps}{\xminuseps} \left( \frac{\etabar}{\boundlambda} \right)^3 \quad \text{ and } \quad \smooth' \define \frac{\xminuseps}{2 {\etabar}^3} .
\end{equation}
Then the mapping \mmode{ \setlambdasmall \ni \maxvar \longmapsto -2 \sqfunc } is \mmode{\strongconv'}-strongly convex and \mmode{\smooth'}-smooth.
\end{lemma}

\noindent \textbf{The Accelerated Chambolle-Pock algorithm.} In view of the Lemmas \ref{lemma:ACP-min-max} and \ref{lemma:convex-regularity-of-dual-function}, the min-max problem \eqref{eq:ACP-min-max-problem} admits an accelerated version of the Chambolle-Pock algorithm \citep[Algorithm 4, (30)]{Chambolle2016OnAlgorithm}. Denoting \mmode{\Pi_{\costball}} and \mmode{\Pi_{\setlambdasmall}} to be the projection operators onto the sets \mmode{\costball} and \mmode{\setlambdasmall} respectively, and \mmode{\big( \strongconv' , \smooth' \big) = \big( \strongconv' \duetlambda , \smooth' (\etabar) \big)} for simplicity, the accelerated CP algorithm for \eqref{eq:ACP-min-max-problem} is
\begin{equation}
\label{eq:ACP-algo}
\begin{cases}
\begin{aligned}
\minvar_{k + 1} & = \; \Pi_{\costball} \Big( \minvar_k \; + \; s_k \linadj \big( \maxvar_k + \theta_k (\maxvar_k - \maxvar_{n - 1}) \big) \Big) \\
\maxvar_{k + 1} & = \; \Pi_{\setlambdasmall} \left( \maxvar_k \; + \; \frac{t_k}{l(\maxvar_k)} \Big( \measurement - (\nicefrac{\eps}{\norm{\maxvar_k}) \maxvar_k } + l(\maxvar_k) \linmap (\minvar_{k+1}) \Big) \right),
\end{aligned}
\end{cases}
\end{equation}
where, \mmode{(t_k , s_k, \theta_k)} are positive real numbers satisfying
\begin{equation}
\label{eq:ACP-step-size-conditions}
\theta_{k+1} = \frac{1}{\sqrt{ 1 + \strongconv' t_k }}, \quad t_{k+1} = \frac{t_k}{\sqrt{ 1 + \strongconv' t_k }}, \quad \text{and} \quad s_{k+1} = s_k {\sqrt{ 1 + \strongconv' t_k }}, \quad \text{for } n \geq 0,
\end{equation}
with \mmode{\theta_0 = 1, \ t_0 = \frac{1}{2\smooth'}}, and \mmode{ s_0 = \frac{\smooth'}{\pnorm{\linmap}{o}^2} }.

\begin{remark}[Ergodic \mmode{O(\nicefrac{1}{k^2})} rate of convergence]
\label{remark:ACP-rate}\rm{
Consider the min-max problem \eqref{eq:ACP-min-max-problem}, and let \mmode{(\minvar_k , \maxvar_k)} for \mmode{n = 1,2,\ldots,} be the sequence generated by the Accelerated CP algorithm \eqref{eq:ACP-algo}. Then there exists a constant \mmode{C > 0} such that
\[
\max_{\maxvar \in \setlambdasmall} \ L(\minvar_k , \maxvar)  \; - \; \min_{\minvar \in \costball} \ L(\minvar , \maxvar_k) \ \leq \ \frac{C}{k^2} \quad \text{for all } n \geq 1.
\] 
The remark is an immediate consequence of Lemma \ref{lemma:convex-regularity-of-dual-function} and \citep[Theorem 4 and Lemma 2]{Chambolle2016OnAlgorithm}. }
\end{remark}

\begin{remark}[Projection onto the set \mmode{\setlambdasmall}]
\label{remark:projection-lambda-set}\rm{
In general, computing projections onto the set \mmode{\setlambdasmall} is non-trivial, and in principle, requires a sub problem to be solved at each iteration. However, since the duality gap along the iterates \mmode{(\minvar_k , \maxvar_k)} generated by \eqref{eq:ACP-algo} converges to zero; it follows that \mmode{\cost\opt = \lim\limits_{k \longrightarrow + \infty} l(\maxvar_k) }. By selecting \mmode{\etabar < \cost\opt}, computing projections onto the set \mmode{\setlambdasmall} becomes trivial for all but finitely many iterates in the sequence \mmode{(\maxvar_k)_k}. To see this, observe that the set \mmode{\setlambdasmall} is the intersection of two convex sets \mmode{\{ \maxvar : \normlambda \leq \boundlambda \}} and \mmode{\{ \maxvar : \sqfunc \geq \etabar \}}. Since \mmode{\cost\opt = \lim\limits_{k \longrightarrow + \infty} l(\maxvar_k) }, the inequality \mmode{l (\maxvar_k) \geq \etabar} is readily satisfied for all but finitely many iterates if \mmode{\etabar < \cost\opt}. Consequently, all but finitely many iterates in the sequence \mmode{(\maxvar_k)_k} are contained in the set \mmode{\{ \maxvar : l(\maxvar) \geq \etabar \}}. Therefore, computing projections onto the set \mmode{\setlambdasmall} eventually reduces to projecting onto the set \mmode{\{ \maxvar : \normlambda \leq \boundlambda \}}, which is trivial.
}\end{remark}

\section{The Fast LIP Solver (FLIPS)}\label{chap:algorithm}

Even though the newly proposed smooth reformulation of \eqref{eq:lip-main} is amenable to acceleration based schemes; it could suffer in practice from conservative estimation of smoothness constant. Moreover, since one has to ensure that \mmode{\minvar \in \hcone} for all the iterates, it further constrains the maximum step-size that could be taken, which results in slower convergence in practice. These issues make applying off-the shelf methods to solve the proposed smooth problem not fully appealing. To overcome this, we propose the Fast LIP Solver (FLIPS), presented in Algorithm \ref{algo:FLIPS}. 

\begin{algorithm}
\caption{The Fast LIP Solver}
\label{algo:FLIPS}
\KwIn{Measurement \mmode{\measurement}, linear operator \mmode{\phi}, \mmode{\eps > 0}, and oracle parameters.}
\KwOut{Sparse representation \mmode{\minvarf\opt}} 

\emph{Initialise}: \mmode{h = \big( \nicefrac{1}{\cost(\minvarf')} \big) \minvarf' }, where \mmode{ \minvarf' = \phi \backslash x \define \argmin_{\minvarf} \norm{\measurement - \linmap (\minvarf)} }.

\emph{Check for strict feasibility (Assumption \ref{assumption:lip-main})}

\emph{Iterate till convergence}

\nl \quad \textbf{Compute \mmode{\etah{}} and \mmode{\etagrad{}}}

\nl \quad \textbf{Check for stopping criteria (for small enough \mmode{\delta \sim 0.01})}
\begin{equation}
\label{eq:stopping-criterion}
     \text{Stopping criterion :} \quad \frac{ \inprod{\etagrad{}}{\minvar} }{ \min\limits_{g \in \costball} \inprod{\etagrad{}}{g} } \geq \ 1 + \delta .
\end{equation}

\nl \quad \textbf{Compute the update direction \(\oraclemin\) using any viable oracle}

\nl \quad \textbf{Exact line search:} Compute
\[
\stepsize(\minvar) 
=  \begin{cases}
    \begin{aligned}
    &\argmin_{\stepsize \in [0,1]} &&\eta \big( \minvar + \stepsize (\oraclemin - \minvar )\big)\\
    &\text{subject to}  && \minvar + \stepsize (\oraclemin - \minvar ) \; \in \; \smallhcone
    \end{aligned}
    \end{cases}
\]

\nl {\bf Update} : \mmode{\minvar^+ = \minvar + \stepsizeh{} \big( \oraclemin - \minvar \big)}

{\em Repeat}

\nl \textbf{Output:} the sparse representation \mmode{\minvarf\opt = \eta(\minvar) \minvar}.

\end{algorithm}

\noindent In a nutshell, FLIPS uses two oracle calls in each iteration; one each to compute an update direction \mmode{\oraclemin} and the step-size \mmode{\stepsizeh{}}. It then updates the current iterate by taking the convex combination \mmode{\minvar^+ = \minvar + \stepsize(\minvar) \big( \oraclemin - \minvar \big) } controlled by \mmode{\stepsizeh{}}. This is repeated until a convergence criterion is met.

\begin{remark}[Initialization and checking feasibility]
\label{remark:initialization}\rm{
The algorithm is initialised with a normalized solution of the least squares problem: \mmode{ \argmin_{\minvarf} \norm{\measurement - \linmap (\minvarf)} }, which is written as \mmode{\phi \backslash \measurement} following the convention used in Matlab. If the LIP is ill-posed, the least squares problem will have infinitely many solutions. Even though the algorithm works with any initialization among the solutions to the least squares problem, it is recommended to use the minimum \mmode{\ell_2}-norm solution \mmode{\minvarf' = \linmap^{\dagger} \measurement }. Since \mmode{\linmap (\minvarf')} is closest point (w.r.t. the \mmode{\norm{\cdot}} used in the least squares problem), it is easily verified that the LIP satisfies the strict feasibility condition in Assumption \ref{assumption:lip-main} if and only if the inequality \mmode{\norm{\measurement - \linmap (\minvarf')} < \eps} holds.
}\end{remark}

\begin{remark}[Constraint splitting and successive feasibility]
\label{remark:constraint-splitting}
\rm{
The novelty of FLIPS is that it combines ideas from canonical gradient descent methods to compute the update direction, but takes a step in spirit similar to that of the Frank-Wolfe algorithm. This allows us to 
perform a sort of constraint splitting and handle different constraints in \eqref{eq:LIP-smooth-problem} separately. To elaborate, recall that the feasible set in \eqref{eq:LIP-smooth-problem} is
\[
\sethsmall = \{ \minvar \in \costball \cap \smallhcone : \etah{} \leq \etahat \} .
\]
On the one hand, since \mmode{\costball} is a convex set, for a given \mmode{\minvar \in \costball}, the direction oracle guarantees that \mmode{\minvar^+ \in \costball} by producing \mmode{\oraclemin \in \costball} at every iteration. On the other hand, selection of \mmode{\stepsizeh{}} in the exact line search oracle ensures that \mmode{\eta (\minvar^+) \leq \etah{} \leq \etahat} and \mmode{ \minvar^+ \in \smallhcone}. Thus, \mmode{\minvar^+ \in \sethsmall }, and 
}\end{remark}

\subsection{Descent direction oracle}
For any given \mmode{\minvar \in \sethsmall} (in principle, for any \mmode{\minvar \in \interior \hcone}), the direction oracle simply computes another point \mmode{\oraclemin \in \costball} such that the function \mmode{\eta (\cdot)} could be potentially minimized along the direction \mmode{\oraclemin - \minvar}. To find \mmode{\oraclemin}, a sub-problem is solved at every iteration of the algorithm. Thus, by varying the complexity of these sub-problems, gives rise to different direction oracles. In addition, the output of a descent direction oracle also provides access to quantities that can be used to define the stopping criteria for FLIPS. In the following, we briefly describe some standard descent direction oracles that could be used in FLIPS along with the corresponding stopping criteria for them.

\begin{enumerate}[leftmargin = *, label = \rm{(\alph*)}]
\item \textbf{Linear Oracle (LO):} For any \mmode{\minvar \in \interior \hcone}, the Linear oracle computes the direction \mmode{\oraclemin}   by solving a linear optimization problem over \mmode{\costball} 
\begin{equation}
\label{eq:linear-oracle}
    \text{LO: } \quad
    \oraclemin \in 
                \begin{cases}
                \begin{aligned}
                & \ \argmin_{\oraclevar \in \costball} && \inprod{\gradient \eta (\minvar)}{\oraclevar}.
                \end{aligned}
                \end{cases}
\end{equation}

Finding \mmode{\oraclemin} via a linear oracle makes the corresponding implementation of FLIPS very similar to the Frank-Wolfe (FW) algorithm \citep{Frank1956AnProgramming, Jaggi2013RevisitingOptimization}, but with constraint splitting as discussed in remark \ref{remark:constraint-splitting}. The only difference between the FW-algorithm and FLIPS is that the linear sub-problem \eqref{eq:linear-oracle} is solved over the set \mmode{\costball} in FLIPS and not over the actual feasible set \mmode{\sethsmall} as we would in the FW-algorithm.

\begin{example}[LO for sparse coding]\rm{
              For the sparse coding problem, i.e., LIP \eqref{eq:lip-main} with \(\cost(\rep)=\pnorm{f}{1}\), the corresponding linear oracle \eqref{eq:linear-oracle} is easily described due to the Hölder inequality \citep{Yang1991}. The \mmode{i}-th component \mmode{g_i (\minvar)} of the direction \mmode{\oraclemin} is given by
             \begin{equation}
             \label{eq:linear-oracle2}
             \begin{aligned}
             g_i (\minvar) = 
             \begin{cases}
                -\sgn \big( \nicefrac{\partial \eta}{\partial \minvar_i} \big) , &\text{ if } \abs{ \nicefrac{\partial \eta}{\partial \minvar_i} } = \|\gradient \eta (\minvar)\|_{\infty}, \\ 
                 0, &\text{ if } \abs{ \nicefrac{\partial \eta}{\partial \minvar_i} } \neq\|\gradient \eta (\minvar)\|_{\infty} .
            \end{cases}
            \end{aligned}
            \end{equation}
}\end{example}

\item {\bf Simple Quadratic Oracle (SQO)}:  For any \mmode{\minvar \in \interior \hcone}, the SQO computes the direction \mmode{\oraclemin} by solving a quadratic optimization problem over \mmode{\costball}.
\begin{equation}
\label{eq:SQO}
    \text{SQO: } \quad
    \oraclemin = \projection_{\costball} \Big( \minvar - ( \nicefrac{1}{\smooth} ) \gradient \eta (\minvar) \Big)
\end{equation}
where \mmode{\projection_{\costball}} is the projection operator onto the set \mmode{\costball}. Thus, an SQO is essentially a composition of taking a gradient descent step with a step-size of \mmode{\nicefrac{1}{\smooth}} and projecting back to the set \mmode{\costball}. The parameter \mmode{\smooth} is a hyper parameter of the SQO, which is usually taken as the inverse of smoothness constant in canonical gradient descent schemes. We emphasise here that FLIPS does not jump from \mmode{h} to \mmode{\oraclemin} right away as in projected gradient descent algorithm, but rather takes a convex combination of these points. This allows us to chose \mmode{\smooth} much smaller than the actual smoothness constant.

\item {\bf Accelerated Quadratic Oracle (AQO)}: Adding momentum/acceleration in gradient descent schemes tremendously improves their convergence speeds, both in theory and practice \citep{Ye.E.Nesterov1983AO1/k2}. Taking inspiration from such ideas, we consider the AQO as
\begin{equation}
\label{eq:AQO}
\text{AQO: } \quad
\begin{cases}
\begin{aligned}
\oraclevar(\minvar, \direction) &= \projection_{\costball} \Big( \minvar - (\nicefrac{1}{\smooth}) \big( \gradient \eta (\minvar) + \accelparam \direction \big) \Big) \\
\direction (\minvar, \direction) &= \oraclevar (\minvar, \direction) - \minvar
\end{aligned}
\end{cases}
\end{equation}
The extra iterate \mmode{\direction} carries the information of the momentum/past update, and the parameter \mmode{\accelparam} controls the weight of the momentum, which is an additional hyper parameter of the AQO.
\end{enumerate}

Solving the quadratic problems \eqref{eq:SQO} and \eqref{eq:AQO} require more computational resources than the linear one \eqref{eq:linear-oracle}. Consequently, the complexity of implementing a quadratic oracle is more than that of the linear oracle. Since, solving the quadratic problem \eqref{eq:SQO} reduces to computing orthogonal projections of points onto the set \mmode{\costball}, a practical assumption in implementing a quadratic oracle is that the set \mmode{\costball} is projection friendly, which is indeed the case whenever the corresponding cost function \mmode{\cost(\cdot)} is {\em Prox-friendly}. For some LIPs like the matrix completion problem, solving the corresponding projection problem requires computing the SVD at every iteration, which could be challenging for large scale problems. However, for other relevant cost functions like the \mmode{\ell_1}, \mmode{\ell_{\infty}}-norms, the corresponding projection problem requires projection onto the corresponding spheres which is easy to implement, for e.g., \citep{Duchi2008EfficientShalev-Shwartz}.

\subsection{Step size selection via exact line search}
Once the direction \mmode{\oraclemin} is computed using a viable oracle, FLIPS updates the iterate \mmode{\minvar} by taking a convex combination: \mmode{\minvar + \stepsize (\oraclemin - \minvar)} in spirit similar to that of Frank-Wolfe algorithm. We select the step-size \mmode{\stepsize \in [0,1]} via exact line search, i.e., by solving the problem
\begin{equation}
\label{eq:gammaopt_unconstrained}
\stepsize(\minvar) =
    \argmin_{\stepsize \in [0,1]} \quad \eta \big( \minvar + \stepsize (\oraclemin - \minvar) \big) .
\end{equation}
Luckily, the reformulation of the LIP as \eqref{eq:LIP-smooth-problem} with new objective function \mmode{\eta} allows us to compute the explicit solution  of \eqref{eq:gammaopt_unconstrained} without any noticeable increase in the computational demand. The following proposition characterises the optimal step-size in the exact line search for a generic direction \mmode{\direction} instead on specific \mmode{\oraclemin - \minvar} as in \eqref{eq:gammaopt_unconstrained}.

\begin{proposition}\label{proposition:gamma_opt}
For any \mmode{\minvar \in \hcone}, and \mmode{\direction \in \hilbert}, then consider the optimization problem
\begin{equation}
\label{eq:step-size-selection-proposition}
\stepsize\opt = \argmin_{\stepsize \in [0,1]} \quad \eta (\minvar + \stepsize \direction) .
\end{equation}
Then exactly one of the following assertions hold
\begin{enumerate}[leftmargin = *, label = \rm{(\arabic*)}]
    \item If \mmode{ \inprod{\measurement}{\linmap (\direction)} \leq \etah{} \inprod{\phih{}}{\linmap (\direction)} }, then \mmode{\stepsize\opt = 0}.

    \item If \mmode{\minvar + \direction \in \hcone} and \mmode{  \inprod{\measurement}{\linmap (\direction)} \geq \eta (\minvar + \direction) \inprod{\linmap (\minvar + \direction)}{\linmap (\direction)} }, then \mmode{\stepsize\opt = 1}.

    \item Otherwise, \mmode{\stepsize\opt \in (0, 1)} is the root of the quadratic equation \mmode{a \stepsize^2 + 2b \stepsize + c = 0},\footnote{
\begin{equation}
\label{eq:rsu-definitions}
\begin{aligned}     a&=
                    \norm{\phi(\direction)}^2( \ed{} - \eps^2 )\\
                    b&= 
                    (\norm{\measurement}^2-\eps^2)\inprod{\phih}{\phi(\direction)} - \inprod{\measurement}{\phih{}} \inprod{\measurement}{\phi(\direction)} \\
                    c&=
                    \frac{(\norm{\measurement}^2 -\eps^2) \inprod{ \linmap(\minvar)}{\linmap(\direction)}^2 }{ \norm{\linmap(\direction)}^2 }
                    - \frac{ 2\inprod{ \measurement}{ \linmap(\direction) } \inprod{ \measurement}{ \linmap(\minvar) } \inprod{ \linmap(\minvar)}{\linmap(\direction)} }{ \norm{\linmap(\direction)}^2 }
                    + \frac{ \norm{\linmap(\minvar)}^2\inprod{ \measurement}{ \linmap(\direction)}^2}{\norm{\linmap(\direction)}^2} .
                \end{aligned}
                \end{equation} }
                that also satisfies
    \[
    \frac{ \inprod{\measurement}{ \linmap (\minvar + \stepsize\opt \direction) } }{ \xminuseps } \leq \frac{ \inprod{\linmap (\minvar + \stepsize\opt \direction) }{\linmap \direction} }{ \inprod{\measurement}{\linmap (\direction)} } .
    \]
\end{enumerate}
\end{proposition}

\begin{remark}[Qualitative convergence of FLIPS]
\label{remark:qualitative-convergence-of-FLIPS}
We would like to provide qualitative convergence of FLIPS with a `Simple Quadratic Oracle (SQO)'. Consider the sequence of iterates \mmode{(\minvar_k)_k}, generated form FLIPS with an SQO, then the following arguments apply

\begin{enumerate}[label = \rm{(\roman*)}]
    \item the sequence of the values of the cost function \mmode{(\eta(\minvar_k))_k} is non-increasing for FLIPS with any oracle, and in particular, with an SQO

    \item the mapping \mmode{h_t \longmapsto \big( g(h_k) , \gamma (h_k) \big)}, is continuous for all \mmode{t}

    \item the equality \mmode{h = g(h)} holds if and only if \mmode{h = h\opt}
\end{enumerate}
Claim (iii) is non-trivial but follows directly from the first order optimality conditions, we omit this proof for the sake of brevity. Putting the three arguments (i) - (iii) together, and considering the Lyapunov function \mmode{V(h) \define \eta(h) - \eta(h^*)}, we conclude from the Lyapunov theorem that the iterations \mmode{ (h_k)_k} converge to \mmode{h^*}.
\end{remark}

\begin{remark}[Techniques for speeding up implementation of FLIPS]
\label{remark:speeding-up-FLIPS} \rm{
We discuss some techniques and tricks to improve the practical implementation of FLIPS
\begin{enumerate}[leftmargin = *]
\item The inner-product terms \mmode{\ipxphih{} , \inprod{\measurement}{\linmap (\direction)}} must be computed as \mmode{ \inprod{\linadj (\measurement)}{\minvar} } and \mmode{ \inprod{\linadj (\measurement)}{\direction} }. This way, we do not compute (and store) the vectors \mmode{ \phih{} } and \mmode{\linmap (\direction)} at each iteration, but rather compute the vector \mmode{\linadj (\measurement)} once.

\item To compute \mmode{\etagrad{}}, we need to compute \mmode{\linadj (\maxvarh{})}. This can be alternatively done by keeping a running iterate of \mmode{(\linadj o \linmap) h}, which is updated at each iteration as \mmode{ (\linadj o \linmap) h_{t+1} = (\linadj o \linmap) h_t + \stepsize_t \big( \linadj o \linmap (g_t) - (\linadj o \linmap) h_t \big) }.

\end{enumerate}
}\end{remark}

\section{Numerical Results}\label{chap:numericalsimulations}
We test FLIPS on a few well-known LIPs namely {\em Compressed Sensing, Binary-Selection problem}, and finally we test FLIPS on the classical {\em Image Denoising} problem. Moreover, since the image processing problems are the most common LIPs and many good solvers already exist, we compare FLIPS with existing state of the art methods for constrained LIPs arising in image processing tasks.

All experiments in Sections \ref{subsection-compressed-sensing}, \ref{subsection:binary-selection}, and \ref{subsection:image-denoising} were run on a laptop with Apple M1 processor with 8GB RAM using MATLAB 2021b. While, the experiments in Section \ref{subsection:FLIPS-vs-FISTA} comparing FLIPS with FISTA were run on a laptop with Apple M1 max processor with 32GB RAM using Python. The open-source code can be found on the author's Github page \citep{FLIPSpackage}.

\subsection{Results on the Binary-selection problem}
\label{subsection:binary-selection}
In this example, we aim to reconstruct a vector \mmode{\minvarf_{tr} \in \R{K}} whose entries are \mmode{\pm 1} from its linear measurements. Without loss of generality, we consider \mmode{f_{tr} (i) = +1} for \mmode{i = 1,2,\ldots,0.5K}, and \mmode{f_{tr} (i) = -1} for \mmode{i > 0.5K}, and then collect \mmode{m \sim 0.55K} linear measurements \mmode{x \in \R{m}} (\mmode{55\%} of the information). For each \mmode{i = 1,2,\ldots,m}, the measurement \mmode{x_i} is obtained as \mmode{x_i = \linmap_i\transp f_{tr} + w_i}, where \mmode{\linmap_i \in \R{K}} is generated randomly by sampling each entry of \mmode{\linmap_i} uniformly over the interval \mmode{[-0.5, 0.5]}. The measurement noise \mmode{w_i \sim \mathcal{N} (0, \sigma)}  with \mmode{\sigma = 0.0125}, is generated by sampling randomly and independently from everything else. 

Following, \cite{chandrasekaran2012convex}, the problem of recovering \mmode{f_{tr}} from \mmode{\measurement, \linmap} is formulated as the LIP
\begin{equation}
\label{eq:binary-selection}
\begin{cases}
\begin{aligned}
& \min_{\minvarf} && \pnorm{f}{\infty} \\
& \sbjto && \norm{x - \linmap \minvarf} \leq \eps .
\end{aligned}
\end{cases}
\end{equation}
We chose the value of \mmode{\eps = 10\sigma\sqrt{m}}. Since we know that the entries can be either \mmode{+1} or \mmode{-1}, it allows us to select a slightly larger value of \mmode{\eps} than necessary, which has an indirect benefit in improving the regularity of the problem and consequently faster convergence.

We consider three different instances of \eqref{eq:binary-selection} for \mmode{K = 500, 1000}, and \mmode{5000}. We apply FLIPS for each of the problems using an AQO oracle with parameters \mmode{\smooth } and \mmode{\rho} tuned for faster convergence. The performance of FLIPS for \mmode{K = 3000} iterations is shown in Figure \ref{fig:binary-selection} following the theme of Figure \ref{fig:compressed-sensing}. The first row plots the true signal \mmode{f_{tr}} and the recovered signal \mmode{f^* = \eta(\minvar_T) \minvar_T}. Following these to the bottom we have the plots for the {\em sub-optimality}, {\em Distance to true solution}, and the sequence of {\em step-sizes} as a function of iterations of FLIPS.

\begin{figure}
\centering
\includegraphics[width = 0.325\textwidth]{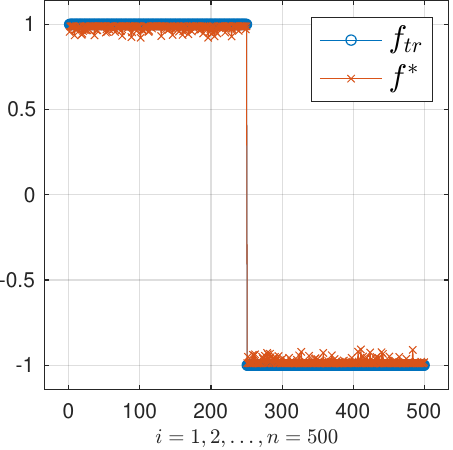} \hfill
\includegraphics[width = 0.325\textwidth]{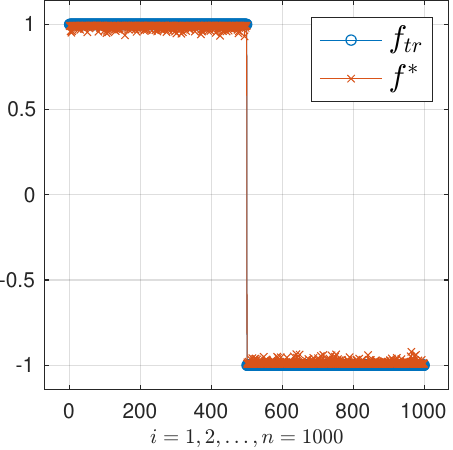} \hfill
\includegraphics[width = 0.325\textwidth]{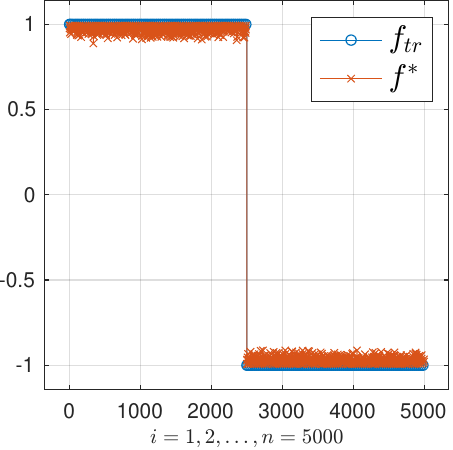} 

\includegraphics[width = 0.325\textwidth]{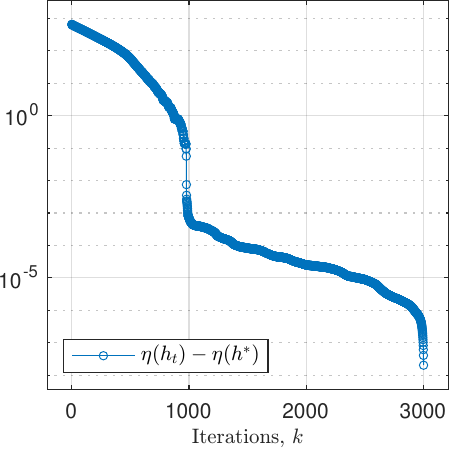} \hfill
\includegraphics[width = 0.325\textwidth]{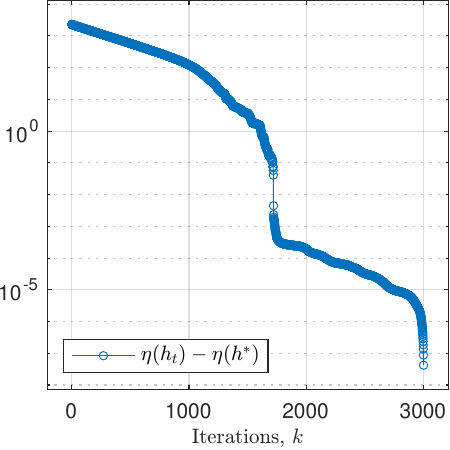} \hfill
\includegraphics[width = 0.325\textwidth]{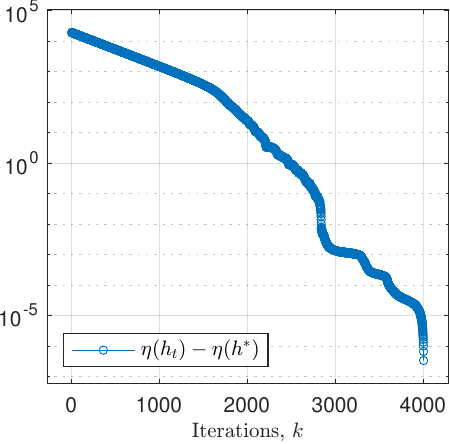} 

\includegraphics[width = 0.325\textwidth]{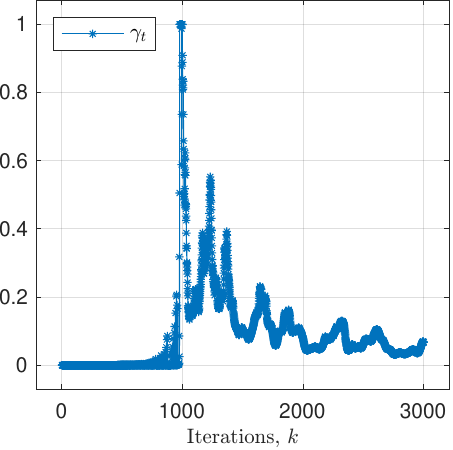} \hfill
\includegraphics[width = 0.325\textwidth]{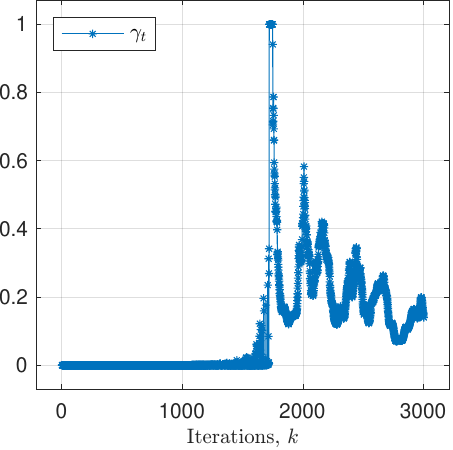} \hfill
\includegraphics[width = 0.325\textwidth]{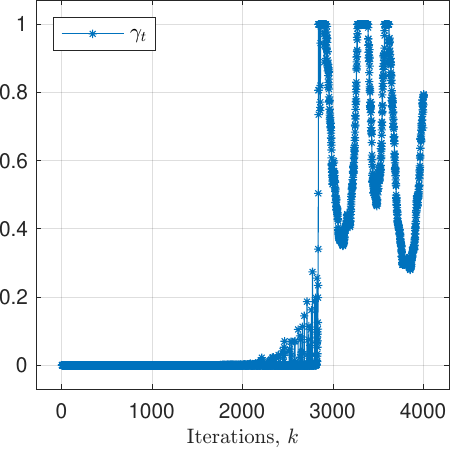}

\includegraphics[width = 0.325\textwidth]{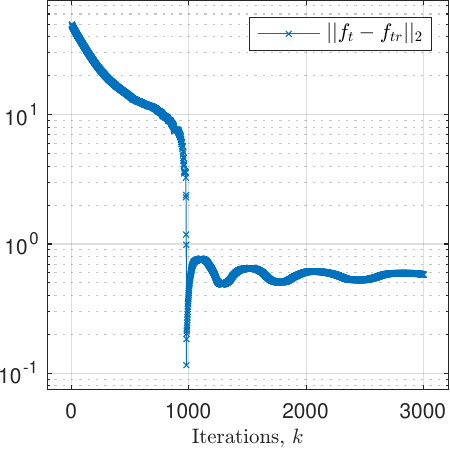} \hfill
\includegraphics[width = 0.325\textwidth]{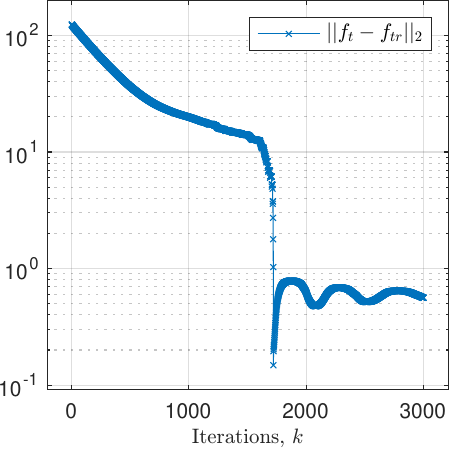} \hfill
\includegraphics[width = 0.325\textwidth]{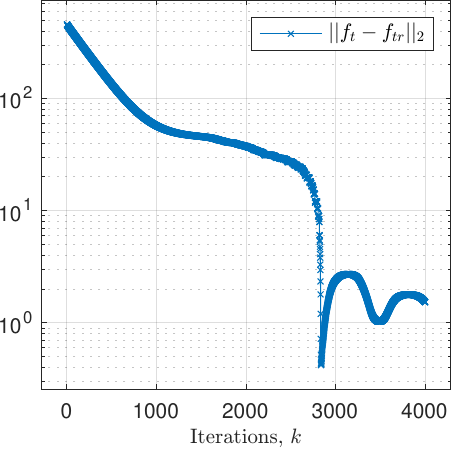}

\caption{Simulation results for Binary selection.}
\label{fig:binary-selection}
\end{figure}

\subsection{Results on Compressed Sensing}
\label{subsection-compressed-sensing}
For an image `I', let \mmode{\mathsf{I} \in \R{K}} be its vectorized form. Then for \mmode{i = 1,2,\ldots, m \sim 0.6 \times K}, we collect the linear measurement \mmode{x_i}, of the image `I' as \mmode{x_i = c_i\transp {\mathsf{I}} + w_i}; where \mmode{c_i \in \R{K}} is a random vector whose each entry is drawn uniformly from \mmode{[-0.5, 0.5]}, and \mmode{w_i} is the measurement noise drawn randomly from a Gaussian distribution with variance \mmode{\sigma^2 = 0.0055}, and independently from everything else. The task of recovering the image `I' from its measurements \mmode{x} is formulated as the LIP
\begin{equation}
\label{eq:compressed-sensing}
\begin{cases}
\begin{aligned}
& \min_{\minvarf} && \pnorm{f}{1} \\
& \sbjto && \norm{x - (CD)f} \leq \eps ,
\end{aligned}
\end{cases}
\end{equation}
where \mmode{\eps = \sigma \sqrt{m}} and \mmode{D \in \R{K \times K}} is chosen to be the dictionary of 2d-IDCT basis vectors since natural images are sparse in 2d-DCT basis. Thus, \eqref{eq:compressed-sensing} is a version of the LIP \eqref{eq:lip-main} with objective function \mmode{\cost(\cdot) = \pnorm{\cdot}{1}} and parameters \mmode{x , \linmap = CD , \eps}. We apply FLIPS with an AQO to find an optimal solution \mmode{f\opt} to the LIP \eqref{eq:compressed-sensing}, and the image `I' is reconstructed as \mmode{D f\opt}. 

\begin{figure}
\centering
\includegraphics[width = 0.275\textwidth]{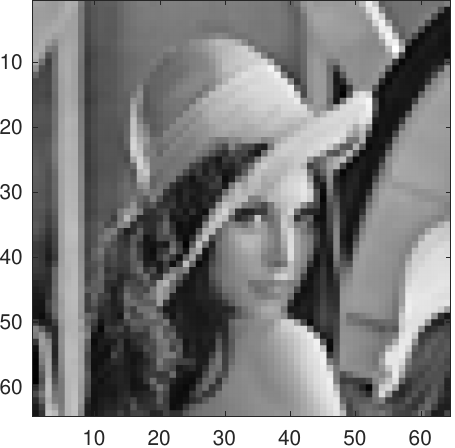} \hfill
\includegraphics[width = 0.275\textwidth]{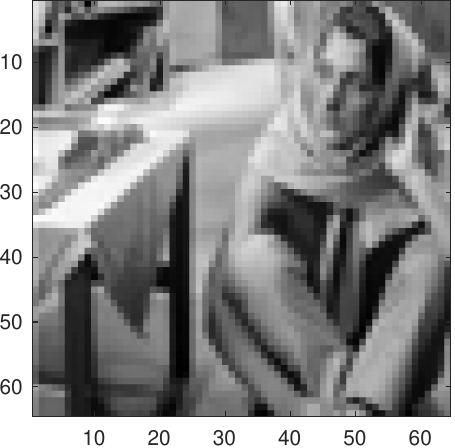} \hfill
\includegraphics[width = 0.275\textwidth]{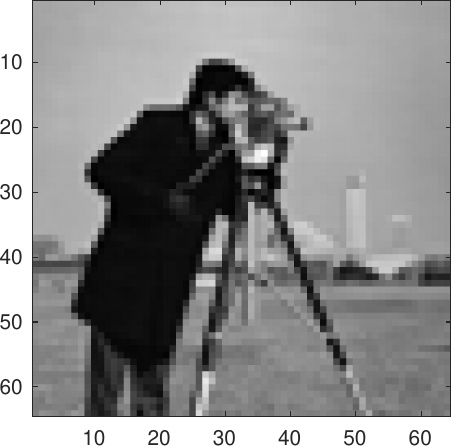}

\includegraphics[width = 0.275\textwidth]{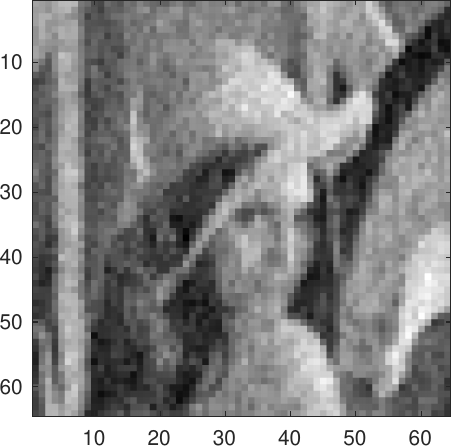} \hfill
\includegraphics[width = 0.275\textwidth]{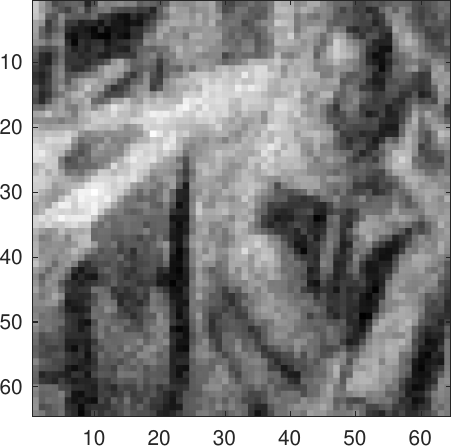} \hfill
\includegraphics[width = 0.275\textwidth]{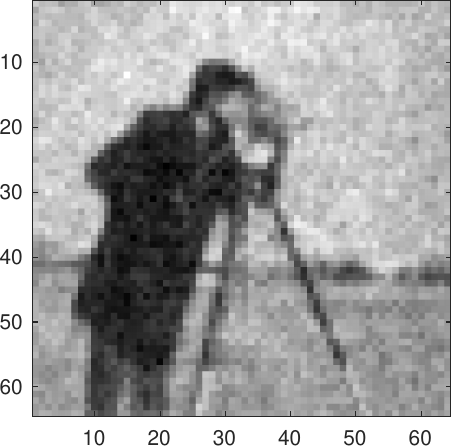}

\includegraphics[width = 0.275\textwidth]{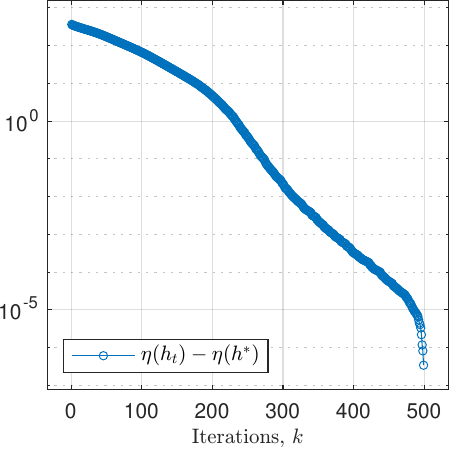} \hfill
\includegraphics[width = 0.275\textwidth]{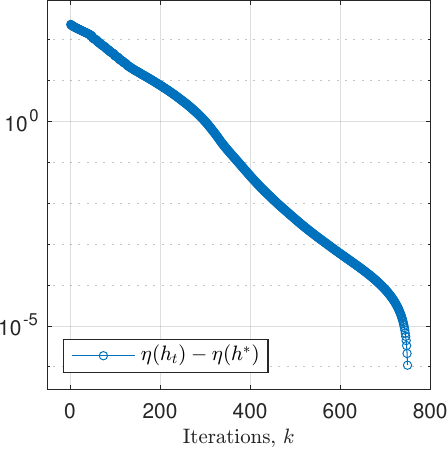} \hfill
\includegraphics[width = 0.275\textwidth]{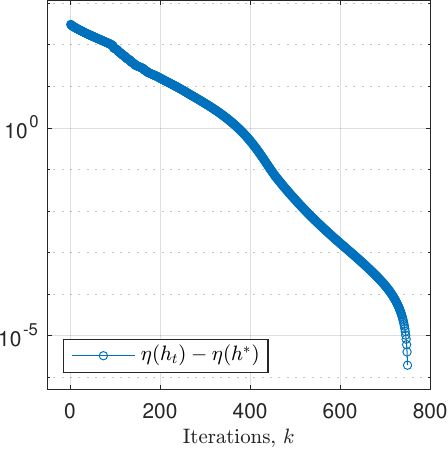} 

\includegraphics[width = 0.275\textwidth]{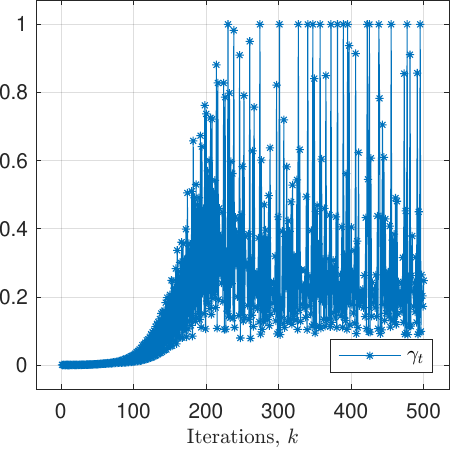} \hfill
\includegraphics[width = 0.275\textwidth]{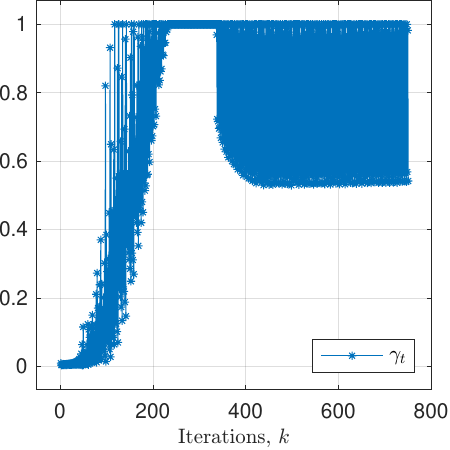} \hfill
\includegraphics[width = 0.275\textwidth]{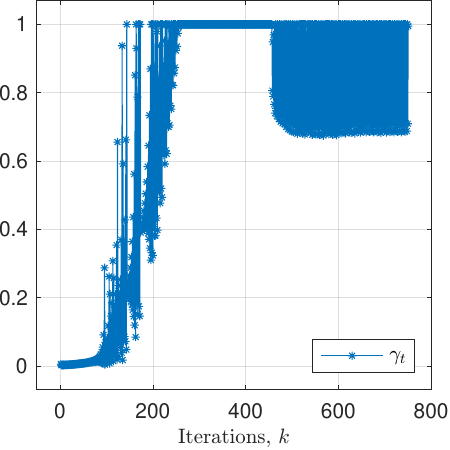}

\includegraphics[width = 0.275\textwidth]{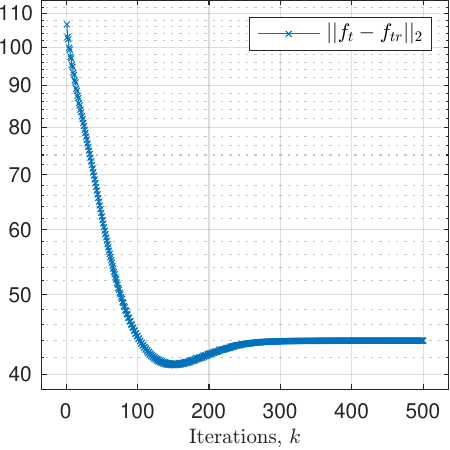} \hfill
\includegraphics[width = 0.275\textwidth]{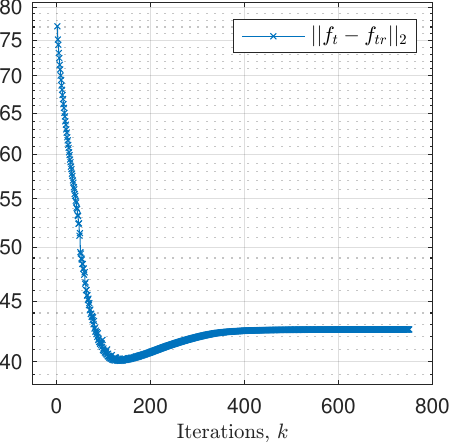} \hfill
\includegraphics[width = 0.275\textwidth]{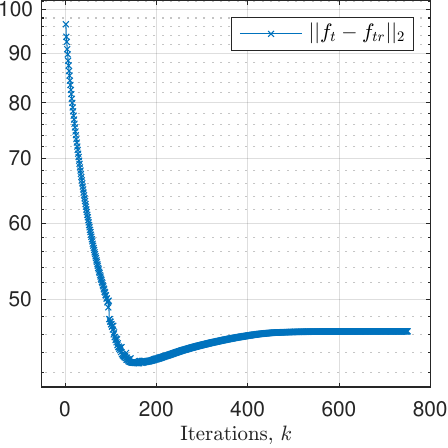}

\caption{Simulation results for Compressed sensing.}
\label{fig:compressed-sensing}
\end{figure}

We consider the Compressed Sensing problem \eqref{eq:compressed-sensing} for the standard images of `Lena', `Cameraman', and `Barbara', each of size \mmode{K = 64 \times 64} pixels. We solve each one of them using FLIPS using an AQO oracle with parameter values tuned for faster convergence. The results on recovery and convergence attributes of FLIPS are shown in Figure \ref{eq:compressed-sensing}. The left column corresponds to the results for the `Lena' image, followed by `Barbara', and `Cameraman' images to their right. The true images are shown in the first row and the recovered images from FLIPS are shown in the second row. Following these, the plots for the {\em sub-optimality}: \mmode{\etah{t} - \eta (\minvar \opt)}, {\em Distance to true solution}: \mmode{\pnorm{\minvarf_t - \mathsf{I}}{2}} is shown where \mmode{\minvarf_t = \etah{t} \minvar_t}. Finally, at the bottom, we plot the sequence of {\em step-sizes}: \mmode{\stepsize_t} as a function of iterations of FLIPS.

\subsection{Results on image denoising}
\label{subsection:image-denoising}
Finally, we also consider another image processing task of denoising an image to demonstrate the performance of FLIPS and compare it with other state-of-the-art methods that denoise an image by solving the corresponding constrained LIP \eqref{eq:lip-main}. In particular, we consider the Chambolle-Pock algorithm (with the current best theoretical convergence guarantee of \mmode{O(\nicefrac{1}{k})}) \citep{Chambolle2016OnAlgorithm, Chambolle2010AImaging}, and also the more well known C-SALSA algorithm \citep{Afonso2009AnProblems}.

\begin{table}[h]
\caption{Comparison of FLIPS with CP and C-SALSA algorithms, for image denoising with sliding patches on the `Lena', `Barbara', and `Cameraman' images.}
\label{tab:denoising-comparison-cpu-times}
\begin{tabular}{|ccccccc|}
\hline
\multicolumn{1}{|c|}{}        & \multicolumn{2}{c|}{\textbf{FLIPS}}                             & \multicolumn{2}{c|}{\textbf{CP}}                                & \multicolumn{2}{c|}{\textbf{C-SALSA}}                           \\ \hline
\multicolumn{1}{|c|}{}        & \multicolumn{1}{|c|}{CPU time} & \multicolumn{1}{|c|}{iterations} & \multicolumn{1}{|c|}{CPU time} & \multicolumn{1}{|c|}{iterations} & \multicolumn{1}{|c|}{CPU time} & \multicolumn{1}{|c|}{iterations} \\ \hline
\multicolumn{7}{|c|}{\textit{\textbf{Results for ``Lena''}}}                                                                                                                                                                            \\ \hline
\multicolumn{1}{|c|}{4 x 4}   & \multicolumn{1}{c|}{\cellcolor{blue!10} 4.95}     & \multicolumn{1}{c|}{5.188}      & \multicolumn{1}{c|}{14.54}    & \multicolumn{1}{c|}{45.18}      & \multicolumn{1}{c|}{16.9}     & 43.61                           \\ \hline
\multicolumn{1}{|c|}{8 x 8}   & \multicolumn{1}{c|}{\cellcolor{blue!10} 5.6}      & \multicolumn{1}{c|}{4.9}        & \multicolumn{1}{c|}{17.7}     & \multicolumn{1}{c|}{45.73}      & \multicolumn{1}{c|}{14.27}    & 34.33                           \\ \hline
\multicolumn{1}{|c|}{16 x 16} & \multicolumn{1}{c|}{\cellcolor{blue!10} 16}     & \multicolumn{1}{c|}{3}        & \multicolumn{1}{c|}{33.85}    & \multicolumn{1}{c|}{47.6}       & \multicolumn{1}{c|}{27.8}     & 44.5                           \\ \hline
\multicolumn{1}{|c|}{32 x 32} & \multicolumn{1}{c|}{\cellcolor{blue!10} 27.23}    & \multicolumn{1}{c|}{2.6}        & \multicolumn{1}{c|}{70.23}     & \multicolumn{1}{c|}{49}         & \multicolumn{1}{c|}{43.2}     & 43.2                            \\ \hline

\multicolumn{7}{|c|}{\textit{\textbf{Results for ``Barbara''}}}                                                                                                                                                                         \\ \hline
\multicolumn{1}{|c|}{4 x 4}   & \multicolumn{1}{c|}{\cellcolor{blue!10} 5.46}     & \multicolumn{1}{c|}{5.37}       & \multicolumn{1}{c|}{15.1}     & \multicolumn{1}{c|}{46.2}       & \multicolumn{1}{c|}{15.95}    & 44.1                            \\ \hline
\multicolumn{1}{|c|}{8 x 8}   & \multicolumn{1}{c|}{\cellcolor{blue!10} 5.74}     & \multicolumn{1}{c|}{5.08}       & \multicolumn{1}{c|}{18.13}    & \multicolumn{1}{c|}{47.66}      & \multicolumn{1}{c|}{14.5}     & 36.3                            \\ \hline
\multicolumn{1}{|c|}{16 x 16} & \multicolumn{1}{c|}{\cellcolor{blue!10} 16}     & \multicolumn{1}{c|}{3}       & \multicolumn{1}{c|}{33.67}       & \multicolumn{1}{c|}{49.1}     & \multicolumn{1}{c|}{27.58}       & 47.5                              \\ \hline
\multicolumn{1}{|c|}{32 x 32} & \multicolumn{1}{c|}{\cellcolor{blue!10} 27.6}    & \multicolumn{1}{c|}{2.6}          & \multicolumn{1}{c|}{71.5}    & \multicolumn{1}{c|}{49.7}       & \multicolumn{1}{c|}{ 44.5}     & 43                              \\ \hline

\multicolumn{7}{|c|}{\textit{\textbf{Results for ``Cameraman''}}}                    \\ \hline
\multicolumn{1}{|c|}{4 x 4}   & \multicolumn{1}{c|}{\cellcolor{blue!10} 5.3}      & \multicolumn{1}{c|}{5.33}       & \multicolumn{1}{c|}{13.73}    & \multicolumn{1}{c|}{44.4}       & \multicolumn{1}{c|}{16.28}    & 43.9                            \\ \hline
\multicolumn{1}{|c|}{8 x 8}   & \multicolumn{1}{c|}{\cellcolor{blue!10} 5.92}     & \multicolumn{1}{c|}{5.313}      & \multicolumn{1}{c|}{17.93}    & \multicolumn{1}{c|}{43.5}       & \multicolumn{1}{c|}{12.63}    & 31.18                           \\ \hline
\multicolumn{1}{|c|}{16 x 16} & \multicolumn{1}{c|}{\cellcolor{blue!10} 15.9}     & \multicolumn{1}{c|}{3}          & \multicolumn{1}{c|}{30.45}    & \multicolumn{1}{c|}{41.75}         & \multicolumn{1}{c|}{25.09}    & 38.7                              \\ \hline
\multicolumn{1}{|c|}{32 x 32} & \multicolumn{1}{|c|}{\cellcolor{blue!10} 28.4}         & \multicolumn{1}{|c|}{2.88}           & \multicolumn{1}{|c|}{66.78}         & \multicolumn{1}{|c|}{46.88}           & \multicolumn{1}{|c|}{42.96}         & \multicolumn{1}{|c|}{42.69}           \\ \hline
\end{tabular}
\end{table}

We consider the three images: `Lena', `Barbara', and `Cameraman', each of size \mmode{256 \times 256}. On each of those images, a Gaussian noise of variance $\sigma^2 = 0.0055$ is added with the MATLAB function \texttt{imnoise}. The image is then denoised by denoising every patch (and overlapping) of a fixed size \mmode{m \times m}. Then the final image is reconstructed by taking the average value of an individual pixel over all the patches it belongs to. Denoising a given patch of size \mmode{m \times m} corresponds to solving the LIP \eqref{eq:lip-main} with \mmode{\cost(\rep)=\pnorm{\rep}{1}}, the linear map \mmode{\linmap} as the 2d-inverse discrete cosine transform for \mmode{m \times m} patches (computed using the function \texttt{idct2} in Matlab), and \mmode{\eps = \sigma m}. The experiment is repeated with different patch sizes of \mmode{m = 4, 8, 16, 32, 64}, and for each image and patch size \mmode{m}, the convergence results are averaged over 10 independent trials with independent noise. The parameter values in the Chambolle-Pock algorithm, C-SALSA, and the AQO oracle parameters in FLIPS are tuned to get the best results for each patch size. 

In Table \ref{tab:denoising-comparison-cpu-times}, we tabulate the average number of iterations per patch required until convergence for each algorithm, averaged over different patches of the respective image and different instances of noise. Besides the average number of iterations, we also tabulate the total CPU time each algorithm takes to solve the denoising problem for all the patches, and averaged over different instances of noise. The convergence criterion for each patch \mmode{\mathsf{I}} is chosen to be \mmode{\min \big\{ k : \norm{\mathsf{I}_k - \mathsf{I}^* }_2 \leq 10^{-3} \norm{\mathsf{I}^*}_2 \big\} }, where \mmode{\mathsf{I}^*} is the optimal solution for the respective patch computed apriori by running FLIPS for a large number of iterations. Of course, the time required to compute \mmode{ \norm{\mathsf{I}_k - \mathsf{I^*} }_2 } at each iteration is excluded from the CPU times.

\begin{figure}
\centering
\includegraphics[width = 0.3\textwidth]{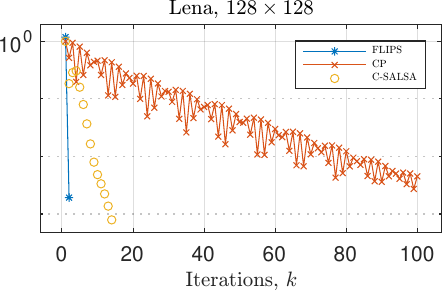} \hfill
\includegraphics[width = 0.3\textwidth]{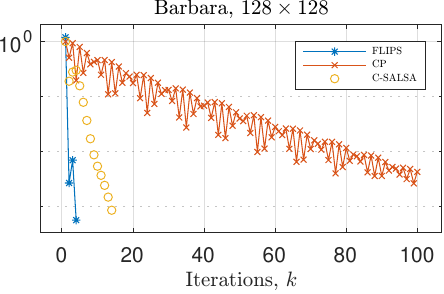} \hfill
\includegraphics[width = 0.3\textwidth]{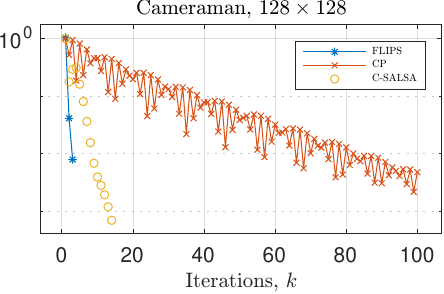}

\vspace{0.5cm}
\includegraphics[width = 0.3\textwidth]{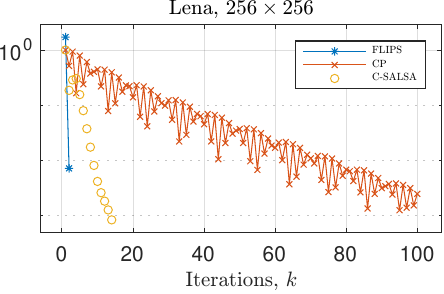} \hfill
\includegraphics[width = 0.3\textwidth]{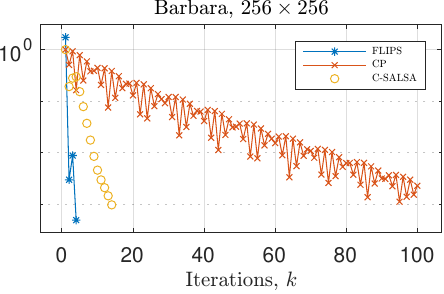} \hfill
\includegraphics[width = 0.3\textwidth]{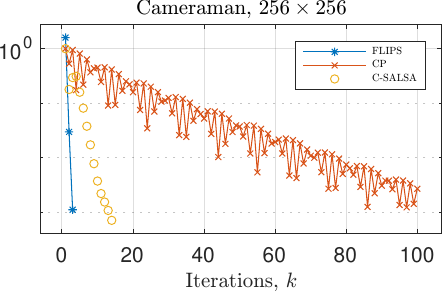}

\vspace{0.5cm}
\includegraphics[width = 0.3\textwidth]{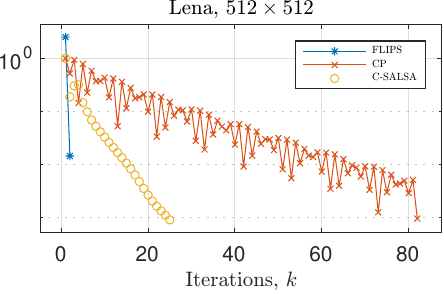} \hfill
\includegraphics[width = 0.3\textwidth]{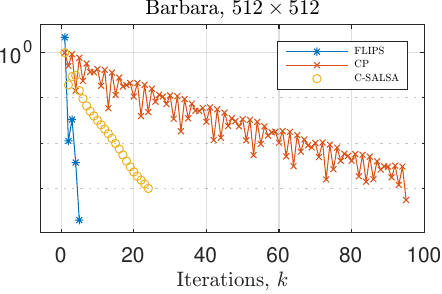} \hfill
\includegraphics[width = 0.3\textwidth]{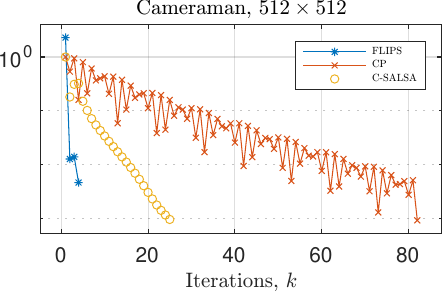}

\caption{Comparison of FLIPS with C-SALSA and CP algorithms on image denoising (full images). }
\label{fig:image-denoising-comparison}
\end{figure}

\begin{table}[]
\caption{Comparison of FLIPS with CP and C-SALSA algorithms, for image denoising on full images of `Lena', `Barbara', and `Cameraman' images.}
\label{tab:CPUtimes-image-denoising-full-image}
\begin{tabular}{|ccccccc|}
\hline
\multicolumn{1}{|c|}{}          & \multicolumn{2}{c|}{FLIPS}                                        & \multicolumn{2}{c|}{CP}                                           & \multicolumn{2}{c|}{C-SALSA}                 \\ \hline
\multicolumn{1}{|c|}{}          & \multicolumn{1}{c|}{CPU time} & \multicolumn{1}{c|}{\# iteration} & \multicolumn{1}{c|}{CPU time} & \multicolumn{1}{c|}{\# iteration} & \multicolumn{1}{c|}{CPU time} & \# iteration \\ \hline
\multicolumn{7}{|c|}{\textbf{\textit{Results for ``Lena''}}}                                                                                                                                                                             \\ \hline
\multicolumn{1}{|c|}{128 x 128} & \multicolumn{1}{c|}{\cellcolor{blue!10} 0.023}    & \multicolumn{1}{c|}{2}            & \multicolumn{1}{c|}{0.163}    & \multicolumn{1}{c|}{60}           & \multicolumn{1}{c|}{0.027}    & 9            \\ \hline
\multicolumn{1}{|c|}{256 x 256} & \multicolumn{1}{c|}{\cellcolor{blue!10} 0.045}    & \multicolumn{1}{c|}{2.25}         & \multicolumn{1}{c|}{0.275}    & \multicolumn{1}{c|}{53}           & \multicolumn{1}{c|}{0.07}     & 9            \\ \hline
\multicolumn{1}{|c|}{512 x 512} & \multicolumn{1}{c|}{ \cellcolor{blue!10}0.219}    & \multicolumn{1}{c|}{3}            & \multicolumn{1}{c|}{0.54}     & \multicolumn{1}{c|}{42}           & \multicolumn{1}{c|}{0.39}     & 16           \\ \hline
\multicolumn{7}{|c|}{\textbf{\textit{Results for ``Barbara''}}}                                                                                                                                                                          \\ \hline
\multicolumn{1}{|c|}{128 x 128} & \multicolumn{1}{c|}{\cellcolor{blue!10} 0.035}    & \multicolumn{1}{c|}{2}            & \multicolumn{1}{c|}{0.145}    & \multicolumn{1}{c|}{53}           & \multicolumn{1}{c|}{0.036}    & 9            \\ \hline
\multicolumn{1}{|c|}{256 x 256} & \multicolumn{1}{c|}{\cellcolor{blue!10} 0.043}    & \multicolumn{1}{c|}{2}            & \multicolumn{1}{c|}{0.285}    & \multicolumn{1}{c|}{53}           & \multicolumn{1}{c|}{0.068}    & 9            \\ \hline
\multicolumn{1}{|c|}{512 x 512} & \multicolumn{1}{c|}{\cellcolor{blue!10} 0.286}    & \multicolumn{1}{c|}{4}            & \multicolumn{1}{c|}{0.627}    & \multicolumn{1}{c|}{53}           & \multicolumn{1}{c|}{0.347}    & 15           \\ \hline
\multicolumn{7}{|c|}{\textbf{\textit{Results for ``Cameraman''}}}                                                                                                                                                                        \\ \hline
\multicolumn{1}{|c|}{128 x 128} & \multicolumn{1}{c|}{0.026}    & \multicolumn{1}{c|}{3}            & \multicolumn{1}{c|}{0.137}    & \multicolumn{1}{c|}{55}           & \multicolumn{1}{c|}{\cellcolor{blue!10} 0.025}    & 9            \\ \hline
\multicolumn{1}{|c|}{256 x 256} & \multicolumn{1}{c|}{\cellcolor{blue!10} 0.06}     & \multicolumn{1}{c|}{3}            & \multicolumn{1}{c|}{0.289}    & \multicolumn{1}{c|}{53}           & \multicolumn{1}{c|}{0.073}    & 9            \\ \hline
\multicolumn{1}{|c|}{512 x 512} & \multicolumn{1}{c|}{\cellcolor{blue!10} 0.248}    & \multicolumn{1}{c|}{4}            & \multicolumn{1}{c|}{0.525}    & \multicolumn{1}{c|}{42}           & \multicolumn{1}{c|}{0.388}    & 16           \\ \hline
\end{tabular}
\end{table}

In addition, we also consider the denoising problem for the three images by directly working on the entire image as a single patch instead of considering smaller and sliding patches as in Table \ref{tab:denoising-comparison-cpu-times}. For this, we first obtain noise-free images of size \mmode{128 \times 128}, \mmode{256 \times 256}, and \mmode{512 \times 512} pixels. Then, similar to the previous experiment, a Gaussian noise of variance \mmode{ \sigma^2 = 0.0055} is added with the MATLAB function \texttt{imnoise} to obtain the noisy image. Then the noisy image is denoised by solving the corresponding LIP (for the full image) using FLIPS, CP, and C-SALSA algorithms with respective parameters tuned to give better respective convergence results. 

We compare their convergence attributes on the metric \mmode{\norm{\mathsf{I}_k - \mathsf{I}^* } }, where \mmode{\mathsf{I}_k} is the image after \mmode{k} iterations of the respective algorithm and \mmode{ \mathsf{I}^* } is the optimal solution obtained apriori by running FLIPS for many iterations (and confirmed with other methods for optimality). Convergence plots for FLIPS, CP, and C-SALSA algorithms for image denoising on the three images of `Lena', `Barbara', and `Cameraman' for varying sizes of \mmode{128 \times 128}, \mmode{256 \times 256}, and \mmode{512 \times 512} are provided in Figure \ref{fig:image-denoising-comparison}, and the corresponding CPU times (averaged over 10 iterations of different noise) is tabulated in Table \ref{tab:CPUtimes-image-denoising-full-image}. It must be observed that FLIPS only takes approximately \mmode{\sim 4} iterations to converge to the optimal solution, which is incredibly fast.

\subsection{Comparison with FISTA for a trajectory of solutions}
\label{subsection:FLIPS-vs-FISTA}
Finally, we would like to compare the convergence of FLIPS with FISTA for solving an LIP. In this regard, we first generate the problem data, namely: (i)-liner map \mmode{\linmap}, (ii)-Sparse vectors \mmode{F^{tr}}, and (iii)-noisy measurements \mmode{X}, randomly and independently from each other as
\begin{enumerate}
    \item \mmode{\linmap \in \R{m \times K}} is generated randomly by sampling each entry \mmode{\linmap_{ij} \sim \mathcal{N}(0,1)}

    \item \mmode{F^{tr} \in \R{K \times T} }, is randomly generated by first sampling every entry \mmode{F^{tr}_{ij} \sim \mathcal{N} (0,1)}, and then every column is made to be \mmode{S}-sparse by zeroing all but the \mmode{S}-largest entries in magnitude

    \item We first obtain the true measurements as \mmode{ \R{m, N} \ni X^{tr} = \linmap \cdot F^{tr} }, from which the noisy measurements \mmode{X = X^{tr} + \sigma_w W} are obtained by adding an AWGN \mmode{W \in \R{m \times T}}. Each entry \mmode{W_{ij} \sim \mathcal{N}(0,1)} is iid and sampled independently from the previous data (i.e., \mmode{\linmap} and \mmode{F^{tr}}), and \mmode{\sigma_w > 0 } is a scalar constant chosen to satisfy a specified SNR level as
    \[
    \sigma_w = \pnorm{X^{tr}}{fro}\sqrt{ \big( \nicefrac{1}{mN} \big) 10^{\frac{-SNR}{10}} }
    \]
\end{enumerate}

Given the problem data: \mmode{(\linmap, X)}, we obtain two estimates of \mmode{F^{tr}} by solving two different formualtions of a Linear Inverse problem. Let \mmode{F^c(\eps)} be the estimate obtained by solving the \textit{constrained} formulation of the LIP using FLIPS, and \mmode{F^r(\lambda)} be the estimate obtained by solving the \mmode{\ell_1}-\textit{regularized LASSO} formulation using FISTA. To this end, for any \mmode{\eps > 0} and \mmode{\lambda > 0}, let us define
\begin{equation}
\label{eq:FLIPS-v-FISTA}
\begin{cases}
\begin{aligned}
F^{c}(\eps) & \in \argmin_{F} \quad \pnorm{F}{(1,1)} \quad \sbjto \ \pnorm{X - \linmap F}{fro} \leq \eps , \text{ and } \\
F^{r}(\lambda) & \in \argmin_{F} \quad \lambda \pnorm{F}{(1,1)} \ + \ \pnorm{X - \linmap F}{fro}^2 ,
\end{aligned}
\end{cases}
\end{equation}
where \mmode{\pnorm{M}{(1,1)} = \summ{i,j}{}\abs{M_{ij}} }, and \mmode{ \pnorm{M}{fro}^2 = \summ{i,j}{} \abs{M_{ij}}^2 }.

The problems \eqref{eq:FLIPS-v-FISTA} are solved for a range of values of \mmode{(\eps_t)_t} and \mmode{(\lambda_t)_t}, for \mmode{t = 1,2,\ldots,T}. The solutions \mmode{F^c (\eps_t)}, \mmode{F^r (\lambda_t)} are used as initial conditions when computing \mmode{F^c (\eps_{t + 1})}, \mmode{F^r (\lambda_{t + 1})}. Since computing the solutions \mmode{F^c (\eps)} and \mmode{F^r (\lambda)} is easier for larger values of \mmode{\eps} and \mmode{ \lambda}, we select the sequences \mmode{\sigma_w \sqrt{mN}  = \eps_0 \geq \cdots \eps_t \geq \eps_{t+1} \geq \cdots }, and \mmode{25 = \lambda_0  \geq \lambda_t \geq \lambda_{t + 1} \geq \cdots} in decreasing order to minimize the number of iterations required. The gradient-descent step-size \mmode{(\nicefrac{1}{\beta})} in FLIPS is chosen such that \mmode{\stepsizeh{} \sim 0.01}, which is achieved by selecting \mmode{\smooth^+ = 0.01 (\nicefrac{\smooth}{\stepsize}) }, whereas for FISTA, the gradient-descent step-size is chosen as \mmode{(\nicefrac{1}{L})}, where \mmode{L = \lambda_{max} (\linadj \linmap) } is the smoothness constant.

\begin{figure}[h]
    \centering
    \includegraphics[width=\linewidth]{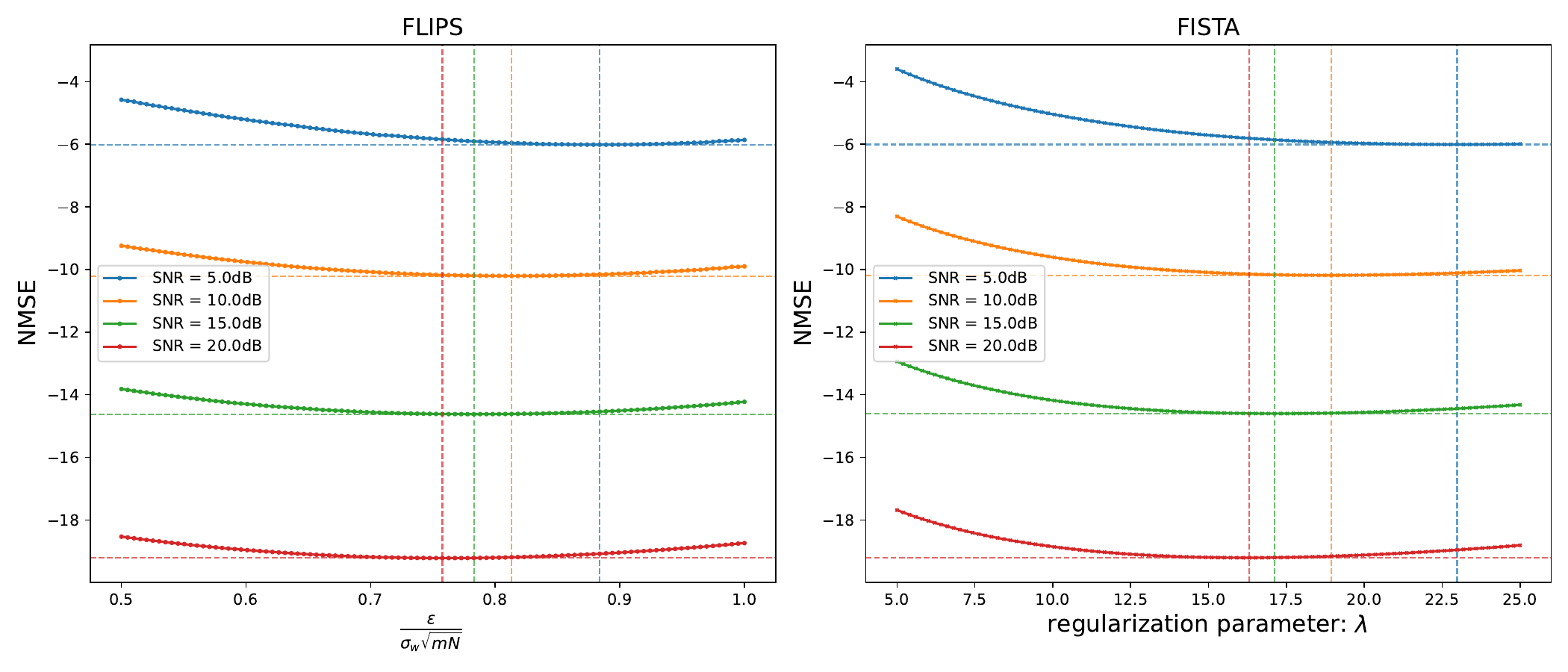}
    \caption{The Normalized Mean Squred Error (NMSE) of the optimal solution computed using FLIPS and FISTA as defined in \eqref{eq:NMSE-definition}. For every value of SNR, the minimum value of NMSE and the corresponding value of the parameters \mmode{\eps} and \mmode{\lambda} are indicated by the intersecting horizontal and vertical dotted lines of the same color.}
    \label{fig:FLIPS-v-FISTA-NMSE}
\end{figure}

For every \mmode{t =1,2,\ldots,T}, we record two metrics:
\begin{itemize}[leftmargin = *]
\item \emph{Normalized Mean Squared Error}  in reconstruction: This measures the quality of a solution \mmode{F} relative to the true solution \mmode{F^{tr}} defined as 
\begin{equation}
\label{eq:NMSE-definition}
NMSE(F) = 10 \log_{10} \left( \frac{ \pnorm{ \widehat{F}  - F^{tr} }{fro}^2 }{ \pnorm{F^{tr}}{fro}^2 } \right) , 
\text{ where } \widehat{F} = \frac{\inprod{X}{\linmap F}}{ \pnorm{\linmap F }{fro}^2 }  F .\footnote{\mmode{\widehat{F}} is a solution obtained by optimally scaling \mmode{F}.}
\end{equation}
In Figure \ref{fig:FLIPS-v-FISTA-NMSE}, we plot the mappings \mmode{t \mapsto NMSE(F^c (\eps_t))} and \mmode{t \mapsto NMSE( F^r(\lambda_t) ) } at various SNR levels.

\item \emph{Iterations to converge}: In Figure \ref{fig:FLIPS-v-FISTA-conv-iter}, we record the number of iterations required \mmode{k(\eps_t)}, \mmode{k(\lambda_t)} to converge for both FLIPS and FISTA as a function of \mmode{\eps_t} and \mmode{\lambda_t} respectively. Convergence is defined as satisfaction of first-order optimality conditions, where the parameters for convergence criteria are tuned separately to get best result for each algorithm. Moreover, in Figure \ref{fig:FLIPS-v-FISTA-cumulative-conv-iter}, we also plot the cumulative number of iterations required \mmode{K(\eps_t) \define \summ{s \geq t}{} k(\eps_s)}, and \mmode{K(\lambda_t) \define \summ{s \geq t}{} k(\lambda_s)} to find a solution for \mmode{\eps_t}, \mmode{\lambda_t} parameterized LIP with FLIPS and FISTA respectively starting from their initial values of \mmode{\eps_0} and \mmode{\lambda_0}. In Table \ref{tab:FLIPS-v_FISTA-iter}, we also report the cumulative number of iterations required to compute a solution corresponding to minimum NMSE value at a specific SNR level, i.e., the solution corresponding to the 

\end{itemize}

\begin{figure}[h]
  \centering
    \begin{subfigure}{\linewidth}
        \centering
      \includegraphics[width=\linewidth]{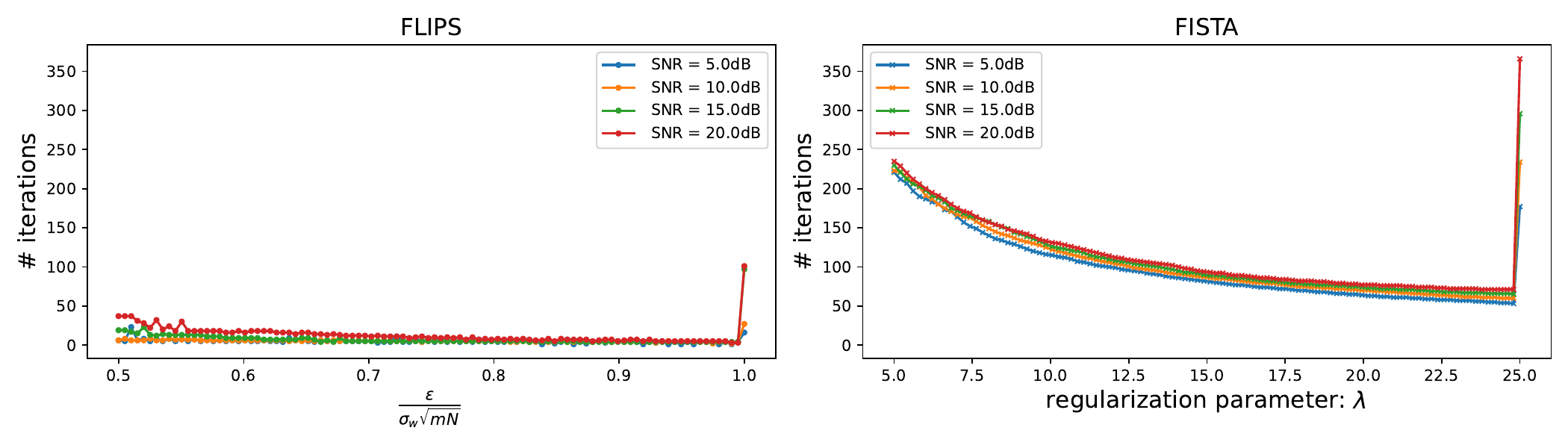}
      \caption{Iterations required to converge.}
      \label{fig:FLIPS-v-FISTA-conv-iter}
    \end{subfigure}\\
    \begin{subfigure}{\linewidth}
        \centering
      \includegraphics[width=\linewidth]{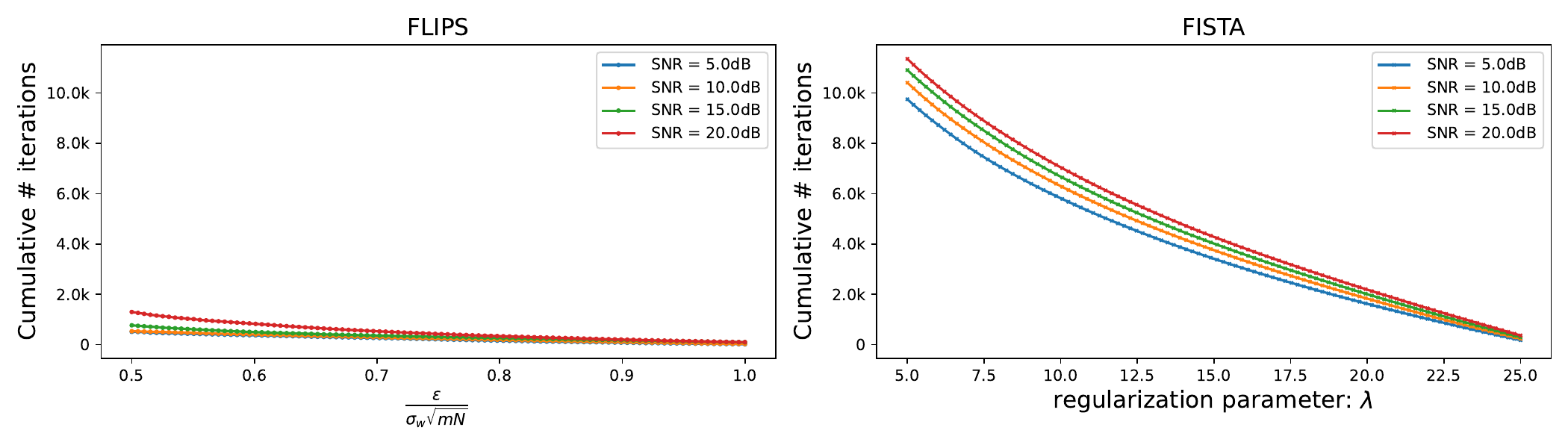}
      \caption{Cumulative number of iterations required to converge.}
      \label{fig:FLIPS-v-FISTA-cumulative-conv-iter}
    \end{subfigure}    
  \caption{Comparison of the speed of FLIPS and FISTA in solving LIPs.}
\end{figure}

\begin{table}[h]
    \centering
    \begin{tabular}{lcccc}
        \toprule
        SNR levels &  5dB & 10dB & 15dB & 20dB \\
        \midrule
        FLIPS (iterations)   & \textbf{93.0}   & \textbf{179.0}  & \textbf{282.0}  & \textbf{412.0}  \\
        FISTA (iterations)   & 731.0 & 2206.0 & 3114.0 & 3690.0 \\
        \bottomrule
    \end{tabular}
    \caption{Cumulative iteration required by FLIPS and FISTA to compute a minimum NMSE solution across different SNR levels.}
    \label{tab:FLIPS-v_FISTA-iter}
\end{table}

It can be easily seen from Figure \ref{fig:FLIPS-v-FISTA-NMSE}, that for \mmode{\eps \sim 0.85 \sigma_w \sqrt{mN}}, the NMSE for the solution of the constrained LIP is consistently close to the minimum for a range of SNR values. Therefore, the solution obtained from FLIPS for \mmode{\eps = 0.85 \sigma_w \sqrt{mN} } would be satisfactory for a range of SNR levels. So, in principle, one does not have to solve the constrained LIP for a range of values of \mmode{\eps}, which would significantly reduce the number of FLIPS iterations required to compute a satisfactory solution. For example, if we were to run FLIPS with \mmode{\eps = 0.85 \sigma_w \sqrt{mN} }, the number of iterations required to converge at SNR levels: 5,10,15, and 20dB would only be 29, 39, 71, and 81 respectively, which are considerably fewer than the ones reported in Table \ref{tab:FLIPS-v_FISTA-iter}.

\subsection{Comparison with Accelerated Projected Gradient Descent \mmode{\epsbar} (Remark \ref{remark:APGD-algorithm})}
In Figure \ref{fig:PAGDcompare}, we compare the sub-optimality \mmode{\etah{k} - \cost\opt} of iterates generated by FLIPS and the canonical projected accelerated gradient descent (PAGD) as in \eqref{eq:PAGD} (applied with \mmode{\big( \nicefrac{1}{b} \big) = 2.2\cdot10^{-6}} after tuning). Figure \ref{fig:PAGDcompare} clearly shows that FLIPS outperforms PAGD in terms of convergence.

\begin{figure}[h]
    \centering
    \begin{subfigure}[t]{0.5\textwidth}
        \centering
        \includegraphics[width = 0.85\textwidth]{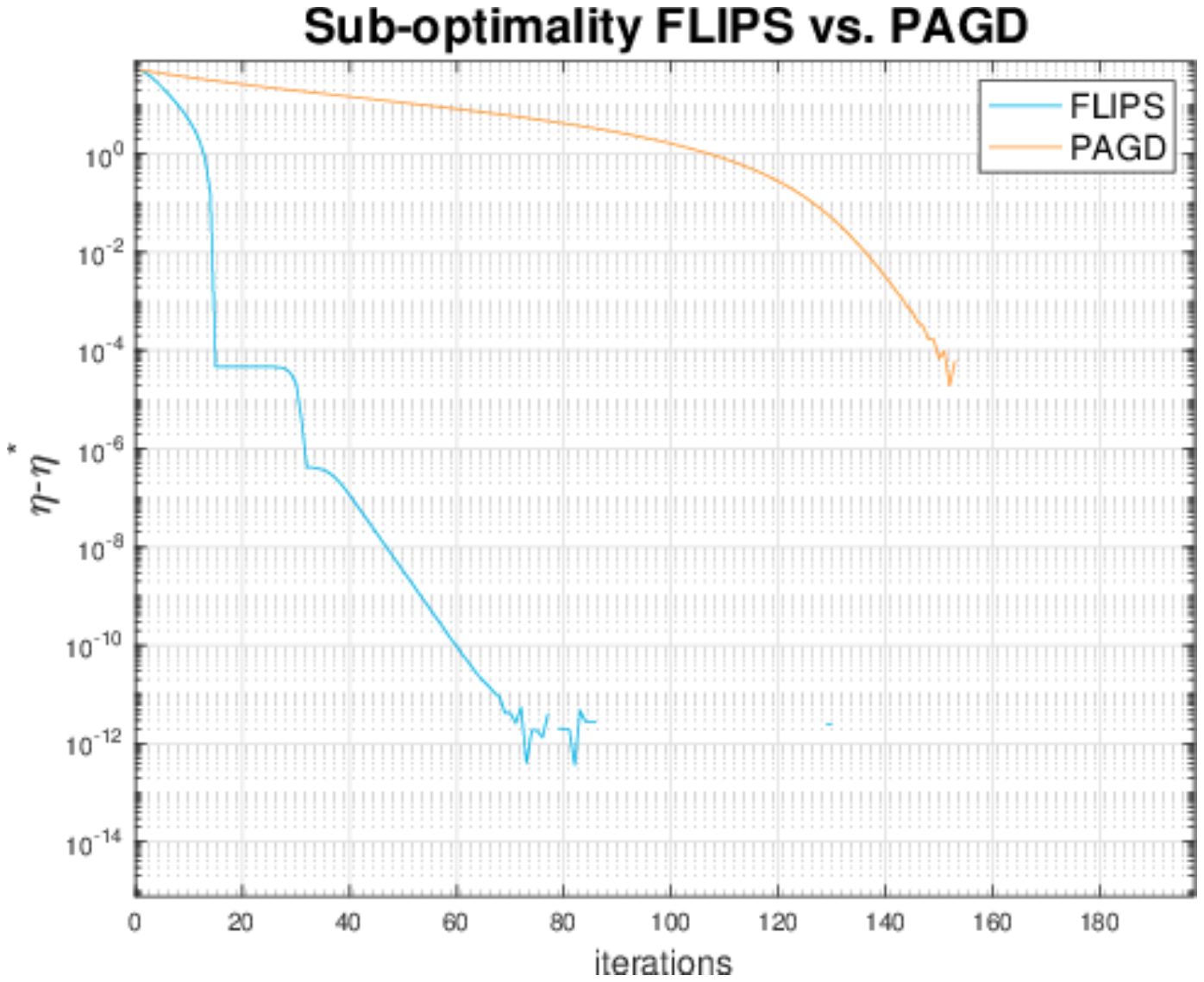}
\caption{FLIPS vs PAGD}\label{fig:PAGDcompare}
    \end{subfigure}%
    ~ 
    \begin{subfigure}[t]{0.5\textwidth}
        \centering
        \includegraphics[width = 0.95\textwidth]{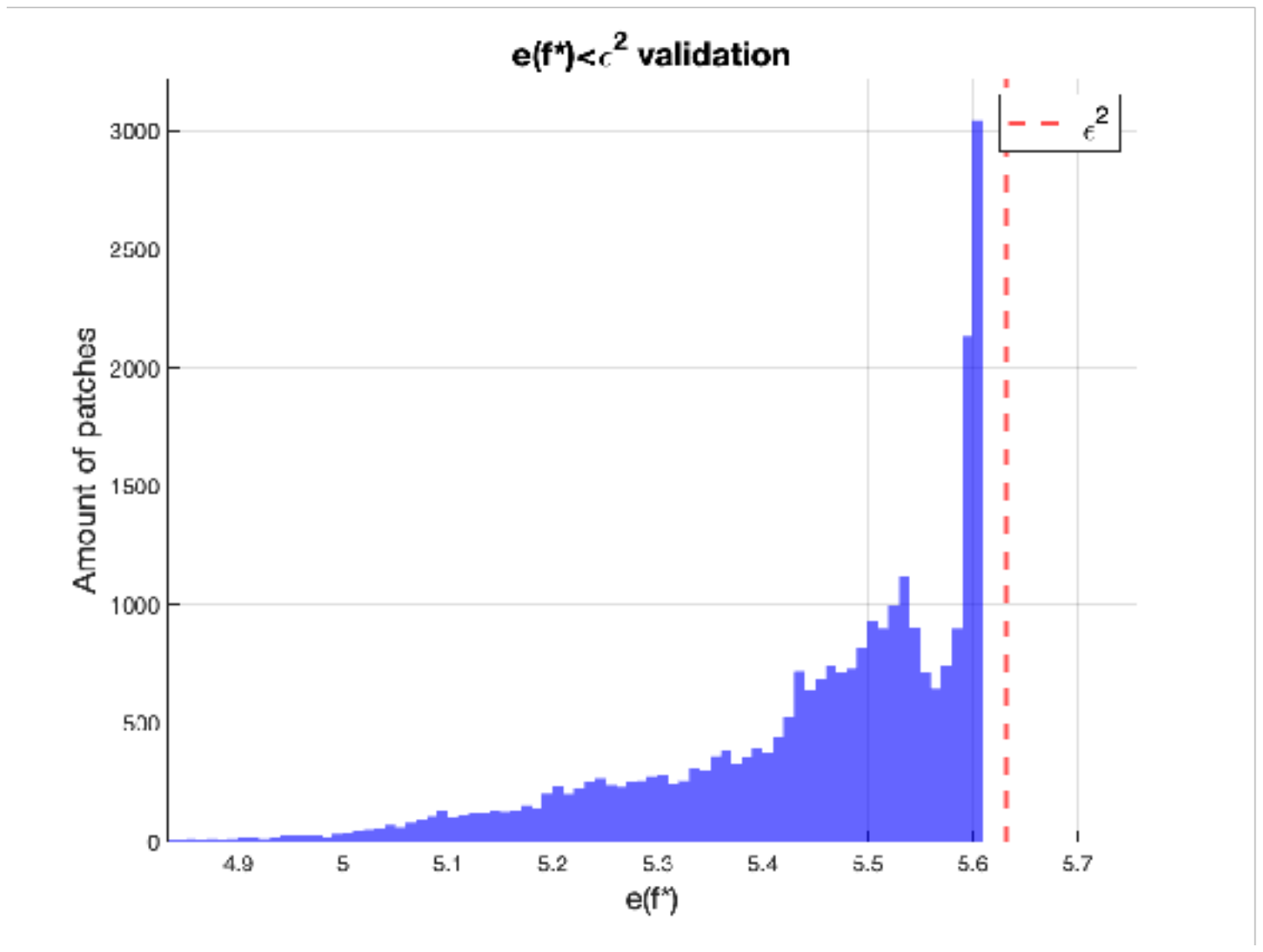}
    \caption{Validation of Remark \ref{remark:chosing-epsbar}.}
    \label{fig:ef_epsilonplot}
    \end{subfigure}
    \caption{Comparison of FLIPS with projected accelerated GD (a), and in (b) the histogram of the values of \mmode{e(\rep^{\opt})} for all 32\(\times\)32 patches of the 'cameraman image'.}
\end{figure}

\subsection{Empirical validation for choosing \mmode{\epsbar} (Remark \ref{remark:chosing-epsbar})}
\label{subsection:epsbar-numerical-evidence}
To empirically validate Remark \ref{remark:chosing-epsbar}, an experiment was conducted to check if \(e(\rep^{\opt})<\epsbar^2\) by a margin. From the full 200\(\times\)200 image, 28561 different 32\mmode{\times}32 patches were extracted and the corresponding LIPs were solved to obtain the optimal solution \mmode{\minvarf\opt} for each plot. Then the histogram of values \mmode{e(\minvarf\opt)} collected for all 28561 patches is plotted in Figure \ref{fig:ef_epsilonplot} (with a bandwidth of 0.01).
It can be seen that there is a clear gap between the maximum \(e(\rep^{\opt})\) and \(\epsbar^2\). Thus, verifying empirically Remark \ref{remark:chosing-epsbar}.

\section{Technical Proofs}\label{chap:discussion-proofs}

Let \mmode{\setlambda \define \{ \lambda \in \hilbert : \inprod{\maxvar}{\measurement} - \eps \norm{\maxvar} > 0 \}} and recall that \mmode{\sqfunc = \sqterm}.

\begin{proposition}
\label{proposition:max-problem-solution}
Let \mmode{\measurement , \linmap}, and \mmode{\eps} be such that Assumption \ref{assumption:lip-main} holds. For any \mmode{\minvar \in \costball}, considering the maximization problem 
\begin{equation}
\label{eq:etah-max-problem}
\begin{cases}
\sup\limits_{\maxvar \in \setlambda } \quad \lagrangian \define 2 \sqfunc \; - \; \inprod{\maxvar}{\linmap (\minvar)} ,
\end{cases}
\end{equation}
the following assertions hold.
\begin{enumerate}[leftmargin = *, label = \rm{(\roman*)}]
\item The maximization problem \eqref{eq:etah-max-problem} is bounded if and only if \mmode{\minvar \in \hcone}, and the maximal value is equal to \mmode{\etah{}}. In other words,
\begin{equation}
\label{eq:eta-as-sup-value}
\etah{} \ = \ \sup_{\maxvar \in \setlambda} \; \lagrangian \; \quad \text{ for all } \minvar \in \hcone.
\end{equation}

\item The maximization problem \eqref{eq:etah-max-problem} admits a unique maximizer \mmode{\maxvar(\minvar)} if and only if \mmode{\minvar \in \interior \left( \hcone \right) = \{ \minvar \in \hilbert : \inprod{\measurement}{\linmap(\minvar)} > 0, \text{ and } \eh{} < \eps^2 \} },  which is given by
\begin{equation}
\label{eq:lambdah-value}
\maxvar (\minvar) = \frac{\etah{}}{\normphih{} \sqrt{\epsminuse{}} } \big( \xminusetaphih{}  \big) .
\end{equation}
\end{enumerate}
\end{proposition}

\begin{figure}[h]
         \centering
         \includegraphics[width=0.6\textwidth]{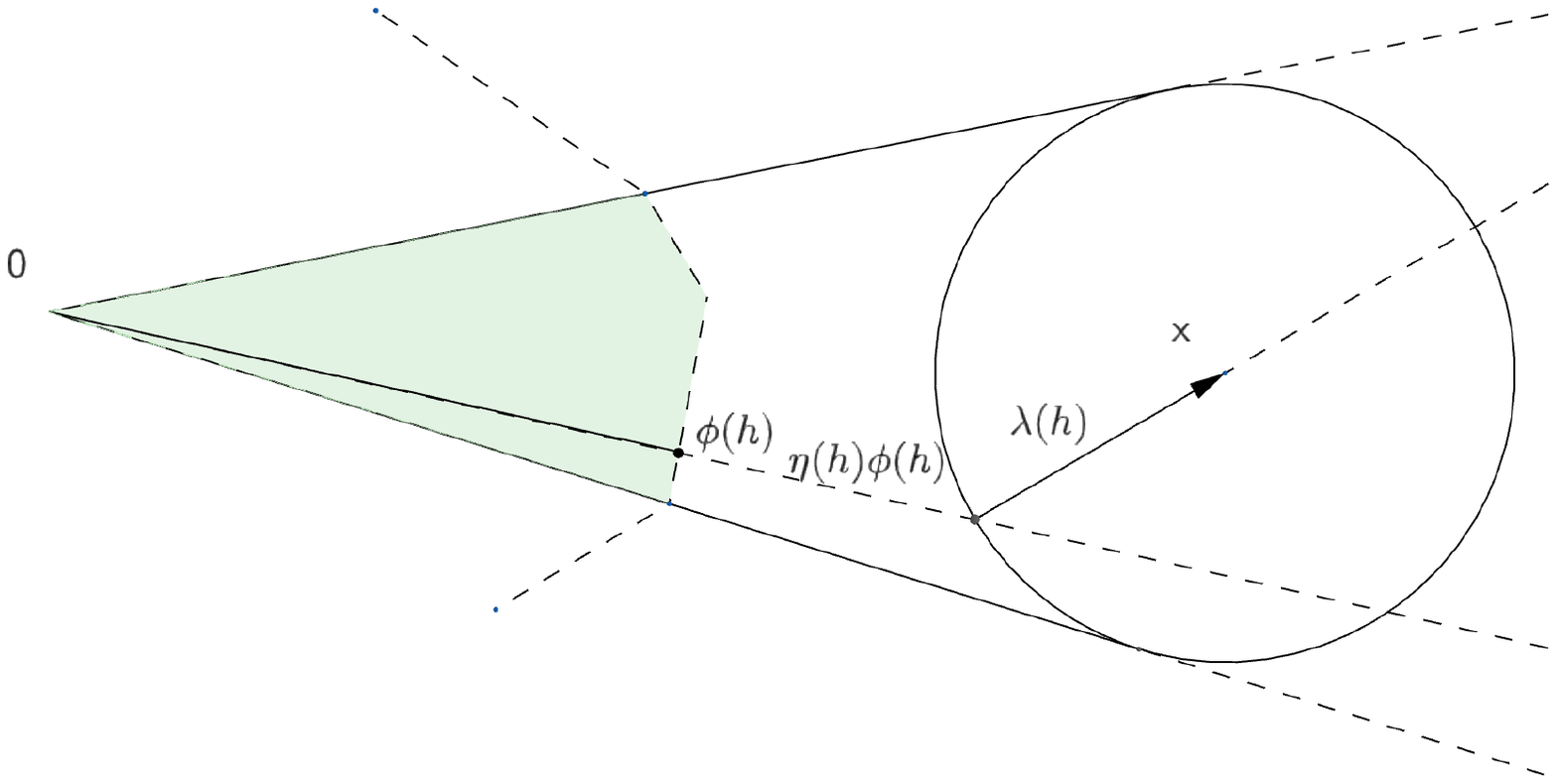}
\caption{Graphical overview of the direction of \mmode{\maxvarh{}}}\label{fig:Diagramlambda}
\end{figure}

\noindent The proof of Proposition \ref{proposition:max-problem-solution} relies heavily on \citep[Lemma 35, 36]{Sheriff2020NovelProblems} under the setting \mmode{r = 2, q = 0.5}, and \mmode{\delta = 0}. We shall first provide three lemmas from which Proposition \ref{proposition:max-problem-solution} follows easily.

\begin{lemma}[Unboundedness in \eqref{eq:etah-max-problem}]
\label{lemma:max-problem-unboundedness}
The maximization problem \eqref{eq:etah-max-problem} is unbounded for \mmode{\minvar \notin \hcone}.
\end{lemma}
\begin{proof}[Lemma \ref{lemma:max-problem-unboundedness}]
We first recall from \citep[Lemma 36 and assertion (iii) of Lemma 35]{Sheriff2020NovelProblems} that the maximal value of \eqref{eq:etah-max-problem} is unbounded if and only if there exists a \mmode{\maxvar'} such that the two following inequalities are satisfied simultaneously
\begin{equation}
\label{eq:unbounded-two-inequalities}
\inprod{\maxvar'}{\phih{}} \ \leq \ 0 \ < \ \inprod{\maxvar'}{\measurement} - \eps \norm{\maxvar'} .
\end{equation}
Since \mmode{\minvar \notin \hcone}, either \mmode{\ipxphih{} < 0} or \mmode{\eh{} > \eps^2}. On the one hand, if \mmode{\ipxphih{} < 0}, then we observe that \mmode{\maxvar' = \measurement} satisfies the two inequalities of \eqref{eq:unbounded-two-inequalities} since \mmode{\norm{\measurement} > \eps}. On the other hand, if \mmode{\eh{} > \eps^2}, then by considering \mmode{ \maxvar' = \measurement - \frac{\ipxphih{}}{\normphih{}^2} \phih{} }, we first observe that \mmode{\inprod{\maxvar'}{\phih{}} = 0}, and by Pythagoras theorem, we have \mmode{\norm{\maxvar'}^2 = \eh{}}. It is now easily verified that \mmode{\maxvar'} satisfies the two inequalities \eqref{eq:unbounded-two-inequalities} simultaneously since
\[
\begin{cases}
\begin{aligned}
\inprod{\maxvar'}{\phih{}} \; &= \; \ipxphih{} - \frac{\ipxphih{}}{\normphih{}^2} \inprod{\phih{}}{\phih{}} \; = \; 0 \\
\inprod{\maxvar'}{\measurement} - \eps \norm{\maxvar'} \; &= \; \normx^2 - \frac{\abs{\ipxphih{}}^2}{\normphih{}^2} - \eps \sqrt{\eh{}} \; = \; \sqrt{\eh{}} \big( \sqrt{\eh{}} - \eps \big) \; > \; 0.
\end{aligned}
\end{cases}
\]
Thus the lemma holds.
\end{proof}

\begin{lemma}[Optimal value of \eqref{eq:etah-max-problem}]
\label{lemma:max-problem-optimal-value}
If \mmode{\minvar \in \hcone}, then the maximal value of \eqref{eq:etah-max-problem} is finite and equal to \mmode{\etah{}} (as in \eqref{eq:etah-definition}).
\end{lemma}
\begin{proof}[Lemma \ref{lemma:max-problem-optimal-value}]
We now recall from \citep[Lemma 36, and (51)-Lemma 35]{Sheriff2020NovelProblems}, that the maximal value of \eqref{eq:etah-max-problem} is bounded if and only if the following minimum exists 
\begin{equation}
\label{eq:etah-alternate-definition}
\min \ \big\{ \theta \geq 0 : \norm{ \measurement - \theta \phih{} } \leq {\eps} \big\} .
\end{equation}
Clearly, the minimum in \eqref{eq:etah-alternate-definition} exists whenever the minimization problem is feasible. Suppose there exists some \mmode{\theta' \geq 0} such that \mmode{\norm{ \measurement - \theta' \phih{} } \leq {\eps} }, it is immediately seen that
\[
\begin{cases}
\begin{aligned}
\ipxphih{} \; &\geq \; \frac{1}{2\theta'} \Big( \xminuseps + \theta'^2 \normphih{}^2 \Big) \; > \; 0 , \text{ and} \\
\eh{} \; &= \; \min_{\theta \in \R{}} \; \norm{ \measurement - \theta \phih{} }^2 \; \leq \; \norm{ \measurement - \theta' \phih{} }^2 \leq {\eps}^2 .
\end{aligned}
\end{cases}
\]
Thus, \mmode{\minvar \in \hcone}. On the contrary, if \mmode{\minvar \in \hcone}, then it also seen similarly that \mmode{\theta' = \frac{\ipxphih{}}{\normphih{}^2}} is feasible for \eqref{eq:etah-alternate-definition}. Thus, the maximal value of \eqref{eq:etah-max-problem}, and the minimum in \eqref{eq:etah-alternate-definition} is finite if and only if \mmode{\minvar \in \hcone}.

It is immediately realised that the value of the minimum in \eqref{eq:etah-alternate-definition} corresponds to the smaller root of the quadratic equation \mmode{\norm{ \measurement - \theta \phih{} }^2 = \eps^2}. Dividing throughout therein by \mmode{\theta^2}, we obtain a different quadratic equation
\[
\frac{1}{\theta^2} \xminuseps \; - \; \frac{2}{\theta} \ipxphih{} \; + \; \normphih{}^2 = 0 .
\]
Selecting the larger root (and hence smaller \mmode{\theta}) gives us that for every \mmode{\minvar \in \hcone}, the optimal value of \eqref{eq:etah-max-problem} (and \eqref{eq:etah-alternate-definition}) is
\[
\frac{\norm{\measurement}^2 - \eps^2}{\ipxphih{} \; + \; \normphih{} \sqrt{ \eps^2 - \eh{}  } }  = \etah{} .
\]
Alternatively, if one selects the smaller root of the quadratic equation \mmode{\norm{ \measurement - \theta \phih{} }^2 = \eps^2}, one gets the expression of eta provided in \eqref{eq:eta-alternate-value}.
\end{proof}

\begin{remark}[Quadratic equation for \mmode{\etah{}}]\rm{
It is apparent from the proof of the Lemma \ref{lemma:max-problem-optimal-value} that for any \mmode{\minvar \in \hcone}, \mmode{\etah{}} satisfies
\begin{equation}
\label{eq:eta-quadratic-equation}
\norm{\measurement - \etah{} \linmap (\minvar) } = \eps .
\end{equation}
For \mmode{\minvar \in \hcone}, since \mmode{\ipxphih{} \geq 0}, we see that the quadratic equation \mmode{ \norm{\measurement - \theta \linmap (\minvar) }^2 = \eps^2 } has two positive real roots. Moreover, \mmode{\etah{}} is the smallest positive root of this quadratic equation, which gives
\begin{equation}
\label{eq:eta-alternate-value}
\etah{} = \frac{\ipxphih{} \; - \; \normphih{} \sqrt{ \eps^2 - \eh{}  } }{\normphih{}^2} \quad \text{ for every } \minvar \in \hcone . 
\end{equation}
}\end{remark}

\begin{lemma}
\label{lemma:eta-equals-sqrt-term}
Suppose, \mmode{\maxvar(\minvar)} is an optimal solution to the maximization problem \eqref{eq:etah-max-problem}, then it also satisfies
\begin{equation}
\label{eq:eta-equals-sqrt-term}
\etah{} \; = \; \lagrangianh{} \; = \; \sqtermh{} .  
\end{equation}
\end{lemma}
\begin{proof}[Lemma \ref{lemma:eta-equals-sqrt-term}]
First of all, we observe that the maximization problem \eqref{eq:etah-max-problem} admits an optimal solution if and only if it satisfies the first order optimality conditions:
\[
0 \; = \; \frac{\partial}{\partial \maxvar} \lagrangianh{} \; = \; \frac{\measurement - \frac{\eps}{\norm{\maxvarh{}}} \maxvarh{} }{\sqtermh{}} \; - \; \phih{} .
\]
By taking inner product throughout with \mmode{\maxvarh{}}, it is readily seen that
\[
\sqtermh{} \; = \; \inprod{\maxvarh{}}{\phih{}} .
\]
Thus, we have
\begin{align*}
\etah{} 
& = \lagrangianh{} \quad \text{from Lemma \ref{lemma:max-problem-optimal-value}}, \\
& = 2 \sqtermh{} - \inprod{\maxvarh{}}{\phih{}} \\
& = \sqtermh{} .    
\end{align*} 
\end{proof}

\begin{proof}[Proof of Proposition \ref{proposition:max-problem-solution}]
Lemma \ref{lemma:max-problem-unboundedness} and \ref{lemma:max-problem-optimal-value} together imply assertion (i) of the proposition. To complete the proof of the proposition, it now only remains to be shown that the maximization problem \eqref{eq:etah-max-problem} admits a unique optimal solution \mmode{\maxvarh{}} if and only if \mmode{\minvar \in \interior (\hcone)}.

Firstly, we observe that the maximization problem \eqref{eq:etah-max-problem} admits an optimal solution if and only if it satisfies the first order optimality conditions: \mmode{0 = \frac{\partial}{\partial \maxvar} \lagrangianh{} }, rearranging terms, we obtain
\begin{equation}
\label{eq:optimality-condition}
\frac{\eps}{\norm{\maxvarh{}}} \maxvarh{} \; = \; \measurement \; - \; \sqtermh{} \; \phih{} \; = \; \measurement \; - \; \etah{} \phih{},\footnote{Observe that by evaluating squared norm on both sides of \eqref{eq:optimality-condition} and using \eqref{eq:eta-equals-sqrt-term} also gives rise to the quadratic equation \mmode{\eps^2 = \norm{\xminusetaphih{}}^2} for \mmode{\etah{}}.}
\end{equation}
the last equality is due to \eqref{eq:eta-equals-sqrt-term}. Now, we observe that any \mmode{\maxvarh{}} that satisfies the implicit non-linear equation \eqref{eq:optimality-condition} must be of the form \mmode{\maxvarh{} = \scaling \big( \xminusetaphih{} \big) } for some \mmode{\scaling > 0}. The precise value of \mmode{r > 0} can be computed using \eqref{eq:eta-equals-sqrt-term}. We have
\[
\begin{aligned}
\etah{} & = \sqtermh{} \\
& = \sqrt{\scaling} \sqrt{ \inprod{ \xminusetaphih{} }{\measurement} - \eps \norm{\xminusetaphih{}} } \\
& = \sqrt{\scaling} \sqrt{ \norm{\xminusetaphih{}}^2 + \inprod{\xminusetaphih{}}{\etah{} \phih{}} - \eps^2 } \\
& = \sqrt{\scaling \etah{} } \sqrt{ \ipxphih{} - \etah{} \normphih{}^2 } ,
\end{aligned}
\]
from which it is easily picked that \mmode{\scaling \; = \; \ipxphih{} - \etah{} \normphih{}^2 \; = \; \normphih{} \sqrt{ \eps^2 - \eh{} } }. Now, \mmode{\scaling > 0} if and only if \mmode{\eh{} < \eps^2}, or equivalently, \mmode{\minvar \in \interior (\hcone)}. Thus, we finally conclude that the optimality condition \eqref{eq:optimality-condition} has a unique solution \mmode{\maxvarh{}} if and only if \mmode{\minvar \in \interior (\hcone)}, and is given by
\begin{equation*}
\maxvar (\minvar) = \frac{\etah{}}{\normphih{} \sqrt{\epsminuse{} }}  \big( \xminusetaphih{} \big) .
\end{equation*}
The proof of the proposition is complete.
\end{proof}

\begin{lemma}
\label{lemma:some-key-relations}
For every \mmode{\minvar \in \hcone}, the following relations hold
\begin{equation}
\label{eq:eta-phih-lower-bound}
\etah{} \normphih{} \; \geq \; \normx - \eps ,
\end{equation}
\begin{equation}
\label{eq:formaulas-xminusetaphi-equals}
\inprod{\xminusetaphih{}}{\phih{}} \; = \; \normphih{} \sqrt{\eps^2 - \eh{}}.
\end{equation}

\end{lemma}

\begin{proof}
For any \mmode{\minvar \in \hcone}, we see from \eqref{eq:etah-definition} that
\[
\begin{aligned}
\frac{1}{\etah{}} 
&= \frac{\normphih{}}{\xminuseps} \Big( \frac{\ipxphih{}}{\normphih{}} + \sqrt{\epsminuse{}} \Big) \\
&< \frac{\normphih{}}{\xminuseps} \big( \normx + \eps \big) \quad \text{due to C-S inequality, and } 0 \leq \eh{} , \\
&= \frac{\normphih{}}{\normx - \eps} ,
\end{aligned}
\]
On rearranging terms, the inequality \eqref{eq:eta-phih-lower-bound} is obtained at once. Similarly, rearranging terms in \eqref{eq:eta-quadratic-equation}, we see
\[
\normphih{} \sqrt{\eps^2 - \eh{}} \; = \; \ipxphih{} - \etah{} \normphih{}^2 ,
\]
Observing that \mmode{\ipxphih{} - \etah{} \normphih{}^2 = \; \inprod{\xminusetaphih{}}{\phih{}} }, \eqref{eq:formaulas-xminusetaphi-equals} follows immediately.
\end{proof}


\begin{lemma}[Derivative of \mmode{\maxvarh{}}]
\label{lemma:derivative-of-lambdah}
Let \mmode{\measurement, \linmap , \eps} be given such that Assumption \ref{assumption:lip-main} holds, then the mapping \mmode{\interior (\hcone) \ni \minvar \longmapsto \maxvarh{}} is continuously differentiable. Moreover, with the continuous maps \mmode{\interior(\hcone) \ni \minvar \longmapsto \big( r(\minvar), \; M(\minvar) \big) \in (0, +\infty) \times \R{\measdim \times \measdim} } defined as
\begin{equation}
\label{eq:Mh-definition}
\begin{cases}
\begin{aligned}
r(\minvar) &\define \frac{2\eps}{\normlambdah{}} + \frac{\xminuseps}{\eta^2(\minvar)}, \quad \text{and } \\
M(\minvar) &\define \frac{\normlambdah{}}{ \eps \etah{}} \Big( \eta^2(\minvar) \identity{\measdim} \; + \; r(h) \big( \maxvarh{} \maxvar\transp (\minvar) \big) - \big( \maxvarh{} x\transp + x\maxvar\transp (\minvar) \big)  \Big) ,
\end{aligned}
\end{cases}
\end{equation}
the derivative of \mmode{\minvar \longmapsto \maxvarh{}} is the linear map \mmode{\gradlambdah : \hilbert \longrightarrow \R{\measdim}} given by
\begin{equation}
\label{eq:derivative-of-lambdah}
\gradlambdah (v) \; = \; - M(\minvar) \cdot \linmap (v) \text{ for all } v \in \hilbert .
\end{equation}
\end{lemma}
\begin{proof}[Lemma \ref{lemma:derivative-of-lambdah}]
First, we rewrite \eqref{eq:lambdah-value} as
\[
\inprod{\xminusetaphih{}}{\phih{}} \maxvarh{} \; = \; \etah{} \big( \xminusetaphih{} \big),
\]
then by differentiating on both sides w.r.t. \mmode{\minvar}, we obtain an equation in the space of linear operators from \mmode{\hilbert} to \mmode{\R{\measdim}}. Evaluating the operators on the both sides of this equation at some \mmode{v \in \hilbert}, we get
\begin{equation}
\label{eq:gradlambdah-calculation-1}
\begin{aligned}
& \inprod{\xminusetaphih{}}{\phih{}} \gradlambdah (v) \; + \; \inprod{ \nabla \Big( \inprod{\xminusetaphih{}}{\phih{}} \Big) }{v} \maxvarh{} \\
& \quad \quad = \ \inprod{\etagrad{}}{v} \measurement \; - \; \eta^2 (\minvar) \linmap (v) \; - \; 2\etah{} \inprod{\etagrad{}}{v} \; \phih{} .
\end{aligned}
\end{equation}
On the one hand, we have
\begin{equation}
\label{eq:gradlambdah-calculation-2}
\begin{aligned}
& \inprod{\etagrad{}}{v} \measurement \; - \; \eta^2 (\minvar) \linmap (v) \; - \; 2\etah{} \inprod{\etagrad{}}{v} \; \phih{} \\
& \quad \quad = \ \inprod{\etagrad{}}{v} \big( \measurement - 2 \etah{} \phih{} \big) \ - \ \eta^2 (\minvar) \linmap (v) \\
& \quad \quad = \ - \inprod{\maxvarh{}}{\linmap(v)} \big( \measurement - 2 \etah{} \phih{} \big) \ - \ \eta^2 (\minvar) \linmap (v) \quad \text{ since } \etagrad{} = - \linadj (\maxvarh{})  \\
& \quad \quad = \ - \Big( \eta^2 (\minvar) \; \identity{\measdim} \ + \ \big( \measurement - 2 \etah{} \phih{} \big) \maxvar\transp (\minvar) \Big) \cdot \linmap (v) \\
& \quad \quad = \ - \Big( \eta^2 (\minvar) \; \identity{\measdim} \ + \ \big( - \measurement + \; \nicefrac{2\eps}{\normlambdah{}} \; \maxvarh{} \big) \maxvar\transp (\minvar) \Big) \cdot \linmap (v) \\
& \quad \quad = \ - \Big( \eta^2 (\minvar) \; \identity{\measdim} \ + \ \nicefrac{2\eps}{\normlambdah{}} \;\big( \maxvarh{} \maxvar\transp (\minvar) \big) \ - \ \big( \measurement \maxvar\transp (\minvar) \big) \Big) \cdot \linmap (v) .
\end{aligned}
\end{equation}
On the other hand, since
\[
\begin{aligned}
\nabla \Big( \inprod{\xminusetaphih{}}{\phih{}} \Big) \ &= \ \linadj (\measurement) \; - \; 2 \etah{} \linadj \big( \phih{} \big) \; - \; \normphih{}^2 \etagrad{} \\
& = \ \linadj \Big( \measurement - 2 \etah{} \phih{} + \normphih{}^2 \maxvarh{} \Big) ,
\end{aligned}
\]
we also have
\begin{equation}
\label{eq:gradlambdah-calculation-3}
\begin{aligned}
& \inprod{ \nabla \Big( \inprod{\xminusetaphih{}}{\phih{}} \Big) }{v} \maxvarh{} \\
& \quad \quad \quad \quad \quad = \ \inprod{\big( \measurement - 2 \etah{} \phih{} + \normphih{}^2 \maxvarh{} \big)}{ \linmap (v) } \maxvarh{} \\
& \quad \quad \quad \quad \quad = \ \Big( \maxvarh{} \big( \measurement - 2 \etah{} \phih{} + \normphih{}^2 \maxvarh{} \big)\transp \Big) \cdot \linmap(v) \\
& \quad \quad \quad \quad \quad = \ \Big( \maxvarh{} \big( - \measurement + \big( \nicefrac{2\eps}{\normlambdah{}} + \normphih{}^2 \big) \maxvarh{} \big)\transp \Big) \cdot \linmap(v) \quad \text{ from \eqref{eq:lambdah-value}} \\
& \quad \quad \quad \quad \quad = \ \Big( - \big( \maxvarh{} x\transp \big) \; + \; \big( \nicefrac{2\eps}{\normlambdah{}} + \normphih{}^2 \big) \big( \maxvarh{} \maxvar\transp (\minvar) \big) \Big) \cdot \linmap (v) .
\end{aligned}
\end{equation}
Collecting \eqref{eq:gradlambdah-calculation-2} and \eqref{eq:gradlambdah-calculation-3}, together with \mmode{\inprod{\xminusetaphih{}}{\phih{}} = \frac{\eps \etah{}}{\normlambdah{}} } (from \eqref{eq:lambdah-value}), \eqref{eq:gradlambdah-calculation-1} simplifies to reveal
\[
\begin{aligned}
\frac{\eps \etah{}}{\normlambdah{}} \gradlambdah (v) \ = \ &- \Big( \eta^2 (\minvar) \identity{\measdim} \; - \; \big( \measurement \maxvar\transp (\minvar) \; + \; \maxvarh{ \measurement\transp} \big) \Big) \cdot \linmap (v) \\
& \quad \quad - \Big( \nicefrac{4\eps}{\normlambdah{}} \; + \; \normphih{}^2 \Big) \big( \maxvarh{} \maxvar\transp (\minvar) \big) \cdot \linmap (v) .
\end{aligned}
\]
Finally, simplifying
\[
\begin{aligned}
\frac{2\eps}{\normlambdah{}} \; + \; \normphih{}^2 \ &= \ \frac{2 \inprod{\xminusetaphih{}}{\phih{}}}{\etah{}} \; + \; \normphih{}^2 \quad \text{from \eqref{eq:lambdah-value}} \\
&= \ \frac{2 \ipxphih{}}{\etah{}} \; - \; \normphih{}^2 \\
&= \ \frac{1}{\eta^2 (\minvar)} \Big( 2\inprod{\measurement}{\etah{} \phih{}} \; - \; \eta^2(\minvar) \normphih{}^2 \Big) \\
&= \ \frac{1}{\eta^2 (\minvar)} \Big( \normx^2 - \norm{\xminusetaphih{}}^2 \Big) \ = \ \frac{\xminuseps}{\eta^2 (\minvar)}.
\end{aligned}
\]
Putting everything together, the derivative \mmode{\gradlambdah (v)} is easily written in terms of the matrix \mmode{\Mh} given in \eqref{eq:Mh-definition} as
\[
\gradlambdah (v) \ = \ \Mh \cdot \linmap (v)
\]
Continuous differentiability of \mmode{\interior (\hcone) \ni \minvar \longmapsto \maxvarh{} \in \R{\measdim}} follows directly from continuity of the map \mmode{\interior (\hcone) \ni \minvar \longmapsto \Mh \in \R{\measdim \times \measdim}}, which is straight forward. The proof of the lemma is complete.
\end{proof}

\begin{proof}[Proof of Proposition \ref{prop:eta-derivatives}]
From assertion (i) of Proposition \ref{proposition:max-problem-solution}, it is inferred that for every \mmode{\minvar \in \hcone}, the value \mmode{\etah{}} is a point-wise maximum of the linear function \mmode{\lagrangian} (linear in \mmode{\minvar}). Thus, the mapping \mmode{ \eta : \hcone \longrightarrow [0,+\infty) } is convex.

From assertion (ii) of Proposition \ref{proposition:max-problem-solution}, it follows that the maximization problem \eqref{eq:etah-max-problem} admits a solution \mmode{\maxvarh{}} if and only if \mmode{\minvar \in \smallhcone}.Then, from Danskin's theorem \citep{bertsekas1971control}, we conclude that the function \mmode{\eta: \hcone \longrightarrow [0, +\infty)} is differentiable if and only if the maximizer \mmode{\maxvarh{}} in \eqref{eq:etah-max-problem} exists. Thus, \mmode{ \eta : \hcone \longrightarrow [0,+\infty) } is differentiable at every \mmode{\minvar \in \interior (\hcone)}, and the derivative is given by \mmode{\etagrad{} = - \linadj \big( \maxvarh{} \big)}. Substituting for \mmode{\maxvarh{}} from \eqref{eq:lambdah-value}, we immediately get \eqref{eq:eta-derivatives}.

Since \mmode{ \etagrad{} = - \linadj\big( \maxvarh{} \big) }, we realise that \mmode{\etah{}} is twice differentiable if and only if the mapping \mmode{\minvar \longmapsto \maxvarh{} } has a well-defined derivative \mmode{\gradlambdah}. In which case, the hessian is a linear operator \mmode{\hessianeta : \hilbert \longrightarrow \hilbert } given by
\[
\hessianeta (v) = - \linadj \circ \gradlambdah (v) \quad \text{ for all } v \in \hilbert.
\]
We know that the derivative \mmode{\gradlambdah} exists for every \mmode{\minvar \in \interior (\hcone)}, thus, \mmode{\eta(\cdot)} is twice differentiable everywhere on \mmode{\interior (\hcone) }. Substituting for \mmode{\gradlambdah} from \eqref{eq:derivative-of-lambdah}, we immediately get
\[
\hessianeta (v) \; = \; \left( \linadj \circ \Mh \circ \linmap \right) (v) \quad \text{ for all } v \in \hilbert ,
\]
where \mmode{\minvar \longmapsto \Mh} is a matrix valued map given in \eqref{eq:Mh-definition}. Moreover, continuity of the hessian i.e., continuity of the map \mmode{\interior (\hcone) \ni \minvar \longmapsto \hessianeta} follows directly from the continuity of \mmode{\interior (\hcone) \ni \minvar \longmapsto \gradlambdah}. The proof in now complete.
\end{proof}

\begin{lemma}[Smallest and largest eigenvalues of \mmode{\Mh}]
\label{lemma:min-max-eig-values-Mh}
For every \mmode{\minvar \in \interior (\hcone)}, consider \mmode{\Mh \in \R{\measdim \times \measdim}} as given in \eqref{eq:Mh-definition}. Then its minimum and maximum eigenvalues, denoted by \mmode{\mineig(\Mh)} and \mmode{\maxeig(\Mh)} respectively, are
\begin{equation}
\label{eq:min-max-eig-value-Mh}
\begin{cases}
\begin{aligned}
\mineig(\Mh) \ &= \ \frac{\xminuseps \normlambdah{}^3}{2 \eps \eta^3(\minvar) } \left( 1 \; - \; \sqrt{ 1 \; - \; \frac{8\eps \eta^6 (\minvar) }{\xminuseps^2 \normlambdah{}^3} } \right) \\
\maxeig(\Mh) \ &= \ \frac{\xminuseps \normlambdah{}^3}{2 \eps \eta^3(\minvar) } \left( 1 \; + \; \sqrt{ 1 \; - \; \frac{8\eps \eta^6 (\minvar) }{\xminuseps^2 \normlambdah{}^3} } \right) .
\end{aligned}
\end{cases}
\end{equation}
\end{lemma}

\begin{proof}[Lemma \ref{lemma:min-max-eig-values-Mh}]
Recall from \eqref{eq:Mh-definition} that
\[
\begin{cases}
\begin{aligned}
r(\minvar) \; &= \; \frac{2\eps}{\normlambdah{}} + \frac{\xminuseps}{\eta^2(\minvar)}, \quad \text{and } \\
M(\minvar) \; &= \; \frac{\normlambdah{}}{ \eps \etah{}} \Big( \eta^2(\minvar) \identity{\measdim} \; + \; r(h) \big( \maxvarh{} \maxvar\transp (\minvar) \big) - \big( \maxvarh{} x\transp + x\maxvar\transp (\minvar) \big)  \Big) .   
\end{aligned}
\end{cases}
\]
First, suppose that \mmode{\maxvarh{}} and \mmode{\measurement} are linearly independent. Then it is clear that the subspace \mmode{S \define \Span \{ \maxvarh{} , \measurement \} } is invariant under the linear transformation given by the matrix \mmode{\Mh}, and this linear transformation is identity on the orthogonal complement of \mmode{S}. Then it is also evident that the hessian has \mmode{\measdim - 2} eigenvalues equal to \mmode{\big( \nicefrac{1}{\eps} \big) \etah{} \normlambdah{}} and the two other distinct eigenvalues corresponding to the restriction of \mmode{\Mh} onto the \mmode{2}-dimensional subspace \mmode{S}.

Let \mmode{T} denote the \mmode{2\times 2} matrix representing the restriction of \mmode{\Mh} onto the subspace \mmode{S} for \mmode{\{ \maxvarh{} , \; \measurement \}} being chosen as a basis for \mmode{S}. In other words, it holds that \mmode{\Mh [\measurement \quad \maxvarh{}]  \; = \; [\measurement \quad \maxvarh{}] T}. Using the fact that \mmode{\eta^2 (\minvar) = \inprod{\maxvarh{}}{\measurement} - \eps \normlambdah{} } from \eqref{eq:eta-equals-sqrt-term}, it is easily verified that the matrix \mmode{T} simplifies to
\begin{equation}
T = \frac{\normlambdah{}}{\eps \etah{} }
\pmatr{-\eps \normlambdah{} & - \normlambdah{}^2 \\
\rh \inprod{\maxvarh{}}{\measurement} - \norm{x}^2 \quad & \rh \normlambdah{}^2 -\eps \normlambdah{} } .
\end{equation}
Furthermore, substituting \mmode{\rh}, it is also verified that \mmode{\trace (T) = \frac{\normlambdah{}^2}{\eta^2 (\minvar)} \xminuseps } and \mmode{\det (T) = 2\eps \eta^2 (\minvar) \normlambdah{}}. Now, it is easily verified that the two eigenvalues of \mmode{T} are precisely equal to \mmode{\{ \mineig (\Mh) , \maxeig (\Mh) \}}.

Since apart from \mmode{\{ \mineig (\Mh) , \maxeig (\Mh) \}}, the rest of the eigenvalues of \mmode{\Mh} are equal to \mmode{\frac{1}{\eps} \etah{} \normlambdah{}}, it remains to be shown that \mmode{\mineig (\Mh) \leq \frac{1}{\eps} \etah{} \normlambdah{} \leq \maxeig(\Mh)}; which we do so by producing \mmode{u_1, u_2 \in S} such that 
\begin{equation}
\label{eq:eigenvalue-inequalities-Mh}
\mineig (\Mh) \; \leq \; \frac{ \inprod{u_1}{\Mh u_1} }{\norm{u_1}^2} \; \leq \; \frac{1}{\eps} \etah{} \normlambdah{} \; \leq \; \frac{ \inprod{u_2}{\Mh u_2} }{\norm{u_2}^2} \; \leq \; \maxeig(\Mh) .    
\end{equation}
Observe that the inequalities \mmode{ \mineig (\Mh) \; \leq \; \frac{ \inprod{u_1}{\Mh u_1} }{\norm{u_1}^2} }, and \mmode{ \frac{ \inprod{u_2}{\Mh u_2} }{\norm{u_2}^2} \; \leq \; \maxeig (\Mh) } readily hold for any \mmode{u_1, u_2 \in S} since \mmode{\mineig (\Mh) , \maxeig (\Mh)} are the two eigenvalues of \mmode{\Mh} when restricted to the subspace \mmode{S}. To obtain the rest of the inequalities in \eqref{eq:eigenvalue-inequalities-Mh}, consider
\[
u_1 \; = \; \measurement \; + \; \frac{\normx^2 - \; \rh \inprod{\maxvarh{}}{\measurement}}{\rh \normlambdah{}^2 - \; \inprod{\maxvarh{}}{\measurement}} \maxvarh{} \quad \text{and} \quad u_2 \; = \; \measurement \; - \; \frac{\normx^2}{\inprod{\maxvarh{}}{\measurement}} \maxvarh{} .
\]
It is easily verified that \mmode{\inprod{u_1}{\rh \maxvarh{} - \measurement} \; = \; 0} and \mmode{ \inprod{u_2}{\measurement} \; = \; 0 }. Moreover, rewriting \mmode{\Mh} by completing squares as
\[
M(\minvar) = \frac{\normlambdah{}}{ \eps \rh \etah{}} \Big( \eta^2(\minvar) \rh \identity{\measdim} \; + \; \big( \rh \maxvarh{} - \measurement \big)\big( \rh \maxvarh{} - \measurement \big)\transp - \; \measurement \measurement\transp \Big) ,
\]
it is also easily verified that the inequalities
\[
\begin{cases}
\begin{aligned}
\frac{\inprod{u_1}{\Mh u_1}}{\norm{u_1}^2 } \; &= \; \frac{\normlambdah{}}{ \eps \rh \etah{}} \Big( \eta^2 (\minvar) \; - \; \frac{\abs{\inprod{u_1}{\measurement}}^2}{\norm{u_1}^2} \Big) \; &&\leq \; \frac{1}{\eps} \rh \normlambdah{} , \\
\frac{\inprod{u_2}{\Mh u_2}}{\norm{u_2}^2 } \; &= \; \frac{\normlambdah{}}{ \eps \rh \etah{}} \Big( \eta^2 (\minvar) \; + \; \frac{ \abs{ \inprod{u_2}{\rh \maxvarh{} - \measurement} }^2 }{\norm{u_2}^2} \Big) \; &&\geq \; \frac{1}{\eps} \rh \normlambdah{} . 
\end{aligned}
\end{cases}
\]
Thus, the inequalities \eqref{eq:eigenvalue-inequalities-Mh} are obtained at once.

To complete the proof for the case when \mmode{\maxvarh{}} and \mmode{ \measurement} are linearly dependent, we first see that the \mmode{\interior(\hcone) \ni \minvar \longmapsto \big( \mineig(\Mh) , \maxeig (\Mh) \big)} is continuous. Secondly, since the mapping \mmode{\interior(\hcone) \ni \minvar \longmapsto \Mh } is also continuous, and the eigenvalues of a matrix vary continuously, these two limits must be the same. The proof is now complete.
\end{proof}

\begin{proof}[Proposition \ref{proposition:eta-convex-regularity}]
For any \mmode{\epsbar \in (0, \eps)} and \mmode{\etahat > \cost\opt}, we know that the set \mmode{\sethsmall \subset \interior (\hcone)}. Consequently, it follows from Proposition \ref{prop:eta-derivatives} that \mmode{\eta : \sethsmall \longrightarrow [0, +\infty) } is twice continuously differentiable. To establish the required smoothness, and strong convexity assertions of the proposition, we first obtain uniform upper (lower) bound on the maximum (minimum) eigenvalue of the Hessian \mmode{\hessianeta}. To this end, for every \mmode{v \in \hilbert}, since \mmode{\inprod{v}{\hessianeta (v)} \; = \; \inprod{\linmap (v)}{\Mh \linmap(v)} }, we see that
\begin{equation}
\label{eq:hessian-inprod-min-max-bound-1}
\mineig (\Mh) \mineig (\linadj \circ \linmap) \ \leq \ \frac{\inprod{v}{\hessianeta (v)}}{\norm{v}^2} \ \leq \ \maxeig (\Mh) \maxeig (\linadj \circ \linmap) \quad \text{ for all } v \in \hilbert .
\end{equation}
The quantities \mmode{\mineig (\linadj \circ \linmap)} and \mmode{\maxeig (\linadj \circ \linmap)} are the minimum and maximum eigenvalues of the linear operator \mmode{\linadj \circ \linmap : \hilbert \longrightarrow \hilbert} respectively. Denoting \mmode{\maxeig \hessianeta} and \mmode{\maxeig \hessianeta} to be the the maximum and minimum eigenvalues of the hessian respectively, it follows from \eqref{eq:hessian-inprod-min-max-bound-1} that
\begin{equation}
\label{eq:hessian-inprod-min-max-bound-2}
\mineig (\Mh) \mineig (\linadj \circ \linmap) \ \leq \ \mineig \hessianeta \ \leq \ \maxeig \hessianeta \ \leq \ \maxeig (\Mh) \maxeig (\linadj \circ \linmap) .
\end{equation}

\noindent \textbf{Uniform upper bound for \mmode{\maxeig (\Mh)}.} For every \mmode{\minvar \in \sethsmall \subset \seth}, we have the inequality
\begin{equation}
\label{eq:uniform-upper-bound-max-eig-Mh-ineq-1}
\begin{aligned}
\ \frac{\norm{\maxvarh{}}}{\etah{}}  \ 
&= \ \frac{ \norm{\xminusetaphih{}} }{ \normphih{} \sqrt{\epsminuse{}} } \; = \; \frac{\eps}{\normphih{} \sqrt{\epsminuse{}} }, \ \text{from \eqref{eq:eta-quadratic-equation}}, \\
&< \  \frac{\eps}{\normphih{} \sqrt{\epsminusebar}}, \ \text{since \mmode{ \eh{} < \epsbar^2} for \mmode{\minvar \in \sethsmall}}. 
\end{aligned}
\end{equation}
On the other hand, since \mmode{\sethsmall \subset \hcone} we conclude from \eqref{eq:eta-phih-lower-bound} that the upper bound \mmode{ \frac{1}{\normphih{}} < \frac{\etah{}}{\normx - \eps}} holds for every \mmode{\minvar \in \sethsmall}. Putting together in \eqref{eq:uniform-upper-bound-max-eig-Mh-ineq-1}, we have
\[
\frac{\norm{\maxvarh{}}}{\etah{}} \ < \ \frac{\etah{}}{\normx - \eps} \, \frac{\eps}{\sqrt{\epsminusebar}} \ < \ \frac{\etahat}{\normx - \eps} \frac{\eps}{\sqrt{\epsminusebar}} \ .
\]
Thus, from \eqref{eq:min-max-eig-value-Mh}, we have
\begin{equation}
\label{eq:uniform-upper-bound-maxeig-Mh}
\maxeig(\Mh) \ \leq \ \frac{1}{ \eps  }\xminuseps  \left( \frac{\normlambdah{}}{\etah{}} \right)^3 \ < \ \frac{\eps^2 \big( \normx + \eps \big)}{\big( \normx - \eps \big)^2} \left( \frac{\etahat}{\sqrt{\epsminusebar}} \right)^3 .
\end{equation}

\noindent \textbf{Uniform lower bound for \mmode{\mineig (\Mh)}.} We know that \mmode{\sqrt{1 - \theta^2} < 1 - \frac{\theta^2}{2} } for every \mmode{\theta \in [0,1] }. Using this inequality in \eqref{eq:min-max-eig-value-Mh} for \mmode{\mineig (\Mh)} and simplifying, we see that
\begin{equation}
\label{eq:uniform-lower-bound-mineig-Mh}
\mineig (\Mh) \ \geq \ \frac{2 \eta^3 (\minvar)}{\xminuseps} \ \geq \ \frac{2 {\cost\opt}^3 }{\xminuseps} \geq \ \frac{2 {\etabar}^3 }{\xminuseps} ,
\end{equation}
for every \mmode{\etabar \in (0 , \cost\opt]}. Collecting \eqref{eq:uniform-upper-bound-maxeig-Mh} and \eqref{eq:uniform-lower-bound-mineig-Mh}, we see that the minimum and maximum eigenvalues of the hessian are uniformly bounded over \mmode{\sethsmall}, and the bounds are
\begin{equation}
\label{eq:uniform-bounds-eta-hessian-eigen-values}
\begin{cases}
\begin{aligned}
\maxeig \hessianeta \ &< \ \frac{\eps^2 {\etahat}^3}{\big(\epsminusebar\big)^{\nicefrac{3}{2}}} \frac{ \big( \normx + \eps \big)}{\big( \normx - \eps \big)^2} \; \maxeig (\linadj \circ \linmap) \ &&\eqqcolon \ \smooth \duet , \\
\mineig \hessianeta \ &\geq \ \frac{2 {\etabar}^3 }{\xminuseps} \mineig (\linadj \circ \linmap) \ &&\eqqcolon \ \strongconv (\etabar) .
\end{aligned}
\end{cases}
\end{equation}
Finally, \mmode{\eta : \sethsmall \longrightarrow [0, +\infty)} is twice continuously differentiable with the maximum eigenvalue of the hessian being uniformly bounded above by \mmode{\smooth \duet}. It then follows that \mmode{\eta : \sethsmall \longrightarrow [0, +\infty)} is \mmode{\smooth \duet}-smooth in the sense of \eqref{eq:def-smoothness}. Moreover, if \mmode{\linmap} is invertible in addition, then the minimum eigenvalue of the hessian is uniformly bounded below by \mmode{\strongconv > 0}. Consequently, the mapping \mmode{\eta : \sethsmall \longrightarrow [0, +\infty)} is \mmode{\strongconv}-strongly convex in the sense of \eqref{eq:def-strong-convexity}. The proof of the proposition is now complete.
\end{proof}


\subsection{Proofs for reformulation as a smooth minimization problem}
\begin{lemma}[Non-smooth reformulation]
\label{lemma:non-smooth-reformulation}
Consider the LIP \eqref{eq:lip-main} under the setting of Assumption \ref{assumption:lip-main}, then the LIP \eqref{eq:lip-main} is equivalent to the minimization problem
\begin{equation}
\label{eq:LIP-non-smooth-reformulation}
\begin{cases}
\min\limits_{\minvar \, \in \, \costball \cap \hcone} \ \etah{} \quad .
\end{cases}
\end{equation}
In other words, the optimal value of \eqref{eq:LIP-smooth-problem} is equal to \mmode{\cost\opt} and \mmode{\minvar\opt} is a solution to \eqref{eq:LIP-smooth-problem} if and only if \mmode{\cost\opt \minvar\opt } is an optimal solution to \eqref{eq:lip-main}.
\end{lemma}
\begin{proof}[Lemma \ref{lemma:non-smooth-reformulation}]
Recall that \mmode{\Lambda = \{ \maxvar \in \R{\measdim} : \inprod{\maxvar}{\measurement} - \eps \norm{\maxvar} > 0 \}}, \mmode{\costball = \{ \minvar \in \hilbert : \cost (\hilbert) \leq 1 \}}, and \mmode{\lagrangian = 2\sqterm - \inprod{\maxvar}{\phih{}} }. The original LIP \eqref{eq:lip-main} was reformulated as the min-max problem. By considering \mmode{r = 2, q = 0.5, \delta = 0} in \citep[Theorem 10]{Sheriff2020NovelProblems} we see that the min-max problem
\begin{equation}
\label{eq:lip-min-max}
\begin{cases}
\min\limits_{\minvar \; \in \; \costball} \; \sup\limits_{\maxvar \; \in \; \Lambda} \ \lagrangian ,
\end{cases}
\end{equation}
is equivalent to the LIP \eqref{eq:lip-main} with the optimal value of the min-max problem equal to \mmode{\cost\opt}. Moreover, from \citep[Theorem 10, assertion (ii)-a]{Sheriff2020NovelProblems}, it also follows that \mmode{ \minvar\opt \in \argmin\limits_{\minvar \; \in \; \costball} \; \left\{ \sup\limits_{\maxvar \; \in \; \Lambda} \ \lagrangian \right\} } if and only if \mmode{\cost\opt \minvar\opt} is an optimal solution to the LIP \eqref{eq:lip-main}. Solving for the maximization problem over \mmode{\maxvar} in the min-max problem \eqref{eq:lip-min-max}, in view of Proposition \ref{proposition:max-problem-solution} we know that the maximum over \mmode{\maxvar} is equal to \mmode{\etah{}} whenevr it is finite. Therefore, we get
\[
\minvar\opt \in \argmin_{\minvar \in \costball \cap \hcone} \ \etah{},
\]
if and only if \mmode{\cost\opt \minvar\opt} is an optimal solution to the LIP \eqref{eq:lip-main}. The proof is now complete.
\end{proof}

\begin{proof}[Theorem \ref{theorem:smooth-reformulation}]
Under the setting of Assumption \ref{assumption:lip-main} we have \mmode{\nbhood{\measurement}{\eps} \cap \image (\linmap) \neq \emptyset }. Thus, it follows from \citep[Proposition 31-(ii)]{Sheriff2020NovelProblems} and consequently, from \citep[Theorem 10-(ii)-b]{Sheriff2020NovelProblems}, that the min-max problem
\[
\begin{cases}
\min\limits_{\minvar \in \costball} \; \sup\limits_{\maxvar \in \setlambda } \quad 2 \sqterm \; - \; \inprod{\maxvar}{\linmap (\minvar)} 
\end{cases}
\]
admits a saddle point solution. Moreover, every saddle point \mmode{(\minvar\opt , \maxvar\opt) \in \costball \times \setlambda } is such that \mmode{\minvar\opt = \big( \nicefrac{1}{\cost\opt} \big) \minvarf\opt } where \mmode{\minvarf\opt} is any optimal solution to the LIP \eqref{eq:lip-main}, and \mmode{\maxvar\opt} is unique that satisfies
\[
\maxvar\opt = \argmax_{\maxvar \in \setlambda} \quad 2\sqterm \; - \; \inprod{\maxvar}{\linmap (\minvar\opt)} .
\]
In view of Proposition \ref{proposition:max-problem-solution}-(ii), we conclude that \mmode{\minvar\opt \in \interior (\hcone)}. Thus, \mmode{e(\minvarf\opt) = e(\minvar\opt) < \eps^2}, this establishes assertion (i) of the lemma.

To prove the rest of the theorem, consider any \mmode{\epsbar, \etahat > 0} such that \mmode{e(\minvarf\opt) \leq \epsbar^2 < \eps^2} and \mmode{\cost\opt \leq \etahat}. Then for any \mmode{\minvar\opt \in \ \argmin\limits_{\minvar \, \in \, \costball \cap \hcone} \ \etah{} }, we conclude from Lemma \ref{lemma:non-smooth-reformulation} that \mmode{\cost\opt \minvar\opt} is an optimal solution to the LIP \eqref{eq:lip-main}. Consequently, assertion (i) of the proposition then implies that \mmode{ e(\minvar\opt) = e(\cost\opt \minvar\opt) \leq \epsbar^2 }. Thus, we have \mmode{\minvar\opt \in \smallhcone}. Moreover, from Lemma \ref{lemma:non-smooth-reformulation} it is also immediate that \mmode{\eta (\minvar\opt) = \cost\opt \leq \etahat}. Thus, \mmode{\minvar\opt \in \sethsmall}, and we have the inclusion
\[
\argmin\limits_{\minvar \, \in \, \costball \cap \hcone} \ \etah{} \ \subset \ \sethsmall .
\]
Since \mmode{\sethsmall \subset \costball \cap \hcone} to begin with, we conclude
\[
\argmin\limits_{\minvar \, \in \, \costball \cap \hcone} \ \etah{} \ = \argmin\limits_{\minvar \; \in \; \sethsmall} \ \etah{} .
\]
Now assertion (ii) of the theorem follows immediately as a consequence of Lemma \ref{lemma:non-smooth-reformulation}.
\end{proof}


\subsection{Proofs for reformulation as a strongly-convex min-max problem}
\begin{proof}[Lemma \ref{lemma:ACP-min-max}]
Recall that \mmode{ \setlambda \ni \maxvar \longmapsto \sqfunc = \sqterm }, then denoting \mmode{\setlambda \define \{ \maxvar \in \R{\measdim} : \inprod{\maxvar}{\measurement} - \eps \normlambda > 0 \} }, it is known from \citep{Sheriff2020NovelProblems} that the LIP \eqref{eq:lip-main} is equivalent to the min-max problem
\begin{equation}
\label{eq:min-max-jmlr}
\begin{cases}
\min\limits_{\minvar \in \costball} \; \sup\limits_{\maxvar \in \setlambda } \quad \lagrangian \; = \; 2 \sqfunc \; - \; \inprod{\maxvar}{\linmap (\minvar)} .
\end{cases}
\end{equation}
In particular, under the setting of Assumption \ref{assumption:lip-main}, it follows that the min-max problem \eqref{eq:min-max-jmlr} admits a saddle point solution. It follows from \citep[Theorem 10]{Sheriff2020NovelProblems}) that a pair \mmode{(\minvar\opt , \maxvar\opt)} is a saddle point of \eqref{eq:min-max-jmlr} if and only if \mmode{ \cost\opt \minvar\opt } is an optimal solution to the LIP \eqref{eq:lip-main}, and \mmode{\maxvar\opt = \maxvarh{\opt}} in view of Lemma \ref{proposition:max-problem-solution}.\footnote{The inclusion \mmode{\maxvar\opt \in \cost\opt \Lambda } provided in \citep[(44), Theorem 10]{Sheriff2020NovelProblems} turns out to be same as the condition \mmode{\maxvar\opt = \maxvarh{\opt}} under the setting of Assumption \ref{assumption:lip-main} for LIP \eqref{eq:lip-main}. This can be formally established by observing from \citep[Proposition 31]{Sheriff2020NovelProblems} that \mmode{\maxvar\opt = \cost\opt \frac{\measurement - \cost\opt \linmap(\minvar\opt)}{\pnorm{\measurement - \cost\opt \linmap(\minvar\opt)}{\linmap}'} }, and then, from \citep[Lemma 33]{Sheriff2020NovelProblems} we also have
\[
\begin{aligned}
\pnorm{\measurement - \cost\opt \linmap(\minvar\opt)}{\linmap}' \; &= \; \max_{\minvar \in \costball} \; \inprod{\measurement - \cost\opt \linmap(\minvar\opt)}{\linmap(\minvar)} \; = \; \inprod{\measurement - \cost\opt \linmap(\minvar\opt)}{\linmap(\minvar\opt)} \\
&= \; \norm{\linmap(\minvar\opt) } \sqrt{ \eps^2 - e(\minvar\opt) } .
\end{aligned}
\]
}

We prove the lemma by establishing that every saddle point solution to the min-max problem \eqref{eq:min-max-jmlr} is indeed a saddle point solution to the min-max problem \eqref{eq:ACP-min-max-problem} as well. We observe that the only difference between the min-max problems \eqref{eq:ACP-min-max-problem} and \eqref{eq:min-max-jmlr} is in their respective feasible sets \mmode{\setlambdasmall} and \mmode{\setlambda} for the variable \mmode{\maxvar}. Moreover, since \mmode{\setlambdasmall \subset \setlambda}, it suffices to show that for every saddle point \mmode{(\minvar\opt , \maxvar\opt)} of \eqref{eq:min-max-jmlr}, the inclusion \mmode{\maxvar\opt \in \setlambdasmall} also holds. To establish this inclusion, we first recall from \eqref{eq:eta-equals-sqrt-term} that 
\[
l(\maxvar\opt) = l (\maxvarh{\opt}) = \eta (\minvar\opt) = \cost\opt \; \geq \; \etabar .
\]
Secondly, using \eqref{eq:lambdah-value} we also have
\[
\begin{aligned}
\norm{\maxvar (\minvar\opt)} \;
&= \; \frac{\eps \; \eta (\minvar\opt)}{\norm{ \linmap (\minvar\opt) } \sqrt{\eps^2 - e(\minvar\opt)} } \; \leq \; \frac{\eps \; \eta (\minvar\opt)}{\norm{ \linmap (\minvar\opt) } \sqrt{\epsminusebar}} \quad \text{since } e(\minvar\opt) \in (\epsbar^2 , \eps^2) , \\
& \leq \; \frac{\eps \; \eta^2 (\minvar\opt)}{ \big( \normx - \eps \big) \sqrt{\epsminusebar} } \quad \text{ from \eqref{eq:eta-phih-lower-bound}} , \\
& \leq \; \frac{\eps \; \etabar^2 }{ \big( \normx - \eps \big) \sqrt{\epsminusebar}} \; = \; \boundlambda .
\end{aligned}
\]
Thus, \mmode{\maxvar\opt \in \setlambdasmall} and the lemma holds.
\end{proof}

\begin{lemma}
\label{lemma:hessian-dual-function-min-max-eigenvalues}
Consider \mmode{\measurement \in \R{\measdim}} and \mmode{\eps > 0} such that \mmode{\normx > \eps }. Then the following assertions hold with regards to the mapping \mmode{ \setlambda \ni \maxvar \longmapsto \sqfunc \define \sqterm }. 
\begin{enumerate}[leftmargin = *, label = \rm{(\roman*)}]
\item the mapping \mmode{ \setlambda \ni \maxvar \longmapsto \sqfunc } is twice continuously differentiable and its hessian \mmode{\hessian} evaluated at \mmode{\maxvar \in \setlambda} is given by
\begin{equation}
\label{eq:hessian}
\hessian = \frac{-\eps}{2\sqfunc \normlambda } \left( \identity{\measdim} \; - \; \frac{1}{\normlambda^2} \maxvar\maxvar\transp \right) \; - \; \frac{1}{4 (\sqfunc)^3 } \left( \xminusepslambda \right)\left( \xminusepslambda \right)\transp .
\end{equation}

\item The smallest and largest absolute values of the eigenvalues of \mmode{\hessian} denoted respectively by \mmode{\mineig (\hessian)} and \mmode{\maxeig (\hessian)}, are given by
\begin{equation}
\label{eq:min-max-eig-value-Hlambda}
\begin{cases}
\begin{aligned}
\mineig \ &= \ \frac{\xminuseps}{8 ( \sqfunc)^3 } \left( 1 \; - \; \sqrt{ 1 \; - \; \frac{8\eps (\sqfunc)^6 }{\xminuseps^2 \norm{\maxvar}^3} } \right) \\
\maxeig \ &= \ \frac{\xminuseps}{8 ( \sqfunc)^3 } \left( 1 \; + \; \sqrt{ 1 \; - \; \frac{8\eps (\sqfunc)^6 }{\xminuseps^2 \norm{\maxvar}^3} } \right) .
\end{aligned}
\end{cases}
\end{equation}
\end{enumerate}
\end{lemma}

\begin{proof}[Lemma \ref{lemma:hessian-dual-function-min-max-eigenvalues}]
First of all, we observe that since \mmode{ \maxvar \longmapsto \sqfunc} is differentiable everywhere on \mmode{\setlambda}, and the gradients are given by \mmode{\nabla \sqfunc = \frac{1}{2\sqfunc} \Big( \xminusepslambda \Big)}. Differentiating again w.r.t. \mmode{\maxvar}, we easily verify that the hessian is indeed as given by \eqref{eq:hessian}.
First, suppose that \mmode{\maxvar} and \mmode{\measurement} are linearly independent, observe that the subspace \mmode{S \define \Span \{ \maxvar , \xminusepslambda \} } is invariant under the linear transformation given by the hessian matrix \mmode{\hessian}, and it is identity on the orthogonal complement of \mmode{S}. Then it is evident that the hessian has \mmode{\measdim - 2} eigenvalues equal to \mmode{\nicefrac{-\eps}{\sqfunc \normlambda }} and the two other distinct eigenvalues corresponding to the restriction of \mmode{\hessian} onto \mmode{S}. Selecting \mmode{\{ \maxvar , \; \xminusepslambda \}} as a basis for \mmode{S}, the linear mapping of the hessian is given by the matrix
\begin{equation}
T = \pmatr{0 & \frac{\eps \sqfunc}{2\normlambda^3} \\
\frac{-1}{4\sqfunc} & \frac{- \xminuseps}{4 (\sqfunc)^3}} .
\end{equation}
It is a straightforward exercise to verify that \mmode{-\mineig} and \mmode{-\maxeig} are indeed the two distinct eigenvalues of \mmode{T} and consequently, the remaining two eigenvalues of the hessian \mmode{\hessian}. Since the rest of the eigenvalues are \mmode{\nicefrac{-\eps}{\sqfunc \normlambda }}, it remains to be shown that \mmode{\mineig \leq \nicefrac{\eps}{2\sqfunc\normlambda} \leq \maxeig}. We establish it by producing \mmode{u_1, u_2 \in S} such that 
\begin{equation}
\label{eq:eigenvalue-inequalities}
\mineig \; \leq \; \frac{\abs{ \inprod{u_1}{\hessian u_1} }}{\norm{u_1}^2} \; \leq \; \frac{\eps}{2\sqfunc\normlambda} \; \leq \; \frac{\abs{ \inprod{u_2}{\hessian u_2} }}{\norm{u_2}^2} \; \leq \; \maxeig .    
\end{equation}
Observe that the inequalities \mmode{ \mineig \; \leq \; \frac{\abs{ \inprod{u_1}{\hessian u_1} }}{\norm{u_1}^2} }, and \mmode{ \frac{\abs{ \inprod{u_2}{\hessian u_2} }}{\norm{u_2}^2} \; \leq \; \maxeig } readily hold for any \mmode{u_1, u_2 \in S} since \mmode{-\mineig , -\maxeig} are the two eigenvalues of \mmode{\hessian} when restricted to the subspace \mmode{S}. Considering \mmode{u_1 = (\sqfunc)^2 \measurement + \Big( \normx^2 - \, \frac{\eps \inprod{\maxvar}{\measurement}}{\normlambda} \Big) \maxvar } and \mmode{u_2 = \maxvar - \frac{\normlambda^2}{\inprod{\maxvar}{\measurement}} }, it is easily verified that \mmode{\inprod{ \xminusepslambda}{u_1} = 0}, and \mmode{\inprod{\maxvar}{u_2} = 0}. Moreover, we also get the inequalities
\[
\begin{cases}
\begin{aligned}
\inprod{u_1}{\hessian u_1} \; &= \; \frac{-\eps}{2\sqfunc \normlambda } \norm{u_1}^2 + \; \frac{\eps}{2\sqfunc \normlambda } \frac{\abs{\inprod{\maxvar}{u_1}}^2}{\normlambda^2}  &&\geq \frac{-\eps}{2\sqfunc \normlambda } \norm{u_1}^2 , \\
\inprod{u_2}{\hessian u_2} \;& = \; \frac{-\eps}{2\sqfunc \normlambda } \norm{u_2}^2 - \; \frac{1}{4(\sqfunc)^3} \abs{ \inprod{\xminusepslambda}{u_2} }^2 &&\leq \frac{-\eps}{2\sqfunc \normlambda } \norm{u_2}^2 .
\end{aligned}
\end{cases}
\]
Since the hessian \mmode{\hessian} is negative semidefinite, the inequalities \eqref{eq:eigenvalue-inequalities} are obtained at once.

To complete the proof for the case when \mmode{\lambda} and \mmode{ \measurement} are linearly dependent, we first see that the expressions in \eqref{eq:min-max-eig-value-Hlambda} are continuous w.r.t. \mmode{\maxvar}. Also, it is evident that the mapping \mmode{\setlambda \ni \maxvar \longmapsto \hessian } is continuous. Since the eigenvalues of a matrix vary continuously, these two limits must be the same. The proof is now complete.
\end{proof}

\begin{proof}[Lemma \ref{lemma:convex-regularity-of-dual-function}]
We begin by first establishing that \mmode{\mineig(\hessian)} and \mmode{\maxeig (\hessian)} as given in \eqref{lemma:hessian-dual-function-min-max-eigenvalues} satisfy the inequalities 
\begin{equation}
\label{eq:strong-conv-smoothness-inequalities}
\strongconv' \; \leq \; 2\mineig (\hessian) \; \leq \; 2\maxeig (\hessian) \; \leq \; \smooth' \quad \text{for every } \maxvar \in \setlambdasmall .
\end{equation}
Since the mapping \mmode{\setlambda \ni \maxvar \longmapsto \hessian} is concave, all the eigenvalues of the hessian \mmode{\hessian} are non-positive (more importantly, real-valued). Thus, \mmode{ \frac{8\eps (\sqfunc)^6 }{\xminuseps^2 \norm{\maxvar}^3} \; \leq \; 1 } since the square root term in \eqref{lemma:hessian-dual-function-min-max-eigenvalues} must be real valued. To prove the lower bound for \mmode{\mineig (\hessian)} in \eqref{eq:strong-conv-smoothness-inequalities}, we use the inequality that \mmode{\sqrt{1 - \theta^2} < 1 - \frac{\theta^2}{2} } for every \mmode{\theta \in [0,1] }. Thereby,
\[
\begin{aligned}
\mineig (\hessian) \ &> \ \frac{\xminuseps}{8 ( \sqfunc)^3 } \left( \frac{4\eps (\sqfunc)^6 }{\xminuseps^2 \norm{\maxvar}^3} \right) \\
& = \ \frac{\eps }{2 \xminuseps } \left( \frac{(\sqfunc)}{\norm{\maxvar}} \right)^3 \ \geq \ \frac{\eps}{2\xminuseps } \left( \frac{\etabar}{\boundlambda} \right)^3 \\
& = \ \big( \nicefrac{1}{2} \big) \strongconv' \ \text{ for all } \maxvar \in \setlambdasmall .
\end{aligned}
\]
For \mmode{\maxeig (\hessian)}, using the inequality \mmode{\sqrt{1 - \theta^2} < 1 } for \mmode{\theta \in [0,1]}, we immediately get
\[
\maxeig (\hessian) \leq \frac{\xminuseps}{8 ( \sqfunc)^3 } \; 2 \; \leq \; \frac{\xminuseps}{4 \etabar^3 } \; = \; \big( \nicefrac{1}{2} \big) \smooth' \quad \text{for all } \maxvar \in \setlambdasmall .
\]

Since \mmode{\setlambdasmall \subset \setlambda }, for every \mmode{\etabar \leq \cost\opt} and \mmode{\boundlambda > 0}, it follows from assertion (i) of Lemma \ref{lemma:hessian-dual-function-min-max-eigenvalues} that the mapping \mmode{\setlambdasmall \ni \maxvar \longmapsto -2\sqfunc } is also twice continuously differentiable.  with the Hessian evaluated at \mmode{\maxvar} being \mmode{-2\hessian}. Moreover, the smallest and largest eigenvalues of this hessian are \mmode{2\mineig (\hessian)} and \mmode{2 \maxeig (\hessian)} respectively. In view of the inequalities \eqref{eq:strong-conv-smoothness-inequalities}, we see that the minimum eigenvalue of the hessian of the map \mmode{\setlambdasmall \ni \maxvar \longmapsto -2\sqfunc} is bounded below by \mmode{\strongconv' \duetlambda} (and the maximum eigenvalue is bounded above by \mmode{\smooth' (\etabar)}),  uniformly over \mmode{\maxvar \in \setlambdasmall}. Thus, the mapping \mmode{\setlambdasmall \ni \maxvar \longmapsto -2\sqfunc} is \mmode{\strongconv'}-strongly convex and \mmode{\smooth'}-smooth. This completes the proof of the lemma.
\end{proof}

\subsection{Proofs for step-size selection}
\begin{proof}[Proposition \ref{proposition:gamma_opt}]
For a given \mmode{\minvar , \direction}, we first observe that the mapping \mmode{[0,1] \ni \stepsize \longmapsto \eta(\minvar + \stepsize \direction)} is convex since the mapping \mmode{\minvar \longmapsto \etah{}} is convex. Consequently, the first order optimality conditions for \eqref{eq:step-size-selection-proposition} are necessary and sufficient. Now, denoting \mmode{\eta_{\stepsize} \define \eta(\minvar + \stepsize \direction) }, we have \mmode{\frac{\partial \eta_{\stepsize}}{\partial \stepsize} = \inprod{\etagrad{+\stepsize \direction} }{\direction} }, and the first order optimality conditions read
\begin{enumerate}[leftmargin = *]
    \item If \mmode{\inprod{\etagrad{}}{\direction} = \ \frac{\partial \eta_{\stepsize}}{\partial \stepsize} \Big\vert_{\stepsize = 0} \geq 0 }, then \mmode{\stepsize\opt = 0}

    \item If \mmode{\etagrad{+ \direction}} exists, and \mmode{  \inprod{\etagrad{+ \direction}}{\direction} = \frac{\partial \eta_{\stepsize}}{\partial \stepsize} \Big\vert_{\stepsize = 1} \leq 0 }, then \mmode{\stepsize\opt = 1}
\end{enumerate}
Substituting for \mmode{\etagrad{+ \stepsize \direction}} from \eqref{eq:eta-derivatives} and observing that 
\[
\sgn \inprod{\etagrad{+\stepsize \direction}}{\direction} = - \sgn \inprod{ \measurement - \eta_{\stepsize} \linmap (\minvar + \stepsize \direction) }{\linmap (\direction)},
\]
the first order optimality conditions are equivalently written as
\begin{enumerate}[leftmargin = *]
    \item \mmode{\inprod{\measurement}{\linmap (\direction)} \leq \etah{} \inprod{\phih{}}{\linmap (\direction)} }, implies that \mmode{ \inprod{\etagrad{}}{\direction} \geq 0 }, and thus \mmode{\stepsize\opt = 0}.

    \item \mmode{ \inprod{\measurement}{\linmap (\direction)} \geq \eta (\minvar + \direction) \inprod{\linmap (\minvar + \direction)}{\linmap (\direction)} } implies that \mmode{ \inprod{\etagrad{+ \direction}}{\direction} \leq 0}, and thus \mmode{\stepsize\opt = 1}.
\end{enumerate}

If both the above conditions fail, then we know that there exists \mmode{\stepsize\opt \in (0,1)} such that \mmode{\frac{\partial \eta_{\stepsize}}{\partial \stepsize} \Big\vert_{\stepsize = 0} = 0}. Equivalently, we have \mmode{0 = \inprod{\etagrad{+ \stepsize\opt \direction}}{\direction} = \inprod{\measurement - \eta_{\stepsize\opt} \linmap(\minvar + \stepsize\opt \direction) }{\linmap (\direction)} }. 
Substituting for \mmode{ \eta_{\stepsize\opt}} from \eqref{eq:etah-definition} and simplifying gives the following equation in \mmode{\stepsize\opt}
\[
\frac{ \inprod{\linmap (\minvar + \stepsize\opt \direction)}{\linmap (\direction)} }{ \inprod{\measurement}{\linmap (\direction)} } = \frac{1}{\eta_{\stepsize\opt}} = \frac{ \inprod{\measurement}{\linmap (\minvar + \stepsize\opt \direction)} + \norm{\linmap (\minvar + \stepsize\opt \direction)} \sqrt{ e (\minvar + \stepsize\opt \direction) - \eps^2} }{ \xminuseps } .
\]
Rearranging terms, and substituting for \mmode{e (\minvar + \stepsize\opt \direction)} from \eqref{eq:e-func-def}, results in the folllowing equation
\begin{equation}
\label{eq:gamma-opt-before-quadratic-eqn}
\begin{aligned}
\frac{ \inprod{\linmap (\minvar + \stepsize\opt \direction)}{\linmap (\direction)} }{ \inprod{\measurement}{\linmap (\direction)} } 
&- \frac{ \inprod{\measurement}{\linmap (\minvar + \stepsize\opt \direction)} }{\xminuseps} \\
& \quad = \frac{ \sqrt{ \inprod{\measurement}{\linmap (\minvar + \stepsize\opt \direction)}^2 - \norm{\linmap (\minvar + \stepsize\opt \direction)}^2 \xminuseps  } }{\xminuseps} .
\end{aligned}
\end{equation}
Finally, on squaring both sides of \eqref{eq:gamma-opt-before-quadratic-eqn}, we obtain the equation \mmode{a {\stepsize\opt}^2 + 2 b \stepsize\opt + c = 0}, for values of \mmode{a, b, c} given in \eqref{eq:rsu-definitions}. 

If \mmode{a = 0}, \mmode{\stepsize\opt = \nicefrac{-c}{2b}} is the only solution. Whereas, if \mmode{a \neq 0}, it must be observed that out of the two roots of the quadratic equation \mmode{a {\stepsize\opt}^2 + 2 b \stepsize\opt + c = 0}, one satisfies \eqref{eq:gamma-opt-before-quadratic-eqn} and the other satisfies
\[
\begin{aligned}
    \frac{ \inprod{\linmap (\minvar + \stepsize \direction)}{\linmap (\direction)} }{ \inprod{\measurement}{\linmap (\direction)} } 
&- \frac{ \inprod{\measurement}{\linmap (\minvar + \stepsize \direction)} }{\xminuseps} \\
& \quad = - \frac{ \sqrt{ \inprod{\measurement}{\linmap (\minvar + \stepsize \direction)}^2 - \norm{\linmap (\minvar + \stepsize\direction)}^2 \xminuseps  } }{\xminuseps} .
\end{aligned}
\]
Thus, the correct root of the quadratic equation can be picked by ensuring the criterion
\[
0 \ \leq \ \frac{ \inprod{\linmap (\minvar + \stepsize \direction)}{\linmap (\direction)} }{ \inprod{\measurement}{\linmap (\direction)} } - \frac{ \inprod{\measurement}{\linmap (\minvar + \stepsize \direction)} }{\xminuseps} .  
\]
The proof of the proposition is now complete.
\end{proof}

\vskip 0.2in
\bibliography{22-1380}

\end{document}